\documentclass[reqno,12pt]{amsart}

\usepackage{amsfonts,amsthm,amsmath,amssymb,amscd,mathrsfs,graphicx,xcolor,subfig,tikz}
\usetikzlibrary{patterns}
\usepackage{hyperref}

\usepackage[capitalize]{cleveref}
\usepackage{graphics}
\usepackage{indentfirst}
\usepackage{cite}
\usepackage{bm, enumerate,xcolor}
\usepackage[dvips]{epsfig}
\usepackage{latexsym}
\usepackage{float}
\numberwithin{equation}{section}
\numberwithin{figure}{section}
\newtheorem{theorem}{Theorem}[section]
\newtheorem{remark}{Remark}[section]
\newtheorem{definition}{Definition}[section]
\newtheorem{lemma}[theorem]{Lemma}
\newtheorem{corollary}[theorem]{Corollary}
\newtheorem{pro}[theorem]{Proposition}

\def\be{\begin{eqnarray}}
	\def\ee{\end{eqnarray}}
\def\ba{\begin{aligned}}
	\def\ea{\end{aligned}}
\def\bay{\begin{array}}
	\def\eay{\end{array}}
\def\bca{\begin{cases}}
	\def\eca{\end{cases}}
\def\p{\partial}

\def\lab{\label}

\def\bt{\begin{theorem}}
	\def\et{\end{theorem}}
\def\bc{\begin{corollary}}
	\def\ec{\end{corollary}}
\def\bl{\begin{lemma}}
	\def\el{\end{lemma}}
\def\bp{\begin{proposition}}
	\def\ep{\end{proposition}}
\def\br{\begin{remark}}
	\def\er{\end{remark}}
\def\bd{\begin{definition}}
	\def\ed{\end{definition}}
\def\bpf{\begin{proof}}
	\def\epf{\end{proof}}
\def\bex{\begin{example}}
	\def\eex{\end{example}}
\def\bq{\begin{question}}
	\def\eq{\end{question}}
\def\bas{\begin{assumption}}
	\def\eas{\end{assumption}}
\def\ber{\begin{exercise}}
	\def\eer{\end{exercise}}

\def\bu{\mathbf{u}}

\def\be{\mathbf{e}}

\def\mfc{\mathfrak{c}}

\def\mfu{\mathfrak{u}}

\def\mfg{\mathfrak{g}}

\def\XXint#1#2#3{{\setbox0=\hbox{$#1{#2#3}{\int}$ }
		\vcenter{\hbox{$#2#3$ }}\kern-.6\wd0}}

\begin{document}
	
	\title[Steady Euler flows with stagnation points]{Analysis on the steady Euler flows with stagnation points in an infinitely long nozzle}
	
	\author{Congming Li}
	\address{School of Mathematical Sciences, CMA-Shanghai, Shanghai Jiao Tong University, Shanghai 200240, China.}
	\email{congming.li@sjtu.edu.cn}
	\author{Yingshu L\"{u}}
	\address{Institute of Natural Sciences, CMA-Shanghai, Shanghai Jiao Tong University, Shanghai 200240, China.}
	\email{yingshulv@sjtu.edu.cn}

 \author{Henrik Shahgholian}
 \address{Department of Mathematics, Royal Institute of Technology, Stockholm, Sweden}
 \email{henriksh@kth.se}
 
	\author{Chunjing Xie}
	\address{School of Mathematical Sciences, Institute of Natural Sciences, Ministry of Education Key Laboratory of Scientific and Engineering Computing, CMA-Shanghai, Shanghai Jiao Tong University, Shanghai 200240, China.}
	\email{cjxie@sjtu.edu.cn}
	
	
	\date{}

	\begin{abstract}
		A recent prominent result asserts that steady incompressible Euler flows strictly away from stagnation in a two-dimensional infinitely long strip must be shear flows. On the other hand, flows with stagnation points, very challenging in analysis, are interesting and important phenomenon in fluids. In this paper, we not only prove the uniqueness and existence of steady flows with stagnation points, but also obtain the regularity of the boundary of stagnation set, which is a class of obstacle type free boundary.
		
		First, we prove a global uniqueness theorem for steady Euler system with Poiseuille flows as upstream far field state in an infinitely long strip. Due to the appearance of stagnation points, the nonlinearity of the semilinear equation for the stream function becomes non-Lipschitz. This creates a challenging analysis problem since many classical analysis methods do not apply directly. Second, the existence of steady incompressible Euler flows, tending to Poiseuille flows in the upstream, are established in an infinitely long nozzle via variational approach. A very interesting phenomenon  is the  regularity of  the boundary of non-stagnant region, which  can be regarded as an obstacle type free boundary and is proved to be globally $C^1$.  
	Finally, the existence of stagnation region is proved as long as the nozzle is wider than the width of the nozzle at upstream where the flows tend to Poiseuille flows.

	\end{abstract}

	\maketitle

 \tableofcontents

	\section{Introduction and main results}

 \subsection{Background}
	An interesting problem in mathematical theory of fluid dynamics is the well-posedness for the flows in an infinitely long nozzle. Lots of efforts have been made on this study, see \cite{alber92, HN2,Bers,CDX,CX1,DXX,DXY,DD,DL, LS, Amick1, AmickF,XX1,XX2,XX3,CHWX,tx09} and references therein for the study on flows in various settings. These studies rely on the analysis for the solutions  in a nozzle domain, which is also an interesting topic in analysis of PDEs.
	We will first present the main problem and results of this paper. The related results in the literature will be discussed in detail later.

	Consider the steady incompressible Euler system
	\begin{equation}\lab{ieuler}
		\left\{
		\begin{aligned}
			&({\bf u}\cdot\nabla){\bf u} + \nabla p =0,\\
			&\text{div }{\bf u}=0
		\end{aligned}
		\right.
	\end{equation}
	in a two-dimensional infinitely long nozzle $\Omega$, where
	${\bf u}=(u_1, u_2)$ and  $p$ are the velocity and pressure of the flows, respectively.
	The boundary condition for the Euler flows is the so called slip boundary condition, i.e.,
	\begin{equation}\label{slip1}
		{\bf u}\cdot {\bf n}=0\quad  \text{on} \,\, \partial{\Omega}.
	\end{equation}

	In the special case that the domain $\Omega$ is the flat strip $\hat\Omega=\mathbb{R}\times(0,1)$, it was proved in \cite{HN} that if the steady incompressible Euler flows stay strictly away from stagnation and satisfy the slip condition \eqref{slip1} in a two-dimensional infinitely long strip, then the flows must be shear flows, i.e., the velocity field is of the form ${\bf u}=(u_1(x_2), 0)$. More precisely,
	\begin{theorem}(\hspace{1sp}\cite[Theorem 1.1]{HN})\label{thm0}
		Suppose that the flow $\bf u$ solving \eqref{ieuler} is defined in the closed strip $\overline {\hat \Omega}$ with $u_2=0$ on $\partial \hat{\Omega}$,  and $\inf_{\hat\Omega}|{\bf u}|>0$. Then $\bf u$ is a shear flow, that is,
		\begin{eqnarray*}
			{\bf u}(x_1,x_2)=(u_1(x_2),0) \ \ \text{in} \ \ \overline {\hat \Omega}.
		\end{eqnarray*}
	\end{theorem}

	This remarkable result in \cite{HN} did not study the flows with stagnation points.
However, many interesting physical shear flows such as the Couette flows ($u_1(x_2)=x_2$) and Poiseuille flows ($u_1(x_2)=x_2(1-x_2)$) contain the stagnation points where the flow speed is zero  (\hspace{1sp}\cite{LL}). In fact, the Poiseuille flow is not only the solution of the Euler system in a strip, but also the solution of the Navier-Stokes system in a strip with no-slip boundary condition. The Poiseuille flows play a significant role in understanding the hydrodynamic stability of shear flows in the fluid dynamics, see \cite{DR}.
Furthermore, It should be noted that the solutions studied in \cite{HN} should be  classical one, i.e., ${\bf u}\in C^{1}(\bar{\Omega})$. However, when the stagnation points appear, usually the velocity field can only belong to $C^{0,1}(\bar{\Omega})$ (cf. Theorem \ref{thm2}) so that one has to take care of particular regularity issue of the problem.
	
	An interesting issue is when one can completely determine the Euler flows in a strip via boundary conditions.
	
	One of the important observations is that two-dimensional steady incompressible Euler system \eqref{ieuler} is in fact a hyperbolic-elliptic coupled system. Therefore, in order to completely determine the Euler flows in a nozzle, one has to prescribe a boundary condition for the hyperbolic mode.
	On the other hand, if the asymptotic behavior of the flows at the upstream is compatible with the steady Euler system together with the slip condition \eqref{slip1}, then the flows should tend to shear flows, i.e.,
	${\bf u}(x_1,x_2)\to(\hat{u}_1(x_2),0)$ {as} $x_1\to -\infty$
	for some $\hat{u}_1$.
	Note that any solution ${\bf u}$ of the Euler system in a strip $\hat\Omega$ supplemented with the slip condition \eqref{slip1} must satisfy
	\begin{equation}
		\int_0^1 u_1(x_1, x_2)dx_2 =Q
	\end{equation}
	for some constant $Q$ which is called the flux of the flow. It is easy to see that a shear flow is always a solution of the Euler system \eqref{ieuler}.  A particular shear flow, the Poiseuille flow in $\hat{\Omega}$ associated with the flux $Q$, is  of the form
	\begin{equation}\label{defPflow}
		\bu =\bar{\bu}(x_2):= (6Qx_2(1-x_2), 0),
	\end{equation}
which will be considered in this paper.

 \subsection{Main Results}
	In this paper, we prove not only a Liouville type theorem for Poiseuille flows of the incompressible steady Euler system in an infinitely long strip, but also the uniqueness of solutions in a general nozzle as long as the horizontal velocity is positive inside the nozzle. Furthermore, we investigate the existence  of steady incompressible Euler flows in an infinitely long nozzle with Poiseuille flows in the upstream via variational approach, and also characterize the boundary of stagnation region.

	We consider the steady incompressible Euler system in the following domain $\Omega$:
	\begin{equation}\label{defOmega}
		\Omega=\{(x_1,x_2): h_0(x_1)<x_2<h_1(x_1), -\infty<x_1<+\infty\},
	\end{equation}
	where $h_i\in C^{2, \alpha}(\mathbb{R})$ ($i=0$, $1$) for some $\alpha\in (0,1)$ satisfying $h_1(x_1)>h_0(x_1)$ for $x_1 \in \mathbb{R}$, and there exist a positive constant $\bar{L}$ and a negative constant $\underline{L}$ such that
	\begin{equation}\label{flatcondition}
		(h_0(x_1), h_1(x_1))=(0,1) \quad \text{for}\,\, x_1\leq \underline{L}\quad \text{and}\quad  (h_0(x_1), h_1(x_1))=(a,b) \,\,\ \text{for}\,\, x_1\geq \bar{L},
	\end{equation}
 	where $0\leq a<b\leq 1$ are two constants.

Our first main result can be stated as follows. 
	\begin{theorem}\label{thmunique}
		Let {${\bf u}\in C^2(\Omega)\cap C^{0,1}({\bar{\Omega}})$} be a solution of the steady incompressible Euler system \eqref{ieuler} in a general nozzle ${\Omega}$ supplemented with the slip condition \eqref{slip1}. If ${\bf u}$ satisfies ${u_1>0 \ \text{in}\,\, {\Omega}}$ and the asymptotic behavior \begin{equation}\label{farbehavior}
			u_1(x_1,x_2)\to \bar u_1(x_2)=6Qx_2(1-x_2) \quad \text{as}\,\, x_1\to -\infty \,\, \text{uniformly for}\,\, x_2\in [0,1],
		\end{equation}
  then there is at most one solution ${\bf u}$.
	\end{theorem}

 The proof of Theorem \ref{thmunique} relies on the following Liouville type theorem for steady Euler flows in a strip, which can be regarded as a special case of Theorem \ref{thmunique}.

	\begin{pro}	\label{propthm1}
 A particular case of Theorem \ref{thmunique} gives the Liouville type theorem for Poiseuille flows in a strip. More precisely, let $\hat{\Omega}=\mathbb{R}\times (0,1)$,
		given any solution {${\bf u}\in C^2(\hat\Omega)\cap C^{0,1}(\overline{\hat\Omega})$} of the steady incompressible Euler system \eqref{ieuler} in $\hat\Omega$ supplemented with the slip condition \eqref{slip1}, if ${\bf u}$ satisfies
		\begin{equation}\label{nostagnation}
			u_1>0\quad \text{in}\,\, \hat\Omega
		\end{equation}
		and \eqref{farbehavior},
		then ${\bf u}=\bar{\bf u}$ with $\bar{\bf u}= (6Qx_2(1-x_2), 0)$ defined in  \eqref{defPflow}.  \end{pro}

There are several remarks in order.

	\begin{remark}
		As mentioned before,
		steady incompressible Euler flows in $\hat\Omega$ strictly away from stagnation must be shear flows (\hspace{1sp}\cite{HN}). Theorem \ref{thmunique} deals with the steady incompressible Euler flows with stagnation, and in this situation, we overcome the challenge that the nonlinear term of semilinear equation for the stream function is not Lipschitz continuous.
	\end{remark}

	\begin{remark}
		One can use the method developed in this paper to prove similar Liouville type theorem for a more general class of flows, for example, $\bar{u}_1(x_2)=x_2^2$ for which the associated elliptic equation has stronger singularity,  the similar Liouville type theorem can also be proved (\hspace{1sp}\cite{GX}).
	\end{remark}

	In the following, we derive an  existence theory for  solutions of  the steady incompressible Euler system in a general two-dimensional nozzle.

	\begin{theorem}\label{thm2}
		For given $\Omega$, there exists a solution ${\bf u}\in C^{0,1}(\bar{\Omega})$	of the steady incompressible Euler system  \eqref{ieuler} supplemented with the slip condition \eqref{slip1}
		and far field condition \begin{equation}\label{farbehavior1}
		u_1(x_1, x_2)\rightarrow \bar u_1(x_2)=6Qx_2(1-x_2) \quad \text{as}\,\, x_1\to -\infty \quad  \text{for} \ x_2\in (0,1).
   	\end{equation}
	Furthermore, the following properties hold.
  \begin{enumerate}
      \item The solutions satisfy the following asymptotic behavior:
  \begin{equation}\label{basicppt}
			{\bf u}\to \bar{\bf u} \quad \text{as}\,\, x_1\to-\infty\quad \text{and}\,\, \ \ {\bf u}\to \bar{\bf u}_{b-a}=(\bar{u}_{1,b-a}(x_2-a), 0) \quad \text{as}\,\, x_1\to \infty,
		\end{equation}
		where $\bar{u}_{1,b-a}$ is uniquely determined by $\bar{u}_1$ and $b-a$, and is given in Lemma \ref{lem3.4}.
\item There exist $\tilde{h}_0$, $\tilde{h}_1\in C^1(\mathbb{R})$ satisfying $h_0(x_1)\leq \tilde{h}_0(x_1)< \tilde{h}_1(x_1)\leq h_1(x_1)$ for all $x_1\in \mathbb{R}$ (\text{cf. Figure \ref{0-1}}) such that
\begin{equation}\label{positiveu1}
			u_1(x_1, x_2)>0\,\, \text{for}\,\, (x_1,x_2)\in \tilde{\Omega}:=
   \{(x_1,x_2): \tilde{h}_0(x_1)<x_2<\tilde{h}_1(x_1), x_1\in \mathbb{R}\} 
		\end{equation}
and
\begin{equation}
{\bf u}\equiv 0 \quad \text{in}\,\,\Omega\setminus \overline{\tilde{\Omega}}.
\end{equation}

\item  If $\tilde{h}_i(x_1)\neq h_i(x_1)$ for some $x_1\in \mathbb{R}$ and $i=0$ or $1$ so that the stagnation region is not an empty set,  then there exist $\delta>0$ and $\beta \in (0,1)$ such that $\tilde{h}_i\in C^{1, \beta}(x_1-\delta, x_1+\delta)$.

\item If $h_1(x_1)=1-h_0(x_1)$ (i.e., the nozzle $\Omega$ is symmetric with respect to $x_2=1/2$, cf. Figure \ref{0-2}(a)) or $h_1(x_1)\equiv 1$ for $x_1\in \mathbb{R}$ (cf. Figure \ref{0-2}(b)), then $\tilde{h}_0 \geq 0$. Consequently, there must be non-emtpy  stagnation region provided that ${h}_0(x_1)< 0$ for some $x_1\in \mathbb{R}$.
	\end{enumerate}	
 
	\end{theorem}
 \begin{figure}[h]
     \centering

 
\tikzset{
pattern size/.store in=\mcSize, 
pattern size = 5pt,
pattern thickness/.store in=\mcThickness, 
pattern thickness = 0.3pt,
pattern radius/.store in=\mcRadius, 
pattern radius = 1pt}
\makeatletter
\pgfutil@ifundefined{pgf@pattern@name@_uhhzgtmf4}{
\pgfdeclarepatternformonly[\mcThickness,\mcSize]{_uhhzgtmf4}
{\pgfqpoint{0pt}{0pt}}
{\pgfpoint{\mcSize+\mcThickness}{\mcSize+\mcThickness}}
{\pgfpoint{\mcSize}{\mcSize}}
{
\pgfsetcolor{\tikz@pattern@color}
\pgfsetlinewidth{\mcThickness}
\pgfpathmoveto{\pgfqpoint{0pt}{0pt}}
\pgfpathlineto{\pgfpoint{\mcSize+\mcThickness}{\mcSize+\mcThickness}}
\pgfusepath{stroke}
}}
\makeatother

 
\tikzset{
pattern size/.store in=\mcSize, 
pattern size = 5pt,
pattern thickness/.store in=\mcThickness, 
pattern thickness = 0.3pt,
pattern radius/.store in=\mcRadius, 
pattern radius = 1pt}
\makeatletter
\pgfutil@ifundefined{pgf@pattern@name@_lrepviptv}{
\pgfdeclarepatternformonly[\mcThickness,\mcSize]{_lrepviptv}
{\pgfqpoint{0pt}{0pt}}
{\pgfpoint{\mcSize+\mcThickness}{\mcSize+\mcThickness}}
{\pgfpoint{\mcSize}{\mcSize}}
{
\pgfsetcolor{\tikz@pattern@color}
\pgfsetlinewidth{\mcThickness}
\pgfpathmoveto{\pgfqpoint{0pt}{0pt}}
\pgfpathlineto{\pgfpoint{\mcSize+\mcThickness}{\mcSize+\mcThickness}}
\pgfusepath{stroke}
}}
\makeatother

 
\tikzset{
pattern size/.store in=\mcSize, 
pattern size = 5pt,
pattern thickness/.store in=\mcThickness, 
pattern thickness = 0.3pt,
pattern radius/.store in=\mcRadius, 
pattern radius = 1pt}
\makeatletter
\pgfutil@ifundefined{pgf@pattern@name@_q88t825dt}{
\pgfdeclarepatternformonly[\mcThickness,\mcSize]{_q88t825dt}
{\pgfqpoint{0pt}{0pt}}
{\pgfpoint{\mcSize+\mcThickness}{\mcSize+\mcThickness}}
{\pgfpoint{\mcSize}{\mcSize}}
{
\pgfsetcolor{\tikz@pattern@color}
\pgfsetlinewidth{\mcThickness}
\pgfpathmoveto{\pgfqpoint{0pt}{0pt}}
\pgfpathlineto{\pgfpoint{\mcSize+\mcThickness}{\mcSize+\mcThickness}}
\pgfusepath{stroke}
}}
\makeatother

 
\tikzset{
pattern size/.store in=\mcSize, 
pattern size = 5pt,
pattern thickness/.store in=\mcThickness, 
pattern thickness = 0.3pt,
pattern radius/.store in=\mcRadius, 
pattern radius = 1pt}
\makeatletter
\pgfutil@ifundefined{pgf@pattern@name@_19r642nby}{
\pgfdeclarepatternformonly[\mcThickness,\mcSize]{_19r642nby}
{\pgfqpoint{0pt}{0pt}}
{\pgfpoint{\mcSize+\mcThickness}{\mcSize+\mcThickness}}
{\pgfpoint{\mcSize}{\mcSize}}
{
\pgfsetcolor{\tikz@pattern@color}
\pgfsetlinewidth{\mcThickness}
\pgfpathmoveto{\pgfqpoint{0pt}{0pt}}
\pgfpathlineto{\pgfpoint{\mcSize+\mcThickness}{\mcSize+\mcThickness}}
\pgfusepath{stroke}
}}
\makeatother
\tikzset{every picture/.style={line width=0.75pt}} 

\begin{tikzpicture}[x=0.75pt,y=0.75pt,yscale=-1,xscale=1]

\draw  [dash pattern={on 4.5pt off 4.5pt}]  (81.33,91) -- (582.33,92) ;
\draw  [dash pattern={on 4.5pt off 4.5pt}]  (81.33,200) -- (582.33,201) ;
\draw [line width=1.5]    (150.33,91) .. controls (174.67,91) and (204.33,87) .. (225.64,68.87) .. controls (246.96,50.75) and (258.33,66) .. (296.33,100) .. controls (334.33,134) and (363.33,74) .. (393.33,68) .. controls (423.33,62) and (431.33,85) .. (457.33,110) .. controls (483.33,135) and (499.33,128) .. (525.33,130) ;
\draw [line width=1.5]    (154,200) .. controls (215.33,200) and (216.33,242) .. (249.33,244) .. controls (282.33,246) and (323.33,157) .. (379.33,179) .. controls (435.33,201) and (449.33,178) .. (530.33,180) ;
\draw [line width=1.5]    (81.33,91) -- (150.33,91) ;
\draw [line width=1.5]    (525.33,130) -- (579.33,130) ;
\draw [line width=1.5]    (81.33,200) -- (154,200) ;
\draw [line width=1.5]    (530.33,180) -- (579.33,180) ;
\draw [color={rgb, 255:red, 208; green, 2; blue, 27 }  ,draw opacity=1 ][line width=1.5]    (150.33,91) .. controls (175.33,92.8) and (200.33,86) .. (206.33,99) .. controls (212.33,112) and (222.31,86.51) .. (232.33,83) .. controls (242.36,79.49) and (248.32,106.09) .. (260.83,89.44) .. controls (273.33,72.8) and (291.33,97.55) .. (296.33,100) .. controls (301.33,102.45) and (303.33,111.45) .. (325.33,110) .. controls (347.33,108.55) and (359.73,78.8) .. (370.33,100) .. controls (380.93,121.2) and (384.93,93.05) .. (398.33,87) .. controls (411.73,80.95) and (411.33,112) .. (422.33,108) .. controls (433.33,104) and (432.4,86) .. (457.33,110) .. controls (482.27,134) and (495.87,128.4) .. (525.33,130) ;
\draw [color={rgb, 255:red, 208; green, 2; blue, 27 }  ,draw opacity=1 ][line width=1.5]    (81.33,91) -- (150.33,91) ;
\draw [color={rgb, 255:red, 208; green, 2; blue, 27 }  ,draw opacity=1 ][line width=1.5]    (525.33,130) -- (579.33,130) ;
\draw [color={rgb, 255:red, 208; green, 2; blue, 27 }  ,draw opacity=1 ][line width=1.5]    (154,200) .. controls (210.33,202) and (204.33,156) .. (234.33,184) .. controls (264.33,212) and (274.04,146.86) .. (290.19,182.43) .. controls (306.33,218) and (336.33,158) .. (379.33,179) .. controls (422.33,200) and (462.33,178) .. (530.33,180) ;
\draw [color={rgb, 255:red, 208; green, 2; blue, 27 }  ,draw opacity=1 ][line width=1.5]    (81.33,200) -- (154,200) ;
\draw [color={rgb, 255:red, 208; green, 2; blue, 27 }  ,draw opacity=1 ][line width=1.5]    (530.33,180) -- (579.33,180) ;
\draw [draw opacity=0][pattern=_uhhzgtmf4,pattern size=14.475000000000001pt,pattern thickness=0.75pt,pattern radius=0pt, pattern color={rgb, 255:red, 7; green, 62; blue, 192}]   (183,87) .. controls (196.13,106.8) and (226.93,47.6) .. (256.53,68.4) .. controls (286.13,89.2) and (262.33,87) .. (254.33,90) .. controls (246.33,93) and (240.33,81) .. (232.33,83) .. controls (224.33,85) and (216.42,105.52) .. (210.33,105.56) .. controls (204.24,105.59) and (223,95.11) .. (190.67,89.56) ;
\draw [draw opacity=0][pattern=_lrepviptv,pattern size=15.149999999999999pt,pattern thickness=0.75pt,pattern radius=0pt, pattern color={rgb, 255:red, 7; green, 62; blue, 192}]   (361.33,90.67) .. controls (378.33,96.56) and (373.68,103.46) .. (379.67,102.89) .. controls (385.65,102.32) and (395.67,79.56) .. (403.33,85.89) .. controls (411,92.22) and (414.67,96.22) .. (417,103.89) .. controls (419.33,111.56) and (438,101.56) .. (439,96.22) .. controls (440,90.89) and (419.33,68.89) .. (401.67,68.22) .. controls (384,67.56) and (367.33,90.67) .. (361.33,90.67) -- cycle ;
\draw [draw opacity=0][pattern=_q88t825dt,pattern size=15.524999999999999pt,pattern thickness=0.75pt,pattern radius=0pt, pattern color={rgb, 255:red, 7; green, 62; blue, 192}]   (167,201) .. controls (178.53,208.8) and (216.13,167.2) .. (226.93,181.2) .. controls (237.73,195.2) and (240.93,190) .. (251.73,192.8) .. controls (262.53,195.6) and (276.35,169.86) .. (279.73,172) .. controls (283.12,174.14) and (285.88,179.82) .. (290.53,185.78) .. controls (295.19,191.73) and (302.61,200.45) .. (300.33,205) .. controls (298.06,209.55) and (292.53,234) .. (255.33,240) .. controls (218.13,246) and (208.07,190.2) .. (173.33,203.6) ;
\draw  [pattern=_19r642nby,pattern size=8.399999999999999pt,pattern thickness=0.75pt,pattern radius=0pt, pattern color={rgb, 255:red, 7; green, 62; blue, 192}] (434,239) -- (462.33,239) -- (462.33,252) -- (434,252) -- cycle ;

\draw (90.33,205.4) node [anchor=north west][inner sep=0.75pt]  [font=\footnotesize]  {$x_{2} =0$};
\draw (90.33,72.4) node [anchor=north west][inner sep=0.75pt]  [font=\footnotesize]  {$x_{2} =1$};
\draw (537.33,162.4) node [anchor=north west][inner sep=0.75pt]  [font=\footnotesize]  {$x_{2} =a$};
\draw (537.33,136.4) node [anchor=north west][inner sep=0.75pt]  [font=\footnotesize]  {$x_{2} =b$};
\draw (283,232.4) node [anchor=north west][inner sep=0.75pt]    {$\Gamma _{0}$};
\draw (265,46.4) node [anchor=north west][inner sep=0.75pt]    {$\Gamma _{1}$};
\draw (165,145.4) node [anchor=north west][inner sep=0.75pt]  [font=\small,color={rgb, 255:red, 208; green, 2; blue, 27 }  ,opacity=1 ]  {$x_{2} =\tilde{h}_{0}( x_{1})$};
\draw (330,106.4) node [anchor=north west][inner sep=0.75pt]  [font=\small,color={rgb, 255:red, 208; green, 2; blue, 27 }  ,opacity=1 ]  {$x_{2} =\tilde{h}_{1}( x_{1})$};
\draw (472.33,239) node [anchor=north west][inner sep=0.75pt]  [font=\small] [align=left] {Stagnation region};

\end{tikzpicture}

     	\caption{The free boundaries}\label{0-1}
 \end{figure}
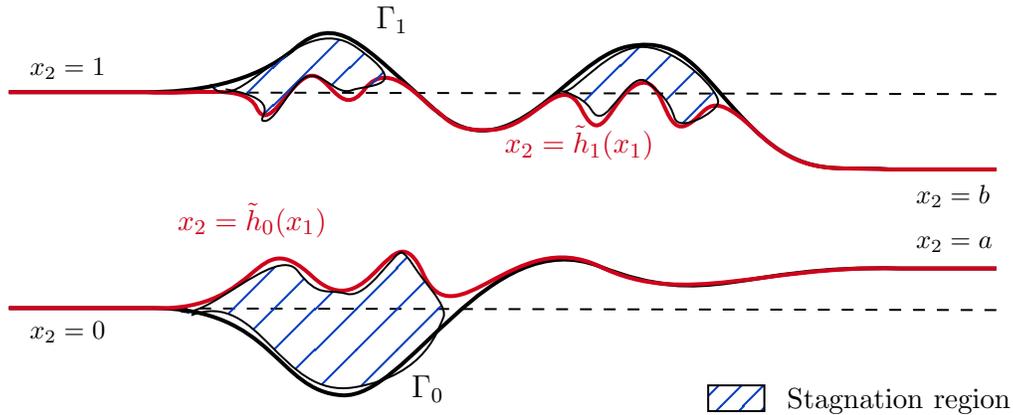
 \hspace{10mm}
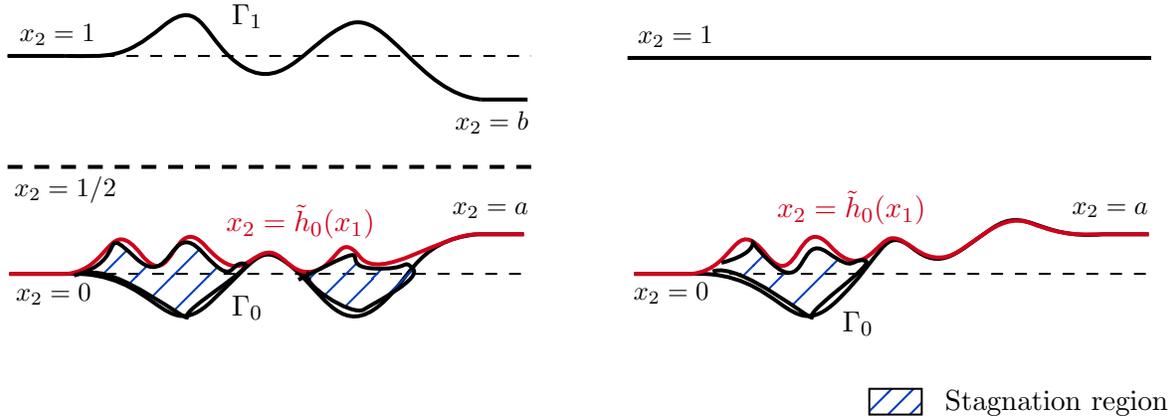
\begin{figure}[h]
				\centering

 
\tikzset{
pattern size/.store in=\mcSize, 
pattern size = 5pt,
pattern thickness/.store in=\mcThickness, 
pattern thickness = 0.3pt,
pattern radius/.store in=\mcRadius, 
pattern radius = 1pt}
\makeatletter
\pgfutil@ifundefined{pgf@pattern@name@_cjgs56owu}{
\pgfdeclarepatternformonly[\mcThickness,\mcSize]{_cjgs56owu}
{\pgfqpoint{0pt}{0pt}}
{\pgfpoint{\mcSize+\mcThickness}{\mcSize+\mcThickness}}
{\pgfpoint{\mcSize}{\mcSize}}
{
\pgfsetcolor{\tikz@pattern@color}
\pgfsetlinewidth{\mcThickness}
\pgfpathmoveto{\pgfqpoint{0pt}{0pt}}
\pgfpathlineto{\pgfpoint{\mcSize+\mcThickness}{\mcSize+\mcThickness}}
\pgfusepath{stroke}
}}
\makeatother

 
\tikzset{
pattern size/.store in=\mcSize, 
pattern size = 5pt,
pattern thickness/.store in=\mcThickness, 
pattern thickness = 0.3pt,
pattern radius/.store in=\mcRadius, 
pattern radius = 1pt}
\makeatletter
\pgfutil@ifundefined{pgf@pattern@name@_vzclbegd2}{
\pgfdeclarepatternformonly[\mcThickness,\mcSize]{_vzclbegd2}
{\pgfqpoint{0pt}{0pt}}
{\pgfpoint{\mcSize+\mcThickness}{\mcSize+\mcThickness}}
{\pgfpoint{\mcSize}{\mcSize}}
{
\pgfsetcolor{\tikz@pattern@color}
\pgfsetlinewidth{\mcThickness}
\pgfpathmoveto{\pgfqpoint{0pt}{0pt}}
\pgfpathlineto{\pgfpoint{\mcSize+\mcThickness}{\mcSize+\mcThickness}}
\pgfusepath{stroke}
}}
\makeatother

 
\tikzset{
pattern size/.store in=\mcSize, 
pattern size = 5pt,
pattern thickness/.store in=\mcThickness, 
pattern thickness = 0.3pt,
pattern radius/.store in=\mcRadius, 
pattern radius = 1pt}
\makeatletter
\pgfutil@ifundefined{pgf@pattern@name@_uygvbw85a}{
\pgfdeclarepatternformonly[\mcThickness,\mcSize]{_uygvbw85a}
{\pgfqpoint{0pt}{0pt}}
{\pgfpoint{\mcSize+\mcThickness}{\mcSize+\mcThickness}}
{\pgfpoint{\mcSize}{\mcSize}}
{
\pgfsetcolor{\tikz@pattern@color}
\pgfsetlinewidth{\mcThickness}
\pgfpathmoveto{\pgfqpoint{0pt}{0pt}}
\pgfpathlineto{\pgfpoint{\mcSize+\mcThickness}{\mcSize+\mcThickness}}
\pgfusepath{stroke}
}}
\makeatother

 
\tikzset{
pattern size/.store in=\mcSize, 
pattern size = 5pt,
pattern thickness/.store in=\mcThickness, 
pattern thickness = 0.3pt,
pattern radius/.store in=\mcRadius, 
pattern radius = 1pt}
\makeatletter
\pgfutil@ifundefined{pgf@pattern@name@_opmnb4xae}{
\pgfdeclarepatternformonly[\mcThickness,\mcSize]{_opmnb4xae}
{\pgfqpoint{0pt}{0pt}}
{\pgfpoint{\mcSize+\mcThickness}{\mcSize+\mcThickness}}
{\pgfpoint{\mcSize}{\mcSize}}
{
\pgfsetcolor{\tikz@pattern@color}
\pgfsetlinewidth{\mcThickness}
\pgfpathmoveto{\pgfqpoint{0pt}{0pt}}
\pgfpathlineto{\pgfpoint{\mcSize+\mcThickness}{\mcSize+\mcThickness}}
\pgfusepath{stroke}
}}
\makeatother
\tikzset{every picture/.style={line width=0.75pt}} 

\begin{tikzpicture}[x=0.7pt,y=0.75pt,yscale=-1,xscale=1]

\draw [line width=0.75]  [dash pattern={on 4.5pt off 4.5pt}]  (20.33,70) -- (303.33,70) ;
\draw [line width=0.75]  [dash pattern={on 4.5pt off 4.5pt}]  (21.33,180) -- (303.33,180) ;
\draw [line width=1.5]    (51.33,70) .. controls (75.67,70) and (82.33,72) .. (104.33,55) .. controls (126.33,38) and (126.33,64) .. (148.33,76) .. controls (170.33,88) and (183.33,62) .. (204.33,54) .. controls (225.33,46) and (242.33,89) .. (276.33,92) ;
\draw [line width=1.5]    (20.33,70) -- (51.33,70) ;
\draw [line width=1.5]    (276.33,92) -- (301.33,92) ;
\draw [line width=1.5]    (21.33,180) -- (52.33,180) ;
\draw  [pattern=_cjgs56owu,pattern size=8.399999999999999pt,pattern thickness=0.75pt,pattern radius=0pt, pattern color={rgb, 255:red, 7; green, 62; blue, 192}] (486,238) -- (514.33,238) -- (514.33,251) -- (486,251) -- cycle ;
\draw [line width=1.5]    (52.33,180) .. controls (101.33,178) and (104.33,217) .. (130.33,194) .. controls (156.33,171) and (157.33,164) .. (173.33,176) .. controls (189.33,188) and (199.33,210) .. (218.33,198) .. controls (237.33,186) and (239.33,163) .. (274.33,160) ;
\draw [line width=1.5]    (274.33,160) -- (299.33,160) ;
\draw [line width=1.5]  [dash pattern={on 5.63pt off 4.5pt}]  (20.33,126) -- (303.33,126) ;
\draw [color={rgb, 255:red, 208; green, 2; blue, 27 }  ,draw opacity=1 ][line width=1.5]    (52.33,180) .. controls (76.6,177.67) and (74.86,149.33) .. (92.6,170.17) .. controls (110.33,191) and (109.33,147.5) .. (127.33,166) .. controls (145.33,184.5) and (149.09,172.4) .. (159.09,169.6) .. controls (169.09,166.8) and (172.36,181.33) .. (185.6,178.67) .. controls (198.83,176) and (199.33,158) .. (210.65,171.29) .. controls (221.96,184.58) and (255.09,159.6) .. (274.33,160) ;
\draw [color={rgb, 255:red, 208; green, 2; blue, 27 }  ,draw opacity=1 ][line width=1.5]    (274.33,160) -- (299.33,160) ;
\draw [color={rgb, 255:red, 208; green, 2; blue, 27 }  ,draw opacity=1 ][line width=1.5]    (21.33,180) -- (33.63,180) -- (52.33,180) ;
\draw [line width=1.5]    (355.33,71) -- (638.33,71) ;
\draw [line width=0.75]  [dash pattern={on 4.5pt off 4.5pt}]  (359,180) -- (641,180) ;
\draw [line width=1.5]    (359,180) -- (390,180) ;
\draw [line width=1.5]    (390,180) .. controls (439,178) and (442,217) .. (468,194) .. controls (494,171) and (493.33,155) .. (509.33,167) .. controls (525.33,179) and (533.33,169) .. (552.33,157) .. controls (571.33,145) and (577,163) .. (612,160) ;
\draw [line width=1.5]    (612,160) -- (637,160) ;
\draw [color={rgb, 255:red, 208; green, 2; blue, 27 }  ,draw opacity=1 ][line width=1.5]    (390,180) .. controls (414.27,177.67) and (412.53,149.33) .. (430.27,170.17) .. controls (448,191) and (447,147.5) .. (465,166) .. controls (483,184.5) and (489.85,162.4) .. (498.33,162) .. controls (506.82,161.6) and (512.6,173.2) .. (526.07,171.6) .. controls (539.54,170) and (554.68,147.6) .. (573.27,154.4) .. controls (591.85,161.2) and (600.47,160) .. (612,160) ;
\draw [color={rgb, 255:red, 208; green, 2; blue, 27 }  ,draw opacity=1 ][line width=1.5]    (612,160) -- (637,160) ;
\draw [color={rgb, 255:red, 208; green, 2; blue, 27 }  ,draw opacity=1 ][line width=1.5]    (359,180) -- (371.3,180) -- (390,180) ;
\draw [draw opacity=0][pattern=_vzclbegd2,pattern size=13.575000000000001pt,pattern thickness=0.75pt,pattern radius=0pt, pattern color={rgb, 255:red, 7; green, 62; blue, 192}][line width=1.5]    (56.29,180.8) .. controls (86.67,173.47) and (68.83,153.67) .. (89.21,171.67) .. controls (109.6,189.67) and (107.71,150.67) .. (125.71,169.17) .. controls (143.71,187.67) and (139.77,173.8) .. (147.89,174.8) .. controls (156,175.8) and (115.21,197.67) .. (117.21,201.67) .. controls (119.21,205.67) and (77.97,174.67) .. (63.09,180) ;
\draw [draw opacity=0][pattern=_uygvbw85a,pattern size=14.774999999999999pt,pattern thickness=0.75pt,pattern radius=0pt, pattern color={rgb, 255:red, 7; green, 62; blue, 192}][line width=1.5]    (406.13,174.4) .. controls (436.51,167.07) and (412.15,156) .. (432.53,174) .. controls (452.92,192) and (447.43,158.73) .. (463.38,169.17) .. controls (479.33,179.6) and (475.22,172.2) .. (483.33,173.2) .. controls (491.45,174.2) and (452.88,197.67) .. (454.88,201.67) .. controls (456.88,205.67) and (416.21,173.47) .. (401.33,178.8) ;
\draw [draw opacity=0][pattern=_opmnb4xae,pattern size=11.850000000000001pt,pattern thickness=0.75pt,pattern radius=0pt, pattern color={rgb, 255:red, 7; green, 62; blue, 192}][line width=1.5]    (177.49,183) .. controls (226.43,159.6) and (196.03,176) .. (217.23,177.2) .. controls (238.43,178.4) and (233.63,172) .. (239.23,177.2) .. controls (244.83,182.4) and (219.22,194.27) .. (212.83,198) .. controls (206.43,201.73) and (176.79,184.03) .. (183.23,179.2) ;

\draw (23.33,183.4) node [anchor=north west][inner sep=0.75pt]  [font=\footnotesize]  {$x_{2} =0$};
\draw (25.33,52.4) node [anchor=north west][inner sep=0.75pt]  [font=\footnotesize]  {$x_{2} =1$};
\draw (259.13,142.8) node [anchor=north west][inner sep=0.75pt]  [font=\footnotesize]  {$x_{2} =a$};
\draw (260.53,96.4) node [anchor=north west][inner sep=0.75pt]  [font=\footnotesize]  {$x_{2} =b$};
\draw (140,189.4) node [anchor=north west][inner sep=0.75pt]  [font=\small]  {$\Gamma _{0}$};
\draw (140,42.4) node [anchor=north west][inner sep=0.75pt]  [font=\small]  {$\Gamma _{1}$};
\draw (137,143.4) node [anchor=north west][inner sep=0.75pt]  [font=\small,color={rgb, 255:red, 208; green, 2; blue, 27 }  ,opacity=1 ]  {$x_{2} =\tilde{h}_{0}( x_{1})$};
\draw (525.33,237) node [anchor=north west][inner sep=0.75pt]  [font=\small] [align=left] {Stagnation region};
\draw (22.33,129.4) node [anchor=north west][inner sep=0.75pt]  [font=\footnotesize]  {$x_{2} =1/2$};
\draw (359.33,53.4) node [anchor=north west][inner sep=0.75pt]  [font=\footnotesize]  {$x_{2} =1$};
\draw (357,183) node [anchor=north west][inner sep=0.75pt]  [font=\footnotesize]  {$x_{2} =0$};
\draw (594,142.4) node [anchor=north west][inner sep=0.75pt]  [font=\footnotesize]  {$x_{2} =a$};
\draw (470,197.4) node [anchor=north west][inner sep=0.75pt]  [font=\small]  {$\Gamma _{0}$};
\draw (434.67,136.4) node [anchor=north west][inner sep=0.75pt]  [font=\small,color={rgb, 255:red, 208; green, 2; blue, 27 }  ,opacity=1 ]  {$x_{2} =\tilde{h}_{0}( x_{1})$};

\end{tikzpicture}

    \caption{Existence of stagnation points}\label{0-2}
			\end{figure}

We have the following remarks on Theorem \ref{thm2}.

	\begin{remark}
		Note that $\bar{u}_1'(0)>0$ and $\bar{u}_1'(1)<0$. This violates the conditions for the asymptotic horizontal velocity of the steady compressible flows studied in \cite{XX3, CHWX} and induces many new difficulties in studying the associated problem for the stream function.
	\end{remark}

	\begin{remark}
		The methods developed in this paper, together with techniques in \cite{CDXX, HN2, GX}, should also help to study the flows in a half-plane or past a wall with stagnation points at far field.
	\end{remark}
	
	\begin{remark}
 We would like to highlight the observation that the flows analyzed in previous studies \cite{DX, CDXX, HN3} exhibit points of stagnation. However, it's worth noting that the far field  states of the flows investigated in \cite{DX, CDXX} remain distinctively distant from stagnation points. To the authors' current understanding, the outcomes derived in this paper are believed to be pioneering outcomes addressing flows characterized by remote states encompassing stagnation points.

	\end{remark}

 \begin{remark}
     The rate of convergence of flows at far field is studied in \cite{LLWX}.
 \end{remark}

  \subsection{Related results}
		Here we give more detailed discussions on the related results for steady flows in nozzles and the major technical points for the analysis.
		The incompressible flows can be regarded as zero Mach number limit for the compressible flows. Bers proposed the problem on the existence of compressible irrotational subsonic flows in an infinitely long two-dimensional nozzle in \cite{Bers}. The behavior of subsonic solutions in nozzles was first studied by Gilbarg in \cite{Gilbarg}. In the last two decades, there are many important progress on the well-posedness theory of subsonic flows in nozzles. When the flows are subsonic, the well-posedness theory for irrotational flows in two and three-dimensional infinitely long nozzles were established in \cite{XX1,XX2,DXY}. The existence of weak solutions for subsonic-sonic flows was established in \cite{XX1,XX2, HWW}, etc.
		
		When the flows have non-zero vorticity, the compressible Euler system in subsonic region is a hyperbolic-elliptic coupled system. After introducing the stream function formulation, the steady compressible Euler system is equivalent to a single second order equation with a source term which represents memory effect for the hyperbolic mode (cf. \cite{XX3}) as long as the velocity field converges at far field. Furthermore, the well-posedness for subsonic flows with small vorticity was first obtained in \cite{XX3}.  This approach was later used to study subsonic flows with non-zero vorticity in nozzles in various settings, such as the flows in periodic nozzles or in axially symmetric nozzles, the non-isentropic flows, and also the limit of subsonic flows, etc, see \cite{CX1, CDXX, CJ, DD, CDX, DL,CHWX, CHW} and references therein.
		
		 \subsection{Key ideas of  proofs}
		
		Now we give the key ideas for the proof of the main results in this paper. The steady incompressible Euler system is reduced into a single second order semilinear elliptic equation for the stream function, and the uniqueness of solutions in an infinite strip  is equivalent to the Liouville type theorem for the semilinear equation in the associated domains. The research in this direction was initiated by Berestycki, Caffarelli, and Nirenberg (cf. \cite{BN,BCN0,BCN} and references therein) when the nonlinearity is Lipschitz continuous.
		These ideas were used in \cite{HN} to study the Liouville type theorem for steady Euler flows in a strip strictly away from stagnation where the associated elliptic equation has Lipschitiz nonlinearity.	
		When the nonlinearity is not Lipschitz continuous, the Liouville type theorem for semilinear elliptic equations was investigated in \cite{Farina1,Farina2} when the nonliearity satisfies some structural conditions. However, the nonlinearity studied in this paper is not Lipschitz continuous and does not satisfy the structural conditions in \cite{Farina1, Farina2}.

		One of the key observations in the proof of Theorem \ref{thmunique} is based on the Liouville type theorem for solutions in a strip (cf. Propositions \ref{propthm1} and \ref{prop2}). In order to prove the Liouville type property of solutions in a strip,  the structure of equation is used to show the positivity of horizontal velocity near the boundary so that the method of moving plane can be initiated.
		Furthermore, some new and interesting properties of one-dimensional solutions, in particular,  the uniqueness and other monotonic properties of the solutions, are obtained although the associated ODE also has non-Lipschitz nonlinearity. This help show that one-dimensional solutions must be Poiseuille flows.

		For the steady incompressible Euler flows in a general infinitely long nozzle, we use the variational structure of the equation for the stream function to prove the existence of solutions in each bounded truncated domains.
		The nonnegativity of the horizontal velocity and the uniqueness of energy minimizers are proved by the shifting method for the solutions and comparison for the associated energies. 
		Finally, the far field behavior of the solutions is proved via the monotonicity of the normalized energy and the uniform estimate for the energy minimizers.
	The key issue for the free boundary is that it is a obstacle type free boundary. The regularity of this free boundary is obtained by combining the analysis for singularity of obstacle free boundary by Caffarelli, et al and the moving plane method. The existence of free boundary is a consequence of the analysis for the energy minimizer of the shear flows.	

   \subsection{Organization of the paper}
		The organization for the   paper is as follows. In Section \ref{secstream}, the stream function formulation is introduced and the equivalence between the incompressible Euler system and stream function formulation is proved.
		In Section \ref{secODE}, we give the detailed analysis for the one-dimensional solutions of the semilinear equation  in the strip, which correspond to the shear flows in an infinitely long strip.
		In Section \ref{secunique}, the Liouville type theorem for solutions of semilinear elliptic equation is established first.  Next, the uniqueness of solutions with positive horizontal velocity is obtained with the aid of Liouville type theorem for the solutions in a strip.
		The existence and various kinds of fine properties for the energy minimizers of the variational problem associated with the stream function are established in Section \ref{secminimizer} via the delicate analysis. In particular, the existence of solutions is obtained and Theorem \ref{thm2} is proved.

\subsection{Notation}

Some notations used in this paper are collected in the following.

\begin{tabular}{ l l  }
 ${\bf u}$ &\qquad  velocity  \medskip  \\
 $\Omega$  &\qquad  The nozzle \quad $\{(x_1,x_2): h_0(x_1)<x_2<h_1(x_1), -\infty<x_1<+\infty\}$,   \medskip \\
 $\Omega_N$ &\qquad $\Omega\cap\{(x_1,x_2): |x_1|<N\}$    \medskip \\
$Q$   &\qquad The flux of the flow\quad  $Q=\int_0^1 u_1(x_1, x_2)dx_2$     \medskip \\
$\tilde{h}_0(x_1)$   &\qquad $\inf\{x_2: |{\bf u}|(x_1,x_2)>0, x_2\in [h_0(x_1), h_1(x_1)] \} $   \medskip \\
$ \tilde{h}_1(x_1)  $  &\qquad $\sup \{x_2: |{\bf u}|(x_1,x_2)>0, x_2\in [h_0(x_1), h_1(x_1)] \}    $  \medskip \\
$\tilde{\Gamma}_{0, N}$ &\qquad $\partial \tilde{\Omega}_N \cap \{x: \psi(x)=0\} $   \medskip \\
$\tilde{\Gamma}_{1, N} $  &\qquad   $\partial \tilde{\Omega}_N \cap \{x: \psi(x)=Q\}$  \medskip \\
$\check{\Gamma}_{i, N}  $ &\qquad    $\tilde{\Gamma}_{i,N}\setminus \Gamma_i $\medskip \\
$\bar{u}_1(x_2)$ &\qquad  $6Qx_2(1-x_2)$\\
$\kappa (\psi) $   &\qquad defined implicitly by  $\psi= \int_0^{\kappa(\psi)} \bar{u}_1(s)ds$ satisfying  
  $\kappa(0)=0$ and $\kappa(Q)=1$.   \medskip \\  
$f(\psi)$&\qquad   $6Q(1-2\kappa(\psi))$  \medskip \\
&\qquad     \medskip \\
\end{tabular}

		\section{Stream function formulation}\label{secstream}
		In this section, we introduce the stream function to reduce the steady incompressible Euler system in the region without stagnation point into a single second order semilinear elliptic equation. One of the crucial points is to determine the exact nonlinearity of the equation.
		
		Let {${\bf u}\in C^2(\Omega)\cap C^{0,1}(\bar\Omega)$} be a strong solution of the Euler system \eqref{ieuler}.
		It follows from the divergence free condition of the velocity field that there exists a stream function $\psi$ satisfying
		\begin{equation}\label{defpsi}
			\partial_{x_2}\psi =u_1\quad\text{and}\quad \partial_{x_1} \psi=-u_2.
		\end{equation}
		Clearly, if ${\bf u}$ satisfies the slip condition \eqref{slip1}, then $\psi$ must be constants on $\Gamma_i$ ($i=0$, $1$), where
		\[
		\Gamma_i=\{(x_1,x_2): x_2=h_i(x_1),\, -\infty<x_1<\infty\}.
		\]
		Without loss of generality, assume that
		\begin{equation}\label{psiGam0}
			\psi=0\quad \text{on}\,\, \Gamma_0.
		\end{equation}
		The straightforward calculations for the incompressible Euler system \eqref{ieuler} yield
		\begin{eqnarray}\label{pro}
			\nabla^\perp\psi\cdot \nabla \omega=	{\bf u}\cdot \nabla  \omega=0,
		\end{eqnarray}
		where $\omega=\partial_{x_2} u_1-\partial_{x_1} u_2$ is called the vorticity of the flows.
		
		First, we have the following elementary lemma.
		\begin{pro}\label{propstream}
			Suppose that ${\bf u}\in C^2(\Omega)\cap C^{0,1}(\bar{\Omega})$ is a  solution of the steady Euler system \eqref{ieuler} supplemented with the slip condition \eqref{slip1}. If ${\bf u}$ satisfies \begin{equation}\label{basicppt1}
				\int_{h_0(x_1)}^{h_1(x_1)} u_1(x_1,x_2)dx_2=Q \quad\text{and}\,\,\  u_1(x_1, x_2)>0\,\, \text{for all}\,\,  (x_1,x_2)\in\Omega,
			\end{equation}
   and
			\begin{equation}\label{con}
				{u}_1(x_1,x_2)\to \bar{u}_1(x_2)=6Qx_2(1-x_2)\,\, \text{as}\,\, x_1\to -\infty
			\end{equation}
			uniformly with respect to $x_2\in[0,1]$,
			then the stream function  $\psi$ satisfies 
			\begin{equation}\label{psibound}
				0 < \psi < Q\quad \text{in}\,\, \Omega
			\end{equation}
 and  the equation
			\begin{equation}\label{streameq}
				\Delta \psi =f(\psi),
			\end{equation}
			where  $f$  is given by
			\begin{equation}\label{formf}
				f(\psi)=6Q(1-2\kappa(\psi))
			\end{equation}
			with $\kappa(\psi)$ satisfying
			\[
			\psi=\int_0^{\kappa(\psi)} \bar{u}_1(s)ds.
			\]
   So that $\kappa(0)=0$ and $\kappa(Q)=1$.
			Furthermore, $f\in C^{1/2}([0, Q])\cap C^\infty((0,Q))$ and this is optimal regularity for $f$ on $[0, Q]$.
		\end{pro}
		\begin{proof}
  The proof is divided into three steps.
			
			\emph{Step 1. Local determination.}	Without loss of generality, assume that $u_1(\bar{x}) \neq 0$. It follows from the implicit function theorem that there exist a $\delta>0$ and a function $\xi$ such that
			\[
			x_2=\xi(x_1, \psi) \ \ \text{in}\ B_\delta(\bar{x}).
			\]
			Thus $\omega=\omega(x_1, \xi(x_1, \psi)):=f(x_1, \psi)$. The straightforward computations give
			\[
			\frac{\partial f}{\partial x_1}=\frac{\partial \omega}{\partial x_1}+\frac{\partial \omega}{\partial x_2}\frac{\partial \xi}{\partial x_1}= \frac{\partial \omega}{\partial x_1}+\frac{\partial \omega}{\partial x_2}(-\frac{\partial \psi}{\partial x_1}/\frac{\partial\psi}{\partial x_2})=\frac{\nabla^\perp \psi \cdot \nabla \omega} {\partial_{x_2}\psi}=0,
			\]
 	where the equation \eqref{pro} has been used.
			Therefore, $f(x_1, \psi)=f(\psi)$. Hence there exists a function $f$ such that $\omega=f(\psi)$ holds in $B_\delta(\bar{x})$.
			
			\emph{Step 2. Global determination.} If $u_1(x)>0$ in $\Omega$, then at each point $x^*\in \Omega$, there exists at most one streamline which is defined as follows,
			\begin{equation}
				\left\{
				\begin{aligned}
					& \frac{dx_1}{ds}=u_1(x_1,x_2),\\
					& \frac{dx_2}{ds}=u_2(x_1,x_2),\\
					&(x_1,x_2)(0)=x^*=(x_1^*, x_2^*).
				\end{aligned}
				\right.
			\end{equation}
   
			Furthermore, the streamline $\Gamma$ passing through $x^*$ cannot intersect with the nozzle boundary.  Indeed, if this fails, we may apply the divergence theorem in the region $D$ bounded by $\Gamma$, $\Gamma_0$, and the line segment $\{(x_1^*, x_2)\in\Omega: h_0(x_1^*)<x_2<h_1(x_1^*)\}$ to arrive at  
			\[
			\int_{h_0(x_1^*)}^{x_2^*} u_1(x_1^*, s)ds=0.
			\]
			This contradicts with \eqref{basicppt1}. Therefore, each streamline can only travel from left to right all the way to the downstream as $x_1\to +\infty$. 
   
   Hence it follows from \eqref{pro} that there exists a function $f$ such that for all $x\in \Omega$, one has $\omega(x)=f(\psi(x))$, and also \eqref{psibound} must hold.

   \emph{Step 3. Explicit form of $f$}.
    For any {${\bf u}\in C^2(\Omega)\cap C^{0,1}(\bar\Omega)$},
			it follows from the divergence free condition that \begin{equation}\label{partial2omega}
				\partial_{x_2}\omega=\partial_{x_2}(\partial_{x_2}u_1-\partial_{x_1}u_2)=\Delta u_1.
			\end{equation}
			For any $k\in \mathbb{N}$, define ${\bf u}^{(k)}(x)={\bf u}(x_1-k, x_2)$. Then \eqref{con} implies that $u_1^{(k)}\to \bar{u}_1$ in $L^2((0,1)\times(0,1))$ as $k \to \infty$. Therefore, one has
			\begin{equation}\label{weak}
				\Delta u_1^{(k)}\to \bar{u}_1'' \quad \text{in}\,\, H^{-2}((0,1)\times(0,1)) \ \ \text{as} \ k\to \infty.
			\end{equation}
			Note that $\omega^{(k)}=\partial_{x_2}u_1^{(k)}-\partial_{x_1}u_2^{(k)}$ satisfies
			\[
			\|\omega^{(k)}\|_{L^\infty((0,1)\times (0,1))}\leq C,
			\]
			where $C$ is a uniform constant independent of $k$. Therefore, there is a subsequence of $\{\omega^{(k)}\}$, which is still labeled by $\{\omega^{(k)}\}$, such that
			$\omega^{(k)} \rightharpoonup \bar \omega$ in $L^2((0,1)\times(0,1))$ for some $\bar\omega\in L^2((0,1)\times (0,1))$ as $k \to \infty$.
			Hence one has
			\begin{equation}\label{weak1}
				\partial_{x_2} \omega^{(k)}\rightharpoonup \partial_{x_2}\bar\omega \ \text{ in} \  H^{-1}((0,1)\times (0,1)) \ \ \text{as} \ k\to \infty.
			\end{equation}
			By the uniqueness of weak limit, combining \eqref{partial2omega}--\eqref{weak1}, one arrives at
			\begin{eqnarray}\label{limit}
				\partial_{x_2}\bar{\omega} =\bar{u}_1''.
			\end{eqnarray}
			Similarly, there exists a function $\bar{u}_2\in L^2((0,1)\times (0,1))$ such that  $u_2^{(k)} \rightharpoonup \bar{u}_2$  in $L^2((0,1)\times(0,1))$. It follows from \eqref{partial2omega} that one has
			\[
			\partial_{x_2}\bar{\omega}=\bar{u}_1''-\partial^2_{x_1x_2}\bar{u}_2.
			\]
			Therefore, from \eqref{limit}, one can obtain that $\partial^2_{x_1x_2}\bar{u}_2\equiv 0$. Hence one has $\bar{u}_2=\bar{\varsigma}_1(x_1)+\bar{\varsigma}_2(x_2)$ for some functions $\bar{\varsigma}_1$ and $\bar{\varsigma}_2$. Note that $\bar{u}_2(x_1, 0)=\bar{u}_2(x_1, 1)=0$. Hence $\bar{\varsigma}_1\equiv C_1$ and $\bar{u}_2=\bar{\varsigma}_2(x_2)+C_1$ for some constant $C_1$. By definition, one has
			\begin{equation}\label{defbaromega}
				\bar{\omega} =\partial_{x_2} \bar{u}_1 -\partial_{x_1}\bar{u}_2=\bar{u}_1'(x_2)=6Q(1-2x_2).
			\end{equation}
			
			Since ${u}_1(x_1,x_2)\to \bar{u}_1(x_2)$ as $x_1\to -\infty$, one has
			\[
			\psi(x_1,x_2)=\int_{h_0(x_1)}^{x_2} u(x_1, s)ds\to \bar{\psi}(x_2) :=\int_0^{x_2} \bar{u}_1(s)ds\quad \text{as}\,\, x_1\to -\infty.
			\]
			The straightforward computations give
			\begin{equation}\label{psi}
				\bar{\psi}(x_2)=Q(3x_2^2-2x_2^3),
			\end{equation}
			which is a strictly increasing function for $x_2\in [0, 1]$, so
			$x_2$ can be written as a function of $\bar{\psi}$. We denote this function by $x_2=\kappa(\bar{\psi})$. Clearly, it follows from \eqref{psi} that $\kappa\in C^{1/2}([0, Q])\cap C^\infty((0,Q))$, $\kappa(0)=0$ and $\kappa$ is an increasing function which maps  $[0,Q]$ to $[0,1]$.
			
			It follows from \eqref{defpsi} and \eqref{pro} that both stream function and vorticity are conserved along each streamline. Hence one has
			\begin{equation}\label{omega}
\omega(x)=\bar\omega(\kappa(\psi(x)))=6Q(1-2\kappa({\psi}(x))),
			\end{equation}
			where the explicit form of $\bar\omega$ in \eqref{defbaromega} has been used.
			Furthermore, by the definition of $\omega$ and $\psi$, one can see that $\omega =\Delta \psi$.
			This, together with \eqref{omega}, implies that \eqref{streameq} holds in $\Omega$ with $f$ defined in \eqref{formf}.
			
			The proof of the proposition is completed.
		\end{proof}

		\section{Analysis on the shear flows}\label{secODE}

  The objective of this section is to examine the one-dimensional solutions of equation \eqref{streameq}, which represent shear flows. The analysis  will contribute to fully understanding the Euler flows within this region. Indeed, for any value of $d \in (0,1]$, we will investigate the existence, uniqueness, and characterizations of solutions to problem 		\begin{equation}\label{pb1.2}
			\left\{
			\begin{aligned}
				&{\psi}''(x_2)=f(\psi(x_2)),\ \ \ x_2\in (0,d),\\
				&{\psi}(0)=0, \ \ {\psi}(d)=Q.
			\end{aligned}
			\right.
		\end{equation}
		In particular when $d=1$ in    problem \eqref{pb1.2} 
		for   one-dimensional solution of Euler system \eqref{ieuler} in $\hat{\Omega}$ is studied in detail. The study of boundary value problem \eqref{pb1.2} closely relates to the following initial value problem
		\begin{equation}\label{pbphic}
			\left\{
			\begin{aligned}
				&{\psi}''(x_2)=f(\psi(x_2)),\ \ \\
				&{\psi}(0)=0, \ \ {\psi}'(0)=c
			\end{aligned}
			\right.
		\end{equation}
		for some $c\geq 0$. Furthermore, the solution of \eqref{pb1.2} is also an energy minimizer of the energy functional
		\begin{equation}\label{1Denergy}
			{\mathcal{J}}_d (\psi) =\int_0^d \left(\frac{|\psi'(s)|^2}{2}+F(\psi(s))\right)ds
		\end{equation}
		over
		\begin{equation}\label{defUd}
			{\mathcal{U}}_d =\{\psi\in H^1(0,d):\psi(0)=0, \psi(d)=Q\},
		\end{equation}
		where $F$ is defined in \eqref{defF}. We can summarize our main results on shear flows and their properties as follows.
		\begin{theorem}\label{thmshear}
			For any $d\in (0,1]$, we have the following results.
			\begin{enumerate}
				\item 	There exists a unique solution $\bar{\varphi}_d\in C^2((0,d))\cap C^1([0,d])$ of \eqref{pb1.2}. Furthermore, $\bar{\varphi}_d'(x_2)>0$ for $x_2\in (0,d)$.
				\item For any $c\geq 0$, the problem \eqref{pbphic} has a unique solution $\phi_c$. Furthermore, for any $d\in (0, 1]$, there is a unique $c=\mathfrak{c}(d)$ such that the solution of \eqref{pb1.2} is exactly the solution of \eqref{pbphic} with $c=\mathfrak{c}(d)$ which  is monotone decreasing with respect to $d$ and satisfies $\mathfrak{c}(1)=0$.
				\item $\bar{\varphi}_d$ is the unique energy minimizer of the energy functional $\mathcal{J}_d$ defined in \eqref{1Denergy} over $\mathcal{U}_d$. Furthermore, if $0<d_1<d_2\leq 1$, then $\mathcal{J}_{d_1}(\bar{\varphi}_{d_1})>\mathcal{J}_{d_2}(\bar{\varphi}_{d_2})$ and $\bar{\varphi}_{d_1}(x_2)>\bar{\varphi}_{d_2}(x_2)$ for $x_2\in (0, d_1)$.
				\item For any two solutions $\phi_{c_i}$ of \eqref{pbphic} with $c=c_i$ ($i=1$, $2$),  even with shift, they have at most one intersection.
			\end{enumerate}
		\end{theorem}
		
		The more precise explanation for part (4) of Theorem \ref{thmshear} is given in Lemma \ref{lem3.7}. In other words, for any $c_1>c_2\geq 0$ and $s\geq 0$, then $\phi_{c_1}(\cdot-s)$ and $\phi_{c_2}$ cannot have two intersection points in the region where both functions belong to $(0, Q)$.
		
		In fact, for the problem \eqref{pb1.2}, if $f$ is Lipschitz continuous, the uniqueness is established in \cite{BCN} as long as $0<\psi(x_2)<Q$ for $x_2\in (0,1)$. The {non-Lipschitzian} nonlinear term usually yields the non-uniqueness of solutions for Cauchy problem of ODEs. However, we prove the uniqueness for global problem \eqref{pb1.2} even if the nonlinear function $f$ is not Lipschitz.

		The rest of this section devotes to the proof of Theorem \ref{thmshear}.
		Here are the brief descriptions for our key ideas and major observations for solutions of the problem \eqref{pb1.2}. First, any classical solution $\psi\in C^2((0,d)) \cap C^1([0,d])$ of
		\begin{equation}\label{ODE}
			\psi''=f(\psi)
		\end{equation}
		satisfying
		\begin{equation}\label{psi0Q}
			\psi(0)=0
			\quad \text{and}\quad
			0<\psi(x_2)< Q\quad \text{for}\,\, \, x_2\in (0,d)
		\end{equation}
		must be monotone increasing with respect to $x_2$ in $(0,d)$. Second, the solutions of Cauchy problem \eqref{pbphic}
		with $c\geq 0$ must be locally unique via contraction mapping theorem and also monotone with respect to $c$. This gives a one-to-one correspondence between the solutions of \eqref{pbphic} and that of \eqref{pb1.2}. Third, the solution of \eqref{pb1.2} is the unique energy minimizer of the energy functional \eqref{1Denergy} and the associated energy is monotone decreasing with respect to the length of the interval (cf. Lemma \ref{lem3.6}). Finally, using the minimum energy property of $\phi_c$, we prove that even with shift, the ODE solutions of \eqref{pbphic} have at most one intersection. This also implies that the solution of \eqref{pb1.2} is also an energy minimizer in any interval $(d', d'')\subset (0,d)$ with the same boundary values.

		\subsection{Preliminaries and monotone properties of solutions}
		
		Let
		\[
		\hat{f}(t)=\left\{
		\begin{aligned}
			&	f(t),\quad &&\text{if}\,\, t\in [0, Q],\\
			&	0,\quad  &&\text{if}\,\, t\notin [0, Q]
		\end{aligned}
		\right.
		\]
		and
		let $F$ be defined as follows
		\begin{equation}\label{def}
			F(t)=\int_0^t\hat{f}(s)ds.
		\end{equation}
		Clearly,	$F(t)$ is a Lipschitz continuous function although $\hat{f}(t)$ is discontinuous at $t=0$ and $t=Q$. Furthermore, $F(t)=0$ for $t\leq 0$,  and
		\begin{equation}\label{defF}
			F(t)=\int_0^t f(s)ds =\int_0^t 6Q(1-2\kappa(s))ds\quad \text{for}\,\, t\in (0, Q).
		\end{equation}
		Recall that
		\[
		\psi=\int_0^{\kappa(\psi)} \bar{u}_1(s)ds.
		\]
		Note that $\kappa$ is an increasing function. Hence $F$ is a concave function on $(0, Q)$.
		Finally, note that the function $\bar{\varphi}_1$, the inverse function of $\kappa$, satisfies
		\begin{equation}\label{defbarvarphi11}
			\bar{\varphi}_1(t)=\int_0^t \bar{u}_1(s)ds=Q(3t^2-2t^3)
			\quad \text{and}\quad
			\frac{d\bar{\varphi}_1}{dt}=\bar{u}_1(t).
		\end{equation}
		Thus
		\[
		F(t)=\int_0^{\kappa(t)} 6Q(1-2\tau)\bar{u}_1(\tau)d\tau=\frac{(\bar{u}_1)^2(\kappa(t))}{2}.
		\]
		Therefore, we have $F(Q)=0$ and thus $F(t)=0$ for $t\geq Q$ (See Figure \ref{3-0}).		
		
		\begin{figure}
			[h]
			\centering

			\tikzset{every picture/.style={line width=0.75pt}}

\tikzset{every picture/.style={line width=0.75pt}} 

\begin{tikzpicture}[x=0.75pt,y=0.75pt,yscale=-1,xscale=1]

\draw [line width=1.5]    (179.83,219) -- (211.5,219) -- (315.08,219) -- (420.58,219) -- (464.67,219) ;
\draw [shift={(467.67,219)}, rotate = 180] [color={rgb, 255:red, 0; green, 0; blue, 0 }  ][line width=1.5]    (14.21,-4.28) .. controls (9.04,-1.82) and (4.3,-0.39) .. (0,0) .. controls (4.3,0.39) and (9.04,1.82) .. (14.21,4.28)   ;
\draw [line width=1.5]    (211.41,241.99) -- (211.5,219) -- (211.5,31.25) ;
\draw [shift={(211.5,28.25)}, rotate = 90] [color={rgb, 255:red, 0; green, 0; blue, 0 }  ][line width=1.5]    (14.21,-4.28) .. controls (9.04,-1.82) and (4.3,-0.39) .. (0,0) .. controls (4.3,0.39) and (9.04,1.82) .. (14.21,4.28)   ;
\draw  [line width=1.5]  (420.58,219) .. controls (350.71,17.27) and (281.02,17.33) .. (211.5,219.18) ;
\draw  [dash pattern={on 4.5pt off 4.5pt}]  (315.91,67.75) -- (315.08,219) ;
\draw [line width=0.75]  [dash pattern={on 4.5pt off 4.5pt}]  (211.53,67.75) -- (420.28,67.75) ;
\draw  [dash pattern={on 4.5pt off 4.5pt}]  (420.28,67.75) -- (420.58,219) ;

\draw (462.75,222.9) node [anchor=north west][inner sep=0.75pt]  [font=\footnotesize]  {$t$};
\draw (216.83,22.4) node [anchor=north west][inner sep=0.75pt]  [font=\footnotesize]  {$F( t)$};
\draw (213.5,222.4) node [anchor=north west][inner sep=0.75pt]  [font=\footnotesize]  {$O$};
\draw (310.33,222.4) node [anchor=north west][inner sep=0.75pt]  [font=\scriptsize]  {$\frac{Q}{2}$};
\draw (414.33,222.4) node [anchor=north west][inner sep=0.75pt]  [font=\scriptsize]  {$Q$};
\draw (181.33,55.4) node [anchor=north west][inner sep=0.75pt]  [font=\scriptsize]  {$\frac{9}{8} Q^{2}$};

\end{tikzpicture}

			\caption{The concave function $F(t)$}\label{3-0}
		\end{figure}
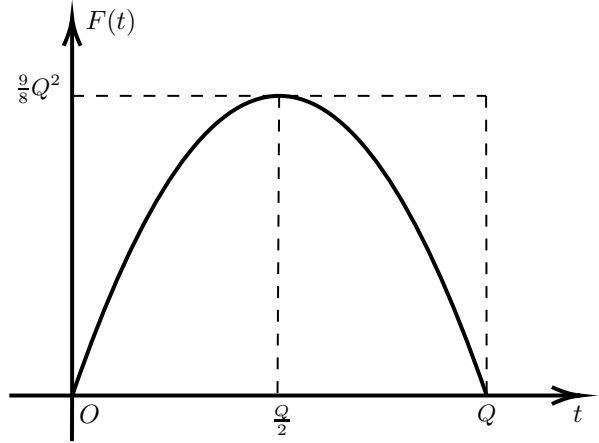
		
		We first prove the monotonicity of solutions to the problem \eqref{pb1.2}.
		\begin{lemma}\label{lem3.1}
			Let {$\psi\in C^2((0,d))\cap C^1([0,d])$}  be a solution of \eqref{ODE}. If $\psi$ satisfies  \eqref{psi0Q},
			then $\psi'(x_2)>0$ for $x_2\in (0,d)$.
		\end{lemma}
		\begin{proof}
			If $\psi$ satisfies \eqref{ODE}, then one has
			\begin{equation}\label{integration}
				\left(\frac{(\psi')^2}{2}\right)' =(F(\psi))'.
			\end{equation}
			Integrating \eqref{integration} from $0$ to $x_2$ yields
			\begin{equation}\label{derivativeest}
				(\psi'(x_2))^2 =(\psi'(0))^2 +2F(\psi(x_2)).
			\end{equation}
			Since $\psi(x_2)\geq 0$ for $x_2\geq 0$ and $\psi(0)=0$, one derives $\psi'(0)\geq 0$. Note that $\psi''(0)=f(\psi(0))>0$, hence there exists a $\delta>0$ such that $\psi'(x_2)>0$ for $x_2\in (0, \delta)$. This, together with \eqref{derivativeest} and the property that $F(\psi)>0$ for $\psi\in (0, Q)$, yields that   $\psi'(x_2)>0$ for $x_2\in (0,d)$. Hence the proof of the lemma is completed.
		\end{proof}
		
		\begin{remark}
			If a solution of \eqref{ODE}
			satisfies \eqref{psi0Q}, then later we will see $d\in (0,1]$.
		\end{remark}
		
		Next, we prove that the solutions of \eqref{pbphic} are monotone with respect to $c$.
		\begin{lemma}\label{lem3.2}
			Let  $c_1> c_2\geq 0$. Suppose that there exist $d_i\in (0, 1]$ ($i=1$, $2$) and
			$\phi_{c_i} \in C^2((0,d_i))\cap C^1([0,d_i])$ satisfying \eqref{pbphic} with $c=c_i$. Let $d=\min\{d_1, d_2\}$. If
			\begin{equation}
				0<\phi_{c_i}(x_2)<Q\quad \text{for}\,\, x_2\in (0, {d}),
			\end{equation}
			then one has
			\begin{equation}
				\phi_{c_1}(x_2)>\phi_{c_2}(x_2) \quad \text{for}\,\, x_2\in (0, d).
			\end{equation}
		\end{lemma}
		\begin{proof}
			Since $c_1>c_2$, there exists a $\delta>0$ such that $\phi_{c_1}(x_2)>\phi_{c_2}(x_2)$ for $x_2\in (0,\delta)$. Let
			\[
			\bar{x}_2 = \sup\{x_2\in(0,d]:\ \phi_{c_1}(s)>\phi_{c_2}(s) \,\, \text{for }s\in (0,x_2) \}.
			\]
			If $\bar{x}_2<d$, then ${\phi}_{c_1}(\bar{x}_2)=\phi_{c_2}(\bar{x}_2)$ and ${\phi}'_{c_1}(\bar{x}_2)\leq \phi_{c_2}'(\bar{x}_2)$.
			Multiplying the ODE in \eqref{pbphic} with $\psi'$ and integrating by parts yield
			\begin{equation}
				\begin{aligned}
					\frac{({\psi}')^2}{2}(\bar{x}_2)-\frac{({\psi}')^2}{2}(0)= &\int_0^{\bar{x}_2}{{\psi}'(s)}{{\psi}''(s)}ds=6Q\int_0^{\bar{x}_2}(1-2\kappa(\psi(s))){\psi}'(s)ds\\
					= &F(\psi)(\bar{x}_2)-F(\psi)(0),
				\end{aligned}
			\end{equation}
			where $F(\psi)(x)$ is defined in \eqref{defF}. Thus one has
			\[
			\frac{({\phi}'_{c_i})^2}{2}(\bar{x}_2)-\frac{({\phi}'_{c_i})^2}{2}(0)=F(\phi_{c_i}(\bar{x}_2))-F(\phi_{c_i}(0))\quad \text{for}\,\, i=1,\,2.
			\]
			This implies that
			\begin{eqnarray}
				\frac{({\phi}'_{c_1})^2}{2}(\bar{x}_2)-\frac{(\phi_{c_2}')^2}{2}(\bar{x}_2)=\frac{c_1^2}{2}-\frac{c_2^2}{2}>0.
			\end{eqnarray}
			It follows from Lemma \ref{lem3.1} that ${\phi}'_{c_i}(\bar{x}_2)> 0$ ($i=1$, $2$) that
			\begin{eqnarray}
				{\phi}'_{c_1}(\bar{x}_2)>\phi_{c_2}'(\bar{x}_2).
			\end{eqnarray}
			This leads to a contradiction. Thus,  $\phi_{c_1}(x_2)>\phi_{c_2}(x_2)$ for $x_2\in (0, d)$.
		\end{proof}
		
		\begin{remark}
			Under the assumption of Lemma \ref{lem3.2}, we can prove that $d_1<d_2$ by the argument in the proof of Lemma \ref{lem3.4}.
		\end{remark}

		\subsection{Uniqueness and existence of solutions}
		Now we prove the local existence and uniqueness of the solutions for problem \eqref{pbphic}.

		\begin{lemma}\label{lem3.3}
			For any $c\geq 0$,
			there exists a $\delta>0$ and a function $\phi_c\in C^2((0,\delta))\cap C^{1}([0,\delta])$ such that $\phi_c$  solves \eqref{pbphic} for  $x_2\in [0,\delta]$.
		\end{lemma}
		\begin{proof} For any solution $\phi_c$ of the problem \eqref{pbphic}, one has
			\begin{equation*}
				\begin{aligned}
					\phi_c(x_2)=&\int_0^{x_2}\phi_c'(\tau)d\tau
					= cx_2+\int_0^{x_2}\int_0^\tau f(\phi_c(s))ds d\tau\\
					=& cx_2+\int_0^{x_2}(x_2-s)6Q(1-2\kappa(\phi_c(s)))ds.
				\end{aligned}
			\end{equation*}
			Hence solving \eqref{pbphic} is equivalent to find a fixed point of the integral equation
			\begin{equation}\label{inteq}
				\phi_c(x_2)=cx_2+\int_0^{x_2}(x_2-s)6Q(1-2\kappa(\phi_c(s)))ds.
			\end{equation}
			It is clear to see that the behavior of solutions for \eqref{inteq} at $x_2=0$ are different for the case $c>0$ and $c=0$, we prove the local existence and uniqueness of \eqref{pbphic} in  two different cases.
			
			{\it Case 1. $c>0$.} Let $w(x_2):=\frac{\phi_c(x_2)}{x_2}$. If $\phi_c$ solves the integral equation \eqref{inteq}, then
			$w$ satisfies
			\begin{eqnarray*}
				w(x_2)=c+\frac{1}{x_2}\int_{0}^{x_2}6Q(x_2-s)(1-2\kappa(s w(s)))ds.
			\end{eqnarray*}
			Let $X=\{w: w\in C([0, \delta]), \frac{c}{2}\leq w(x_2) \leq 2c \,\, \text{for}\,\, x_2\in [0, \delta]\}$ where $\delta>0$ is to be determined. For any $w\in X$, define
			\begin{eqnarray}
				\mathcal{S}w(x_2)=\frac{1}{x_2}\int^{x_2}_06Q(x_2-s)(1-2\kappa(sw(s)))ds+c.
			\end{eqnarray}
			One sees that for any $w\in X$, $\mathcal{S}w \in X$ if $\delta\in (0, \delta_1]$ for some $\delta_1>0$.
			We claim that $\mathcal{S}$ is a contraction map from $X$ to itself when $\delta>0$ is sufficiently small. For any $w_1, w_2\in X$,
			\begin{equation*}
				\begin{aligned}
					&\mathcal{S}w_1(x_2)-\mathcal{S}{w}_2(x_2)=\frac{1}{x_2}\int_{0}^{x_2}12Q(x_2-s)(\kappa(sw_2(s))-\kappa(s{w_1}(s)))ds\\
					=\,\,&\frac{12Q}{x_2}\int_{0}^{x_2}s(x_2-s)\left(\int_0^1{\kappa}'\Big(s\big(w_1(s)+\theta(w_2-w_1)(s)\big)\Big)d\theta\right)({w}_2-w_1)(s)ds.
				\end{aligned}	
			\end{equation*}
			Therefore,
			\begin{equation*}
				\begin{aligned}
					&\|\mathcal{S}w_1-\mathcal{S}{w}_2\|_{C([0,\delta])}\\
					\leq&\left\|\frac{12Q}{x_2}\int_{0}^{x_2}s(x_2-s)\left(\int_0^1{\kappa}'(s(w_1(s)+\theta(w_2-w_1)(s)))d\theta\right)ds\right\|_{C([0,\delta])}\|w_1-{w}_2\|_{C([0,\delta])}\\
					\leq&C\left\|\frac{1}{x_2}\int_{0}^{x_2}\sqrt{s}(x_2-s)ds\right\|_{C([0,\delta])}\|w_1-{w}_2\|_{C([0,\delta])}\\
					\leq&C{\delta}^{\frac{3}{2}}\|w_1-{w}_2\|_{C([0,\delta])},
				\end{aligned}
			\end{equation*}
			where the constant $C$ depends on $c$ and we have used the fact that $\kappa({\psi})\in C^{\frac{1}{2}}([0,Q])\cap C^{\infty}((0,Q))$.
			Hence $\mathcal{S}$ is a contraction map from $X$ to itself if $\delta$ is small enough. Therefore, there is a unique fixed point of the map $\mathcal{S}$ on $X$.

			{\it Case 2. $c=0$.} Denote $\bar{w}(x_2):=\phi_0(x_2)/x_2^2$.
			Then one has
			\begin{eqnarray*}
				\bar{w}(x_2)=\frac{1}{x_2^2}\int_{0}^{x_2}6Q(x_2-s)(1-2\kappa(s^2\bar{w}(s)))ds.
			\end{eqnarray*}
			
			Let $\tilde{X}=\{w: w\in C([0, \delta]), 2Q\leq w(x_2) \leq 4Q \,\, \text{for}\,\, x_2\in [0, \delta]\}$ where $\delta>0$ is to be determined. For any $\bar{w}\in \tilde{X}$, define
			\begin{eqnarray}
				\bar{\mathcal{S}}\bar{w}(x_2)=\frac{1}{x_2^2}\int^{x_2}_06Q(x_2-s)(1-2\kappa(s^2\bar{w}(s)))ds.
			\end{eqnarray}
			The straightforward computations show that for any $\bar{w}\in \tilde{X}$, $\bar{\mathcal{S}}\bar{w} \in \tilde{X}$.
			We claim that $\bar{\mathcal{S}}$ is a contraction map from $\tilde{X}$ to itself when $\delta>0$ is sufficiently small. For any $\bar{w}_1, {\bar{w}}_2\in\tilde{X}$,
			\begin{equation*}
				\begin{aligned}
					&\bar{\mathcal{S}}\bar{w}_1(x_2)-\bar{\mathcal{S}}\bar{w}_2(x_2)=\frac{1}{x_2^2}\int_{0}^{x_2}12 Q(x_2-s)\Big(\kappa(s^2\bar{w}_2(s))-\kappa(s^2\bar{w}_1(s))\Big)ds\\
					=\,\,&\frac{12Q}{x_2^2}\int_{0}^{x_2}s^2(x_2-s)\left(\int_0^1{\kappa}'\Big(s^2\big(\bar{w}_1+\theta(\bar{w}_2-\bar{w}_1)(s)\big)\Big)d\theta\right)(\bar{w}_2-\bar{w}_1)(s)ds.
				\end{aligned}
			\end{equation*}
			Therefore,
			\begin{equation*}
				\begin{aligned}
					&\|\bar{\mathcal{S}}\bar{w}_1-\bar{\mathcal{S}}\bar{w}_2\|_{C([0,\delta])}\\
					\leq&\left\|\frac{12Q}{x_2^2}\int_{0}^{x_2}s^2(x_2-s)\left(\int_0^1{\kappa}'(s^2(\bar{w}_1+\theta(\bar{w}_2-\bar{w}_1)(s)))d\theta\right)ds\right\|_{C([0,\delta])}\|\bar{w}_1-\bar{w}_2\|_{C([0,\delta])}\\
					\leq&C\left\|\frac{1}{x_2^2}\int_{0}^{x_2}s(x_2-s)ds
					\right\|_{C([0,\delta])}\|\bar{w}_1-\bar{w}_2\|_{C([0,\delta])}
					\\                                   \leq&C\delta\|\bar{w}_1-\bar{w}_2\|_{C([0,\delta])},
				\end{aligned}
			\end{equation*}
			where the constant $C$ depends on $Q$ and we have used the fact that $\kappa\in C^{\frac{1}{2}}([0,Q])\cap C^{\infty}((0,Q))$.
			Hence $\bar{\mathcal{S}}$ is a contraction map from $\tilde{X}$ to itself when $\delta$ is sufficiently small. Therefore, there is a unique fixed point of the map $\bar{\mathcal{S}}$ on $\tilde{X}$.
			
			This finishes the proof of the lemma.
		\end{proof}
		Similarly, one can prove that the solution of the problem
		\begin{equation}
			\left\{
			\begin{aligned}
				&{\psi}''(x_2)=f(\psi(x_2)),\ \ \\
				&{\psi}(d)=Q, \ \ {\psi}'(d)=c\geq 0
			\end{aligned}
			\right.
		\end{equation}
		is unique in an interval $[d-\delta, d]$ for some $\delta>0$.

		In the following, we prove the uniqueness of the boundary value problem 
		\begin{equation}
	\label{eeeq}		\left\{
			\begin{aligned}
				&{\psi}''(x_2)=f(\psi(x_2)),\ \ x_2\in (0,1), \\
				&{\psi}(0)=0, \ \ {\psi}(1)=Q.
			\end{aligned}
			\right.
		\end{equation}
		\begin{corollary}
			Let {$\psi\in C^2((0,1))\cap C^1([0,1])$ }be a solution of \eqref{eeeq}. If $\psi$ satisfies
			\begin{equation}\label{uniquecondition}
				0<\psi(x_2)< Q\quad \text{for}\,\, \, x_2\in (0,1),
			\end{equation}
			then
			$$\psi(x_2)=\bar{\varphi}_1(x_2),$$
			where $\bar{\varphi}_1$ is defined in \eqref{defbarvarphi11}.
		\end{corollary}

		\begin{proof} 
			Clearly, $\bar{\varphi}_1$  satisfies \eqref{pb1.2} with $d=1$ and also the problem \eqref{pbphic} with $c=0$. It follows from Lemma \ref{lem3.3} that there exists a $\delta>0$ such that $\bar{\varphi}_1$ is the unique solution of  \eqref{pbphic} with $c=0$ on $[0,\delta]$.
			Since $\psi\in (0, Q)$ for $x_2\in (0, 1)$,  $\kappa({\psi})$ and $f(\psi)$ are Lipschitz continuous on any region away from the boundaries $x_2=0$ and $x_2=1$. It follows from standard ODE theory that the solution is unique when $f(\psi)$ is Lipschitz continuous. Thus, the uniqueness of the solution holds until $x_2=1$.
			This completes the proof of the corollary.
		\end{proof}
		
		\begin{remark}
			Note that the solutions we are interested should satisfy  \eqref{uniquecondition}. This condition plays a crucial role in proving the uniqueness. However, in general, there might be a different solution  of two-points boundary value problem such as \eqref{defbarvarphi11} for ODE, which does not satisfy the condition \eqref{uniquecondition}.
		\end{remark}
		
		Now we are going to study the existence of solution for the problem \eqref{pb1.2}.
		\begin{lemma}\label{lem3.4}
			For any $d\in (0,1]$, there exists a unique solution $\bar{\varphi}_d$ of the problem
			\eqref{pb1.2}.
		\end{lemma}
		\begin{proof}
			For any $d\in (0, 1]$, there exists a unique $\mathfrak{c}(d)\geq 0$ such that
			\begin{equation}\label{ODElifespan}
				d= \int_0^1 \frac{\bar{u}_1(x_2)}{\sqrt{\bar{u}_1^2(x_2)+\mfc^2(d)}} dx_2.
			\end{equation}
			Define
			\[
			\theta_d(x_2)=\int_0^{x_2}\frac{\bar{u}_1(s)}{\sqrt{\bar{u}_1^2(s)+\mfc^2(d)}}ds
			\]
			and
			\[
			\bar{u}_{1,d}(\theta_d(x_2))= \sqrt{\bar{u}_1^2(x_2)+\mfc^2(d)}.
			\]
			Hence $\theta_d$ is a bijective map from $[0,1]$ to  $[0,d]$.
			Furthermore, the straightforward computations yield
			\begin{equation}\label{ODEmass}
				\int_0^{\theta_d(x_2)}\bar{u}_{1,d}(s)ds=	\int_0^{x_2} \bar{u}_1(s)ds
			\end{equation}
			and
			\begin{equation}\label{ODEvorticity}
				\bar{u}_{1,d}'(\theta_d(x_2))	=\bar{u}_{1}'(x_2).
			\end{equation}
			One can check that
			\begin{equation}\label{defvarphid}
   \bar{\varphi}_d(x_2)=\int_0^{x_2}\bar{u}_{1,d}(s)ds
			\end{equation}
			satisfies the problem \eqref{pb1.2}.
			This finishes the proof of the lemma.
		\end{proof}
		
		\begin{remark}
			Note that $f(\psi)=-f(Q-\psi)$. Hence if $\psi$ is a solution, then $Q-\psi(d-x_2)$ is also a solution of \eqref{pb1.2}. Therefore, by the uniqueness, $\psi$ satisfies
			\[
			\psi(x_2)=Q-\psi(d-x_2).
			\]
		\end{remark}
		
		\begin{remark}
			Lemma \ref{lem3.4} implies that there is a one-to-one correspondence between solutions of boundary value problem \eqref{pb1.2} and solutions of Cauchy problem \eqref{pbphic} with some uniquely determined $c$.
		\end{remark}
		\subsection{Variational formulation}
		In this subsection, we show that the one-dimensional solution of the problem \eqref{pb1.2} is in fact a minimizer of energy functional \eqref{1Denergy} over $\mathcal{U}_d$ defined in \eqref{defUd}.
		In the following, we prove the existence and  some interesting properties of the energy minimizer for the functional ${\mathcal{J}}_d (\psi)$. These properties play an important role in understanding shear flows. Furthermore, the ideas to prove these properties also motivate the proof for properties for the energy minimizers in two-dimensional setting in Section \ref{secminimizer}.

		{\begin{lemma}\label{lem3.6}
				The solution $\bar{\varphi}_d$ obtained in Lemma \ref{lem3.4} is the unique minimizer for energy functional ${\mathcal{J}}_d$ over ${\mathcal{U}}_d$, i.e.,
				\begin{equation}\label{defbarJd}
					{\mathcal{J}}_d(\bar\varphi_d)= \bar{\mathcal{J}}_d:=	\inf_{\psi\in {\mathcal{U}}_d} {\mathcal{J}}_d (\psi) .
				\end{equation}
				Furthermore, if $\tilde{d}\in (d, 1]$, then one has
				\[
				\bar{\mathcal{J}}_d>\bar{\mathcal{J}}_{\tilde{d}}.
				\]	
		\end{lemma}}
		
		\begin{proof}
			The proof is divided into three steps.
			
			{\it Step 1. Existence of energy minimizer.} Note that
			\[
			\mathcal{J}_d(\psi)\geq 0 \quad \text{for any}\,\, \psi\in \mathcal{U}_d.
			\]	
			Therefore, one has $\bar{\mathcal{J}}_d\geq 0$. Suppose that  $\{\psi^{(n)}\}\subset \mathcal{U}_d$  is a minimizing sequence such that
			\[
			\mathcal{J}_d(\psi^{(n)})\leq \bar{\mathcal{J}}_d+\frac{1}{n}.
			\]
			Thus one has
			\[
			\int_0^d \left|\frac{d}{dx_2}\psi^{(n)}(x_2)\right|^2dx_2\leq  2\bar{\mathcal{J}}_d+1.
			\]
			Therefore, there exists a function $\psi^*\in \mathcal{U}_d$ such that
			\[
			\psi^{(n)}\rightharpoonup \psi^* \quad \text{in}\,\, H^1(0, d)\,\,\ \text{and}\,\ \,\psi^{(n)}\to \psi^* \ a.e. \ \text{in}\,(0,d).
			\]
			This implies that
			\[
			\mathcal{J}_d(\psi^*)\leq \bar{\mathcal{J}}_d.
			\]
			Hence $\psi^*$ is an energy minimizer of $\mathcal{J}_d$ over $\mathcal{U}_d$.

			{\it Step 2. Properties and characterization of energy minimizer.}
			Suppose that there exists an $\bar{x}_2\in (0, d)$ such that $\psi^*(\bar{x}_2)\in (0, Q/2)$ and  $\frac{d\psi^*}{dx_2}(\bar{x}_2)<0$. Denote
			\[
			\hat{x}_2=\inf\{s:\psi^*(x_2)\geq \psi^*(\bar{x}_2) \,\, \text{for any}\,\, x_2\in (s, \bar{x}_2) \}
			\]
			and
			\begin{equation}
				\hat{\psi}=\left\{
				\begin{aligned}
					& \psi^*(\bar{x}_2),\quad && x_2\in (\hat{x}_2, \bar{x}_2),\\
					& \psi^*(x_2),\quad && x_2\in (0,d)\setminus (\hat{x}_2, \bar{x}_2).
				\end{aligned}
				\right.
			\end{equation}
			Clearly, $\hat{\psi}\in \mathcal{U}_d$ and one has
			\[
			\mathcal{J}_d(\psi^*)> \mathcal{J}_d(\hat{\psi}).
			\]
			This contradicts with the fact that $\psi^*$ is an energy minimizer of $\mathcal{J}_d$ over $\mathcal{U}_d$.
			Similarly, we can prove that there does not exist an $\bar{x}_2\in (0, d)$ such that $\psi^*(\bar{x}_2)\in [Q/2, Q)$ and  $\frac{d\psi^*}{dx_2}(\bar{x}_2)<0$.
			Thus,
			\begin{eqnarray}
				\frac{d\psi^*}{dx_2}({x}_2)
				\geq 0 \ \text{in}\ (0,d).
			\end{eqnarray}
			
			Denote
			\begin{equation}\label{defd_*}
				d_*=\sup\{s: \psi^*(x_2)=0 \,\,\text{for any}\,\, x_2\in (0, s)\}
			\end{equation}
			and
			\begin{equation}\label{defd^*}
				d^*=\inf\{s: \psi^*(x_2)=Q \,\,\text{for any}\,\, x_2\in (s, d)\}.
			\end{equation}
			Clearly, $\psi^*(x_2)\in (0, Q)$ for any $x_2\in (d_*, d^*)$. It is easy to show that $\psi^*$ satisfies the Euler-Lagrange equation
			\[
			\frac{d^2}{dx_2^2}\psi^*=f(\psi^*)\quad \text{in}\,\, (d_*, d^*)
			\]
			and
			\[
			\psi^*(d_*)=0, \quad \psi^*(d^*)=Q.
			\]
   
			Therefore, $\psi_{d_*}^*(x_2):=\psi^*(x_2-d_*)$ solves the problem \eqref{pb1.2} with ${d}=d^*-d_*$. Note that the solution for the problem \eqref{pb1.2} is unique. Hence one has $\psi_{d_*}^*= \bar{\varphi}_{\bar{d}}$ with $\bar{d}:=d^*-d_*$.

			{\it Step 3. Energy comparison and determination of energy minimizer.}	For any  $\tilde{d} \in (d,1]$, define
			\[
			\tilde{\varphi}_{\tilde{d}}=\left\{
			\begin{aligned}
				& \bar{\varphi}_{\bar{d}},\quad &&\text{if}\,\, x\in (0, \bar{d}],\\
				& Q, \quad &&\text{if}\,\, x \in (\bar{d}, {\tilde{d}}).
			\end{aligned}	
			\right.
			\]
			Clearly, $\tilde{\varphi}_{\tilde{d}}\in {\mathcal{U}}_{\tilde{d}}$. Therefore, one has
			\[
			\bar{\mathcal{J}}_{\tilde{d}}\leq	{\mathcal{J}}_{\tilde{d}}(\tilde{\varphi}_{\tilde{d}})=\mathcal{J}_{\bar{d}}(\bar{\varphi}_{\bar{d}})={\mathcal{J}}_{{d}}(\psi^*) =\bar{\mathcal{J}}_{{d}}.
			\]
			
			We claim that $\bar{\mathcal{J}}_{\tilde{d}}<\bar{\mathcal{J}}_{{d}}$.  Let  $\eta\in C_0^\infty (\bar{d}-\delta, \bar{d}+\delta)$ satisfy $\eta\leq 0$ for some $\delta\in (0, ({\tilde{d}}-\bar{d})/2)$. Then for any $\epsilon>0$, one has
			\begin{equation}
				\begin{aligned}
					{\mathcal{J}}_{\tilde{d}}(\tilde{\varphi}_{\tilde{d}}  +\epsilon\eta)-{\mathcal{J}}_{\tilde{d}}(\tilde{\varphi}_{\tilde{d}} ) = &\epsilon \int_{\bar{d}-\delta}^{\bar{d}}(\bar{\varphi}_{\bar{d}}'\eta'+f(\bar{\varphi}_{\bar{d}})\eta) dx +\int_{\bar{d}}^{\bar{d}+\delta}F(Q+\epsilon \eta)dx +O(\epsilon^2)\\
					=& \epsilon \left(\bar{\varphi}_{\bar{d}}'(\bar{d})\eta(\bar{d})+f(Q)\int_{\bar{d}}^{\bar{d}+\delta}
					\eta dx\right) +O(\epsilon^2).
				\end{aligned}
			\end{equation}
			Let $\eta(x)=-\cos \left(\frac{(x-\bar{d})\pi}{2\delta}\right)$. Then one has
			\begin{equation}
				\begin{aligned}
					\frac{1}{\epsilon}(	{\mathcal{J}}_{\tilde{d}}(\tilde{\varphi}_{\tilde{d}} +\epsilon \eta)-{\mathcal{J}}_{\tilde{d}}(\tilde{\varphi}_{\tilde{d}})) =  -\bar\varphi_{\bar{d}}'(\bar{d}) -f(Q)\frac{2\delta}{\pi}+O(\epsilon)<0
				\end{aligned}
			\end{equation}
			provided that $\delta$ is sufficiently small, where $\bar{\varphi}_{\bar{d}}'(0)=\bar{\varphi}_{\bar{d}}'(\bar{d})>0$ has been used. This implies that $\bar{\mathcal{J}}_{\tilde{d}}<\bar{\mathcal{J}}_{{d}}$.	
			
			For $d_*$ and $d^*$ defined in \eqref{defd_*} and \eqref{defd^*},
			if $\bar{d}=d^*-d_*<d$, then one has
			\[
			\bar{\mathcal{J}}_d= \mathcal{J}_{d}(\psi^*)=\mathcal{J}_{\bar{d}} (\psi^*_{d_*})={\mathcal{J}}_{\bar{d}}(\bar{\varphi}_{\bar{d}})>\mathcal{J}_d(\bar{\varphi}_d).
			\]
			This contradiction implies that $\bar{d}=d$ so that $d_*=0$, $d^*=d$, and $\psi^*=\bar{\varphi}_d$.
			
			It follows from {Lemmas \ref{lem3.3} and \ref{lem3.4}} that
			$\bar{\varphi}_d$ is the unique energy minimizer $\mathcal{J}_d$ over $\mathcal{U}_d$.
		\end{proof}

		We have the following interesting properties for the energy minimizers in the intervals with different lengths.
		\begin{lemma}\label{lem3.7}
			Even with a shift, two solutions of \eqref{pbphic} have at most one intersection. More precisely,  for any $c_1>c_2\geq 0$ and $s\geq 0$, if $\phi_{c_1}(\bar{x}_2-s)=\phi_{c_2}(\bar{x}_2)$ for some $\bar{x}_2\geq s$, then
			\begin{equation}
				\phi_{c_1}(x_2-s)>\phi_{c_2}(x_2)\,\, \text{for}\,\, \bar{x}_2<x_2<d_1+s \quad\text{and}\quad	\phi_{c_1}(x_2-s)<\phi_{c_2}(x_2)\,\, \text{for}\,\, s<x_2<\bar{x}_2,
			\end{equation}
			where
			\[
			d_i=\sup\{x_2: \phi_{c_i}(x_2)\in (0, Q), x_2>0\}, \quad i=1,\, 2.
			\]
		\end{lemma}
		\begin{proof}
			
			It follows from Lemma \ref{lem3.4} that $d_2>d_1$.
			Suppose that there exists an $\tilde{x}_2\in(0,d_2)$ such that
			$\phi_{c_1}(\tilde{x}_2-s) =\phi_{c_2}(\tilde{x}_2)$. Without loss of generality, assume that
			\[
			\int_{\bar{x}_2}^{\tilde{x}_2}\frac{|\phi_{c_1}'|^2}{2}(\tau-s)+F(\phi_{c_1}(\tau-s)) d\tau \leq \int_{\bar{x}_2}^{\tilde{x}_2}\frac{|\phi_{c_2}'|^2}{2}(\tau)+F(\phi_{c_2}(\tau )) d\tau.
			\]
			Define
			\[
			\hat{\psi}(x_2)=\left\{
			\begin{aligned}
				& \phi_{c_1}(x_2-s)\quad &&\text{if}\,\, \bar{x}_2<x_2<\tilde{x}_2,\\
				&\phi_{c_2}(x_2)\quad &&\text{if}\,\, x_2\in (0,d_2)\setminus (\bar{x}_2, \tilde{x}_2).
			\end{aligned}
			\right.
			\]
			Clearly, ${\mathcal{J}}_{d_2}({\hat{\psi}})\leq {\mathcal{J}}_{{d}_2}(\phi_{c_2})$. Therefore, ${\hat{\psi}}$ must satisfy the problem \eqref{pb1.2} on $(0,{{d_2}})$ and thus is smooth by the standard regularity theory for ODEs.
			
			Note
			that
			\begin{equation}
				\begin{aligned}
					\frac{|\phi_{c_1}'|^2 (\tilde{x}_2-s)}{2}-\frac{|\phi_{c_1}'|^2(\bar{x}_2-s)}{2}=& F(\phi_{c_1}(\tilde{x}_2-{s}))-F(\phi_{c_1}(\bar{x}_2-s))\\
					=& F(\phi_{c_2}(\tilde{x}_2))-F(\phi_{c_2}(\bar{x}_2))
					= 	\frac{|\phi_{c_2}'|^2 (\tilde{x}_2)}{2}-\frac{|\phi_{c_2}'|^2(\bar{x}_2)}{2}.
				\end{aligned}
			\end{equation}
			It follows from the uniqueness of solutions for \eqref{ODE} that $\phi_{c_1}'(\bar{x}_2-s)\neq\phi_{c_2}'(\bar{x}_2)$.
			Thus, $\phi_{c_1}'(\tilde{x}_2-s)\neq \phi_{c_2}'(\tilde{x}_2)$.
			Hence ${\hat{\psi}}$ is not twice differentiable on $(0, d_2)$. This leads to a contradiction.
			Hence the proof of the lemma is completed.
		\end{proof}

		The proof of Lemma \ref{lem3.7} immediately gives the following corollary.
		\begin{corollary}\label{cor3.9}
			$\bar\varphi_d$ is also an energy minimizer in any local interval in $(0,d)$. More precisely,
			for any $d'$, $d''\in (0,d)$, one has
			\begin{equation}\label{localminimizer}
				\int_{d'}^{d''} \left(\frac{|\bar\varphi_d'|^2}{2}+F(\bar\varphi_d)\right)(s)ds =\inf_{\psi\in \mathcal{U}_{d;d',d''}}	\int_{d'}^{d''} \left(\frac{|\psi'|^2}{2}+F(\psi)\right)(s)ds,
			\end{equation}
			where  $\mathcal{U}_{d;d',d''}=\{\psi: \psi\in H^1(d',d''),\, \psi(d')=\bar\varphi_{d}(d'), \,\psi(d'')=\bar\varphi_{d}(d'')\}$.
		\end{corollary}
		\begin{remark}	
			The idea of the proof for Lemma \ref{lem3.7} is also generalized to study PDEs in Section \ref{secminimizer}, which should be also useful to study other problems.
		\end{remark}

  \begin{remark}
      Corollary \ref{cor3.9} plays a crucial role in the proof for the existence of stagnation points in Section \ref{secexistfbpoint}.
  \end{remark}
		
		Choosing $s=0$ in Lemma \ref{lem3.7} gives the following corollary.
		\begin{corollary}\label{cor3.10}
			For any $0<d<\tilde{d}\leq 1$, one has
			$\bar\varphi_d(x_2)>\bar\varphi_{\tilde{d}}(x_2)$ for $x_2\in (0,d)$.
		\end{corollary}
		
		\begin{proof}[Proof of Theorem \ref{thmshear}]
			Clearly, Part (1) of Theorem \ref{thmshear} follows from {Lemmas   \ref{lem3.1},  \ref{lem3.3}}, and \ref{lem3.4}. Part (2) of Theorem \ref{thmshear} is a direct consequence of Lemmas \ref{lem3.3} and \ref{lem3.4}.  Combining  Lemmas \ref{lem3.6}, \ref{lem3.7}, and Corollary \ref{cor3.10} yields Part (3) of Theorem \ref{thmshear}. Part (4) of Theorem \ref{thmshear} is exactly Lemma \ref{lem3.7}.
		\end{proof}
		
		In the following, we prove that one-dimensional solution $\bar{\varphi}_d$ in fact is an energy minimizer in the rectangle domain $\hat{\Omega}_{l, m}^d$ and that if the energy of a sequence of one-dimensional functions converges to the minimum of energy, then these functions converge to the energy minimizer in $L^\infty$-norm. This plays an important role in proving the asymptotic behavior of two-dimensional flows in a general nozzle in Section \ref{secminimizer}.
		\begin{lemma}\label{lem3.9}
			For any given domain $\hat{\Omega}_{l, m}^d:=\{(x_1, x_2): l<x_1<m, 0<x_2<d\}$ with the constants $l,m\in\mathbb{R}$ and $d\in (0,1]$, let $\hat{\mathcal{K}}^{d}_{l,m}=\{\psi: \psi\in H^1(\hat\Omega_{l, m}^d), \psi(x_1,0)=0, \psi(x_1,d)=Q\}$. Then
			\begin{equation}\label{defJlmd}
				\mathcal{J}_{l,m}^d(\bar{\varphi}_d):=(m-l)\int_0^d \frac{|\bar{\varphi}_d'(x_2)|^2}{2} +F(\bar{\varphi}_d(x_2))dx_2 =\inf_{\psi\in \hat{\mathcal{K}}^d_{l,m}} \int_{l}^m \int_0^d \frac{|\nabla \psi|^2}{2} +F(\psi)dx.
			\end{equation}
			Furthermore, for any $\epsilon>0$, if $\psi\in \mathcal{U}_d$ satisfies
			\[
			\|\psi-\bar{\varphi}_d\|_{L^{\infty}([0,d])}>\epsilon,
			\]
			then there exists a $\sigma>0$ such that
			\begin{equation}
				{\mathcal{J}}_d(\psi) \geq {\mathcal{J}}_d(\bar{\varphi}_d)+\sigma.
			\end{equation}
		\end{lemma}
		\begin{proof}
  The proof is divided into two steps.
  
			{\it  Step 1.}	Given any $\psi\in \hat{\mathcal{K}}^{d}_{l,m}$, one has
			\begin{equation}
				\begin{aligned}
					\int_{\hat\Omega_{l,m}^d} \frac{|\nabla \psi|^2}{2} +F(\psi)dx
					\geq &  \int_{\hat\Omega_{l,m}^d} \frac{|\partial_{x_2}\psi|^2}{2} +F(\psi)dx
					=  \int_{l}^m \int_0^d \frac{|\partial_{x_2} \psi|^2}{2} +F(\psi)dx \\
					\geq & \int_l^m \mathcal{J}_d(\bar{\varphi}_{d})dx_1=\mathcal{J}_{l,m}^d(\bar{\varphi}_d).
				\end{aligned}	
			\end{equation}
			This proves \eqref{defJlmd}.
			
			{\it Step 2.}
			Suppose that there exists an $\epsilon_0>0$ and a sequence $\{\psi^{(n)}\}\subset \mathcal{U}_d$ such that
			\begin{equation}\label{devlarge}
				\|\psi^{(n)}-\bar{\varphi}_d\|_{L^{\infty}([0,d])}>\epsilon_0
			\end{equation}
			and
			\[
			\mathcal{J}_d(\psi^{(n)})\leq \mathcal{J}_d(\bar{\varphi}_d)+\frac{1}{n}.
			\]

			Therefore, $\{\psi^{(n)}\}$ is an energy minimizing sequence. It follows from Lemma \ref{lem3.4} that one has
			\[
			\psi^{(n)} \rightharpoonup \bar{\varphi}_d \quad\text{in}\,\, H^1((0,d)).
			\]
			This implies that $\psi^{(n)} \to\bar{\varphi}_d$ in $C([0,d])$ and contradicts with \eqref{devlarge}. Hence the proof of the lemma is completed.	
		\end{proof}
		
		\section{Uniqueness of the solutions}\label{secunique}
  In this section, we prove the uniqueness of steady Euler flows with positive horizontal velocity inside the nozzle as far field horizontal velocity tends to the one of Poiseuille flows.  The key ingredient of the analysis is to prove a Liouville type theorem for steady Euler system in a strip (cf. Proposition \ref{propthm1}).
		\subsection{Liouville type theorem for flows in a strip}\label{secLiou}
		In this subsection, we study the Liouville type theorem for Poiseuille flows of steady incompressible Euler system \eqref{ieuler}  in an infinitely long strip. With the aid of this particular uniqueness result, we prove the uniqueness of steady solutions in a general nozzle as long as its horizontal velocity is positive inside the nozzle.
 It follows from Proposition \ref{propstream} that the steady incompressible Euler system satisfying \eqref{basicppt1} and \eqref{con} can be reduced into the single second order semilinear elliptic equation \eqref{streameq} of the stream function, where $f$ is given in \eqref{formf}.  Then the key point is to prove that the solutions of problem
		\begin{equation}\label{pb1}
			\left\{
			\begin{aligned}
				&	\Delta \psi =f(\psi)\quad \text{in}\,\,\hat{\Omega}^d,\\
				&	\psi=0\,\, \text{on}\,\, S_0,\quad \psi=Q\,\,\text{on}\,\, S_d
			\end{aligned}\right.
		\end{equation}
		must be one-dimensional solutions as long as
		\begin{equation}\label{nonnegativevelocity}
			\partial_{x_2}\psi \geq 0 \ \ \ \mbox{in} \ \ \hat{\Omega}^d,
		\end{equation}
		where $\hat{\Omega}^d=\mathbb{R}\times (0,d)$, $S_0=\{(x_1,x_2):x_1\in \mathbb{R}, x_2=0\}$, and $S_d=\{(x_1,x_2):x_1\in \mathbb{R}, x_2=d\}$ for any $d \in (0,1]$.
		This is exactly Proposition \ref{prop2}.
		
		\begin{pro}\label{prop2}
			If $\psi\in C^2(\hat{\Omega}^d)\cap C^{0,1}({\overline{\hat{\Omega}^d}})$ is a  solution of \eqref{pb1} and satisfies
			\eqref{nonnegativevelocity},
			then $\psi$ is a function of $x_2$ only, i.e.,
			\begin{eqnarray}
				\psi(x_1,x_2)={\psi}(x_2).
			\end{eqnarray}
		\end{pro}
		\begin{remark}
			The proof of this proposition  mainly uses the sliding method developed in  \cite{BN,BCN0,BCN} and adapted in \cite{HN} to study  the Euler system.
			Based on the elegant argument in \cite{HN}, we mainly need to overcome the difficulties near the boundary where the equation has non-Lipschitz nonlinearity.
		\end{remark}
		\begin{proof}
			The proof is divided into three steps.
			
			{\it Step 1. Set up.} For any ${\nu}=({\nu}_1,{\nu}_2)$ with ${\nu}_1\in \mathbb{R}$ and ${\nu}_2>0$, given
			$\tau\in(0, \frac{d}{\nu_2})$, denote
			\begin{eqnarray*}
				{\Psi}^{\tau}_{\nu}(x)=\psi(x+\tau\nu)-\psi({x})\ \ \mbox{for any}  \ x\in \overline{{D}^{\tau, \nu}},
			\end{eqnarray*}
			where ${D}^{\tau, \nu}=\hat\Omega^{d-\tau\nu_2}$.
			In the following, we compare ${\psi}(x)$ and $\psi(x+\tau\nu)$ in $\overline{{D}^{\tau, \nu}}$ via the maximum principle. A major difficulty is the lack of the Lipschitz continuity for the nonlinear function $f$ in \eqref{pb1}, which is only $C^{1/2}$ in the interval $([0,Q])$. Hence one cannot apply the maximum principle  directly.
			It follows 	from the standard elliptic regularity estimate (\hspace{1sp}\cite{gt}) that there exists a constant $C>0$ such that
			\begin{eqnarray}\label{uni1}
				\|\psi\|_{C^{2,\frac{1}{2}}(\overline{\hat{\Omega}^d})}\leq C.
			\end{eqnarray}
			
			Given any $x_1\in \mathbb{R}$, since
			${\partial}_{x_1x_1} \psi(x_1,0)=0$,
			one has
			\begin{eqnarray*}\label{par1}
				{\partial}_{x_2x_2} \psi(x_1,0)=\Delta \psi (x_1,0)
				=6Q(1-2\kappa (\psi(x_1,0)))=6Q.
			\end{eqnarray*}
			It follows from \eqref{uni1} that there exists a $\delta_1$ independent of $x_1$ such that for any $x_2\in (0,\delta_1)$,
			\begin{eqnarray}\label{par3}
				{\partial}_{x_2x_2}\psi(x_1, x_2)\geq 5Q \,\,\text{for any}\,\, (x_1, x_2)\in \mathbb{R}\times (0, \delta_1).
			\end{eqnarray}
			Therefore, for any $x_2 \in (0, \delta_1)$, there exist $\xi_1, \xi_2\in (0,x_2)$ such that
			\begin{eqnarray}\label{par2.0}
				\partial_{x_2}\psi(x_1,x_2)=\partial_{x_2}\psi(x_1,0)+{\partial}_{x_2x_2}\psi(x_1, \xi_1)x_2 \geq 5Qx_2,
			\end{eqnarray}
			and
			\begin{eqnarray}\label{par2}
				\psi(x_1,x_2)=\psi(x_1,0)+{\partial}_{x_2}\psi(x_1,0)x_2+\frac{1}{2}{\partial}_{x_2x_2}\psi(x_1, \xi_2)x_2^2 \geq \frac{5Q}{2}x_2^2.
			\end{eqnarray}
			
			Similarly, there exists a $\delta_2\in (0, \delta_1)$ such that $\partial_{x_2}\psi(x_1,x_2)\geq 5Q(d-x_2)$, and
			\begin{eqnarray}\label{par2.1}
				\psi(x_1,x_2)\leq Q-\frac{5Q(d-x_2)^2}{2} \ \ \text{for } (x_1, x_2)\in \mathbb{R}\times (d-\delta_2, d).
			\end{eqnarray}
			Hence, there exists a $\sigma_1\in(0, \delta_2)$ such that one has
			\begin{eqnarray}\label{es1}
				0<\psi(x_1,x_2)<\frac{Q}{4} \,\,\text{for any}\,\, (x_1,x_2)\in \mathbb{R}\times (0,\sigma_1),
			\end{eqnarray}
			and
			\begin{eqnarray}\label{es2}
				\frac{3Q}{4}<\psi(x_1,x_2)<Q \,\, \text{for any}\,\,(x_1,x_2)\in \mathbb{R}\times (d-\sigma_1,d).
			\end{eqnarray}
			Clearly, the estimates \eqref{es1} and \eqref{es2}, together with the assumption \eqref{nonnegativevelocity}, not only show that $0<\psi<Q$ in $\hat\Omega^d$, but also yield that there exists an $\epsilon>0$ such that for any $\tau \in  \left(\frac{d}{{\nu}_2}-\epsilon, \frac{d}{{\nu}_2}\right)$,
			\begin{eqnarray*}
				\Psi^{\tau}_{\nu}(x)> 0\ \  \ \text{for any}\,\, (x_1,x_2)\in \overline{{D}^{\tau, \nu}}.
			\end{eqnarray*}

			Define
			\begin{eqnarray*}
				\bar{\tau}=\mbox{inf}\left\{\tau\in\left(0,\frac{d}{\nu_2}\right): \Psi^s_{\nu}(x)>0 \ \mbox{in}\ \overline{{D}^{s, \nu}}  \ \  \mbox{for all}\ s\in\left(\tau,\frac{d}{\nu_2}\right)\right\}.
			\end{eqnarray*}
			Clearly, it holds that $0\leq \bar{\tau}\leq \frac{d}{{\nu}_2}-\epsilon$.
			
			{\it Step 2. Proof for $\bar{\tau}=0$.} Suppose that $\bar{\tau}>0$, then there are three cases:
			
			\emph{(i)} There exists a point $\bar{x}=(\bar{x}_1,\bar{x}_2)\in \partial{{D}^{\bar{\tau}, \nu}} $ such that
			\begin{eqnarray}\label{pf6.0}
				{\Psi}^{\bar{\tau}}_{\nu}(\bar{x}_1,\bar{x}_2)= 0.
			\end{eqnarray}
			
			\emph{(ii)} There exists a point $\bar{x}\in {{D}^{\bar{\tau}, \nu}} $ such that
			\begin{eqnarray}\label{pf6}
				{\Psi}^{\bar{\tau}}_{\nu}(\bar{x}_1,\bar{x}_2)= 0.
			\end{eqnarray}
			
			\emph{(iii)} There exist a sequence $\{\tau_k\}\subset (0,\bar{\tau}]$ converging to $\bar{\tau}$ and a sequence $\{x^k\}_{k\in\mathbb{N}}=\{(x_1^k,x_2^k)\}_{k\in\mathbb{N}}\subset {{D}^{{\tau}_k, \nu}}$  such that
			\begin{eqnarray}\label{2}
				x_1^k\to \infty\quad \text{and}\quad 	{\Psi}^{\tau_k}_{\nu}(x^k)\rightarrow 0 \ \ \ \mbox{as } \ k\to\infty.
			\end{eqnarray}
			In the following, we show that none of these cases happens.
			
			\emph{Case (i).}  If (\ref{pf6.0}) holds, then either $\bar{x}_2=0$ or $\bar{x}_2=d-\bar{\tau}\nu_2$. If $\bar{x}_2=0$, one has
			\begin{eqnarray}\label{bdy1}
				{\Psi}^{\bar{\tau}}_{\nu}(\bar{x}_1, 0)=\psi(\bar{x}_1+\bar{\tau}\nu_1, \bar{\tau}\nu_2)-\psi{(\bar{x}_1, 0)}=\psi(\bar{x}_1+\bar{\tau}\nu_1, \bar{\tau}\nu_2)>0.
			\end{eqnarray}
			If $\bar{x}_2=d-\bar{\tau}\nu_2$, one has
			\begin{eqnarray}\label{bdy2}
				{\Psi}^{\bar{\tau}}_{\nu}(\bar{x}_1, d-\bar{\tau}\nu_2)=\psi(\bar{x}_1+\bar{\tau}\nu_1, d)-\psi{(\bar{x}_1, d-\bar{\tau}\nu_2)}=Q-\psi{(\bar{x}_1, d-\bar{\tau}\nu_2)}>0.
			\end{eqnarray}
			Hence it is impossible that $\bar{x}\in \partial{{D}^{\bar{\tau}, \nu}}$.
			
			\emph{Case (ii).} If (\ref{pf6}) holds,
			it follows	from the equation (\ref{pb1}) that there exists a $\delta>0$ such that $B_\delta(\bar{x})\subset{D}^{\bar{\tau}, \nu}$ and
			\begin{equation*}\label{6}
				\left\{
				\begin{aligned}
					&\Delta {\Psi}^{\bar{\tau}}_{\nu}(x)+c(x){\Psi}^{\bar{\tau}}_{\nu}(x)=0 \quad &&\text{in}\,\,  B_\delta(\bar{x}),\\
					&{\Psi}^{\bar{\tau}}_{\nu}(x)\geq 0 \quad &&\text{on}\,\, \partial B_\delta(\bar{x}),\\
					&\Psi_{\nu}^{\bar{\tau}}(\bar{x})=0,
				\end{aligned}
				\right.
			\end{equation*}
			where
			\begin{eqnarray*}
				c(x)=-\frac{{f({\psi}(x+\bar{\tau}\nu)})-f({\psi}({x}))}{{\psi}(x+\bar{\tau}\nu)-{\psi}({x})}.
			\end{eqnarray*}
			Note that $B_\delta(\bar{x})\subset{D}^{\bar{\tau}, \nu}$. Hence there exists an $\varepsilon_0 >0$ such that $\psi(B_\delta(\bar{x}))\subset [\varepsilon_0, Q-\varepsilon_0]$ and $\psi(x+\tau\nu)\in [\varepsilon_0, Q-\varepsilon_0]$ for $x\in B_\delta(\bar{x})$. Clearly, $f$ is  Lipschitz continuous on $[\varepsilon_0, Q-\varepsilon_0]$. Hence $c(x)$ is a bounded function in $ B_\delta(\bar{x})$. It follows from the strong maximum principle (\hspace{1sp}\cite{hl}) that ${\Psi}^{\bar{\tau}}_{\nu}(x)\equiv 0$ in $B_\delta (\bar{x})$. 	
			Thus  ${\Psi}^{\bar{\tau}}_{\nu}(x)\equiv 0$ in ${D}^{\bar{\tau}, \nu}$ and then on $\partial{D}^{\bar{\tau}, \nu}$ by continuity. This leads to a contradiction.
			
			\emph{Case (iii).} Suppose that (\ref{2}) holds. Denote
			\begin{equation}\label{pb4}
				{\psi}_k(x)={\psi}(x_1+x_1^k, x_2) \ \ \ \mbox{in} \ \ \overline{\hat{\Omega}^d}.
			\end{equation}
			It follows	from the standard elliptic estimates and $f(\psi)\in C^{\frac{1}{2}}(\overline{\hat{\Omega}^d})$ that the sequences of the functions $\{{\psi}_k(x)\}$ and $\{{\psi}_k(x+{\tau}_k\nu)-{\psi}_k(x) \}$ are bounded in $C^{2,\frac{1}{2}}(\overline{\hat{\Omega}^d})$ and $C^{2,\frac{1}{2}}(\overline{{D}^{\bar{\tau}, \nu}})$, respectively. Then there exists a subsequence which is still labeled by $\{{\psi}_k(x)\}$ and converges in $C^2_{\text{loc}}(\overline{\hat{\Omega}^d})$ to $\bar\psi(x)$. Clearly, $\bar\psi \in C^2({\overline{\hat{\Omega}^d}})$ satisfies the equation in \eqref{pb1},
			\begin{eqnarray}\label{Psi}
				\bar\psi(x)=0 \ \ \ \mbox{on} \ S_0,\ \ \bar\psi(x)=Q \ \ \ \mbox{on} \ S_d,
			\end{eqnarray}
			and
			\begin{eqnarray}\label{ps}
				\partial_{x_2}{\bar\psi}(x)\geq 0 \ \ \ \mbox{in} \ \overline{\hat{\Omega}^d}.
			\end{eqnarray}
			Thus $\bar\psi$ satisfies the problem \eqref{pb1}.
			It follows from \eqref{par2} and \eqref{par2.1} that  $0<\bar\psi<Q$ in $\mathbb{R}\times ((0,\delta_0)\cup (d-\delta_0,d))$. This, together with \eqref{ps}, yields
			\begin{eqnarray}\label{psi1}
				0<\bar\psi(x)<Q \  \ \ \mbox{in} \ \hat\Omega^d.
			\end{eqnarray}
			
			Similarly, after choosing a subsequence if necessary, $\{{\psi}_k(x+{\tau}_k\nu)-{\psi}_k(x)\}$ converges in $C^2_{\text{loc}}(\overline{{D}^{\bar{\tau}, \nu}})$ to the function $\bar\psi(x+\bar{\tau}\nu)-\bar\psi(x)$ defined in $\overline{{D}^{\bar{\tau}, \nu}}$.
			
			Note that (\ref{2}) implies that
			\begin{eqnarray*}
				{\psi}(x_1^k+{\tau_k}{\nu}_1, x_2^k+\tau_k{\nu}_2)- {\psi}(x_1^k, x_2^k)\rightarrow 0\textup{ as }k\to\infty.
			\end{eqnarray*}
			By the definition of ${\psi}_k(x)$, one has
			\begin{eqnarray*}
				{\psi}_k(\tau_k{\nu}_1, x_2^k+\tau_k{\nu}_2)- {\psi}_k(0, x_2^k)\rightarrow 0\mbox{ as }k\to\infty.
			\end{eqnarray*}
			Up to choosing a subsequence, one can assume that $x_2^k\rightarrow x_2^{*}$ for some
			$x_2^{*}\in[0,d-\bar{\tau}{\nu}_2]$.
			Denote	$x^{*}=(0,x_2^{*})$.
			It holds that
			\begin{eqnarray}\label{4}
				\bar\psi(x^{*}+\bar{\tau}{\nu})=\bar\psi(x^{*}).
			\end{eqnarray}
			Hence it follows from \eqref{bdy1}, \eqref{bdy2}, and \eqref{psi1} that $x^{*}$ is an interior point of the set $\mathbb{R}\times [0,d-\bar{\tau}{\nu}_2]$. This leads to a contradiction with the strong maximum principle, as in the proof of {\it{Case (ii)}}. Therefore, the claim \eqref{2} cannot happen. Thus $\bar{\tau}=0$.
			
			{\it Step 3. Reduction to one-dimensional solutions.}
			Since  $\nabla \psi$ is continuous, one has ${\partial}_{\nu}\psi\geq 0$ for any $\nu$ with ${\nu}_2=0$. Taking $\nu=(1,0)$ and $\nu=(-1,0)$ gives $\partial_{x_1}\psi(x)=0$. Therefore, $\psi(x)$ does not depend on $x_1$. This implies that
			$\psi(x)={\psi}(x_2)$ in $\hat\Omega^d$.
			
			The proof of Proposition \ref{prop2} is completed.
		\end{proof}
		
	\begin{proof}[Proof of Proposition \ref{propthm1}] It follows from Proposition \ref{prop2} that ${\psi}(x_1, x_2)=\psi(x_2)=\bar{\varphi}_d(x_2)$. As it was showed in \eqref{uni1}, one has $\psi\in C^{2, 1/2}([0,d])$. This, together with Lemmas \ref{lem3.3} and \ref{lem3.4} that $\psi(x_2)=\bar{\varphi}_d(x_2)$ for any $d\in(0,1]$, where $\bar{\varphi}_d$ is defined in \eqref{defvarphid}. Hence ${\bf u}=(\bar{u}_{1,d}(x_2), 0)$ in $\hat\Omega^d$. Proposition \ref{propthm1} corresponds to the special case with $d=1$, i.e., ${\bf u}=(\bar{u}_1(x_2),0)= (6Qx_2(1-x_2), 0)$ in $\hat\Omega$.
 \end{proof}

		\subsection{Fine properties and uniqueness of flows in a nozzle}
		In this subsection, we prove the uniqueness of solution in a general nozzle with the aid of the Liouville type theorem established in Section \ref{secLiou}.
		
		First, we study far field behaviors of the solution to the problem
		\begin{equation}\label{psipb}
			\left\{
			\begin{aligned}
				&	\Delta\psi =f(\psi)\quad \text{in}\,\, \Omega,\\
				&	\psi=0\,\,\text{on}\,\, \Gamma_0,\quad \psi=Q\,\,\text{on}\,\, \Gamma_1.
			\end{aligned}\right.
		\end{equation}
		
		\begin{lemma}\label{lemmaasymptotic}
			If $\psi\in C^2(\Omega)\cap C^{0,1}(\bar\Omega)$ is a solution of \eqref{psipb} and satisfies $\partial_{x_2}\psi\geq 0$ in $\Omega$, then one has
			\begin{equation}\label{downstreambehavior}
				\|\psi(x_1, \cdot)- \bar\varphi_{1}(\cdot)\|_{C_{loc}^1((0,1))}\to 0 \quad \text{as}\,\, x_1\to -\infty
			\end{equation}
			and
			\begin{equation}\label{upstreambehavior}
				\|\psi(x_1, \cdot)-\bar\varphi_{b-a}(x_1,\cdot-a)\|_{C_{loc}^1((a,b))} \to 0\quad \text{as}\,\, x_1\to +\infty.
			\end{equation}	
		\end{lemma}
		\begin{proof} The proof is based on the shifting method and the Liouville type theorem proved in Section 4.
			For any fixed  $x_1\in \mathbb{R}$, define
			\begin{equation}
				\psi^{(n)}(x_1,x_2)=\psi(x_1-n, x_2)\chi_{\{h_0(x_1-n)<x_2<h_1(x_1-n)\}}.
			\end{equation}
			It follows from \eqref{uni1} that
			\[
			\|\psi^{(n)}\|_{C^{2,\frac{1}{2}}(K)} \leq C, \ \ \text{as } n \text{ sufficiently large},
			\]
			for any compact set $K\Subset \mathbb{R}\times (0,1)$.
			Therefore, by Arzela-Ascoli lemma and a diagonal procedure, for any $\gamma\in (0,1/2)$, there exists a subsequence $\{{\psi}^{(n_k)}\}$ such that
			\begin{eqnarray}
				{\psi}^{(n_k)}\rightarrow {\psi}^{*} \ \ \mbox{in} \ \ C^{2,\gamma}(K) \ \ \mbox{as} \ k\to \infty.
			\end{eqnarray}
			Furthermore, $\psi^*$ satisfies $\partial_{x_2}\psi^*\geq 0$ and the problem
			\begin{equation}\label{psi*pb}
				\left\{
				\begin{aligned}
					&	\Delta\psi^* =f(\psi^*)\quad \text{in}\,\, \hat\Omega,\\
					&	\psi^*=0\,\,\text{on}\,\, S_0,\quad \psi^*=Q\,\,\text{on}\,\, S_1.
				\end{aligned}\right.
			\end{equation}
			It follows from Proposition \ref{prop2}  and Lemma \ref{lem3.4} that $\psi^*(x_1, x_2) =\bar\varphi_{1}(x_2)$. This yields the asymptotic behavior \eqref{downstreambehavior}.
			Similarly, one can prove that \eqref{upstreambehavior} holds.
			Hence the proof of the lemma is completed.
		\end{proof}
		
		Next, we show the positivity of the horizontal velocity in the nozzle for the solution of \eqref{psipb}.
		
		\begin{lemma}\label{lemmapositive}
			Suppose that $\psi\in C^2({\Omega})\cap C^{0,1}({\overline{{\Omega}}})$ solves the problem
			\eqref{psipb}
			and
			satisfies $\partial_{x_2}\psi\geq 0$ in $\Omega$. Then
			\[
			\partial_{x_2}\psi>0\quad \text{in}\,\, \Omega.
			\]
		\end{lemma}
  
		\begin{proof}The proof is divided into two steps.
			
			{\it Step 1. Proof for the property $0<\psi<Q$.} 
			Clearly, $\psi$ satisfies $0\leq\psi\leq Q$ in $\Omega$. Suppose that there exists an $x\in \Omega$ such that $\psi(x)=0$. Define
			\[
			A=\{(x_1, x_2): \psi(x_1, x_2)=0,(x_1, x_2)\in \Omega\}.
			\]
			Choose a point $x^* =(x_1^*, x_2^*)\in A$ such that
			\[
			x_2^*-h_0(x_1^*)=\sup\{x_2-h_0(x_1):(x_1, x_2)\in A\}.
			\]
			Note that	$x^*$ may not be unique. It follows from Lemma \ref{lemmaasymptotic} that there exists an $L>0$ such that $|x_1^*|\leq L$.  Then one has
			$\psi(x_1^*, x_2)=0$ {for} $x_2\in (h_0(x_1^*), x_2^*)$.
			This yields
			\begin{equation}\label{vderivative}
				\partial_{x_2}\psi(x_1^*, x_2)= \partial_{x_2x_2}\psi(x_1^*, x_2)=0 \quad \text{for}\,\, x_2\in (h_0(x_1^*), x_2^*).
			\end{equation}
			Let $B$ be the ball centered at $(x_1^*, x_2^*+\epsilon)$ with radius $\epsilon$. If $\epsilon>0$ is small enough, then $B\subset \Omega$ and $0\leq \psi(x)<Q/2$. Let $W(x_1, x_2)=\psi(x_1, x_2)-\psi(x_1, x_2-\tau)$ for some sufficiently small $\tau \in (0, 2\epsilon)$. Then
			\begin{equation*}
				\left\{
				\begin{aligned}
					& \Delta W(x)=12Q(\kappa(\psi(x_1, x_2-\tau))-\kappa(\psi(x)))\leq 0,\quad &&\text{in}\,\, B,\\
					& W(x) \geq 0,\quad &&\text{in}\,\, B,\\
					& W(x^*)=0,
				\end{aligned}
				\right.
			\end{equation*}
			where the properties of $\partial_{x_2}\psi$ and that $\kappa$ is an increasing function have been used.
			Note that $W$ is not identically zero in $B$.
			It follows from Hopf lemma (\hspace{1sp}\cite[Lemma 1 in Section 9.5]{evans}) that $\frac{\partial W}{\partial x_2}(x^*)>0$, which leads to a contradiction with \eqref{vderivative}. 	
			Thus one has
			\begin{equation*}
				\psi(x)>0 \  \text{ in } \ \Omega.
			\end{equation*}
			Similarly, one can prove that $\psi(x)<Q$ in $\Omega$.
			
			{\it Step 2. Proof for the property $\partial_{x_2}\psi>0$.}  Suppose that there exists a point $x^*=(x_1^*, x_2^*)\in \Omega$ such that $\partial_{x_2}\psi(x^*)=0$. Note that $\psi(x^*)\in (0, Q)$. Hence there exist $\epsilon, \delta>0$ such that $B_\delta(x^*)\subset\Omega$ {and}
			\[
			\epsilon\leq \psi \leq Q-\epsilon\,\,\text{in}\,\, B_\delta(x^*).
			\]  
   Note that $f\in C^\infty([\epsilon, Q-\epsilon])$ so that one can apply the regularity theory for elliptic equations to get $\psi\in C^\infty(\overline{B_\delta(x^*)})$.
   Clearly, $\partial_{x_2}\psi$ satisfies
			\begin{equation}
				\left\{
				\begin{aligned}
					&	\Delta \partial_{x_2}\psi -f'(\psi)\partial_{x_2}\psi=0,\,\, && \text{in}\,\, B_\delta(x^*),\\
					&	\partial_{x_2}\psi\geq 0,\,\,&&\text{in}\,\, B_\delta(x^*),\\
					& \partial_{x_2}\psi(x^*)=0.
				\end{aligned}\right.
			\end{equation}
			Applying the strong maximum principle yields $\partial_{x_2}\psi=0$ in $B_\delta (x^*)$. Thus $\partial_{x_2}\psi\equiv 0$ in $\Omega$. This is absurd.
			Hence the proof of the lemma is completed.
		\end{proof}

		Now we are ready to prove the uniqueness of the solutions of \eqref{psipb}.  Our basic strategy is  to compare the solution with its shifts in the vertical direction.
		\begin{pro}\label{propunique}
			There is at most one solution $\psi$ of \eqref{psipb} satisfying $\partial_{x_2}\psi\geq 0$.
		\end{pro}
		\begin{proof}
			It follows from Lemma \ref{lemmapositive} that $0<\psi(x)<Q$ in $\Omega$. Suppose that there are two solutions $\psi_i$ ($i=1$, $2$) of the problem \eqref{psipb} with the properties $\partial_{x_2}\psi_i\geq 0$ ($i=1$, $2$) in $\Omega$. Extend $\psi_i$ on the whole plane as follows
			\begin{equation}
				\tilde\psi_i(x_1,x_2)=\left\{
				\begin{aligned}
					&Q\quad &\text{if}\,\,x_2\geq h_1(x_1),\\
					&\psi_i(x_1,x_2)\quad &\text{if}\,\, (x_1,x_2)\in \Omega,\\
					&0\quad &\text{if}\,\,x_2 \leq h_0(x_1).
				\end{aligned}\right.
			\end{equation}	
			
			For any $\tau>0$, denote $\tilde\psi_{i,\tau}(x)=\tilde\psi_i(x_1, x_2+\tau)$. It follows from Lemma \ref{lemmapositive} that $\tilde\psi_{1, \tau}(x)>\tilde\psi_2(x)$ for $x\in\Omega$ when $\tau$ is sufficiently large.
			
			Let
			\[
			\bar\delta =\inf\{ \delta>0:\tilde\psi_{1, \tau}(x)>\tilde\psi_2(x) \, \text{in } \Omega \text{ for any}\,\, \tau >\delta\}.
			\]
			If $\bar \delta>0$, then it follows from Lemma \ref{lemmaasymptotic} that there are two possibilities:
			
			(i) There exists a point $\bar{x}\in \Omega$ such that
			\begin{equation}\label{uniq}
				\tilde\psi_{1,\bar{\delta}}(\bar x)=\tilde\psi_2(\bar x).
			\end{equation}
			
			(ii) There exist a sequence $\{\delta_k\}\subset (0,\bar{\delta}]$ converging to $\bar{\delta}$ and a sequence $\{x^k\}_{k\in\mathbb{N}}\subset \Omega$ such that
			\begin{eqnarray}\label{uniq1}
				\tilde\psi_{1,{\delta}_k}(x^k)-\tilde\psi_2(x^k)\leq 0  \quad \text{and}\quad x^k\to \bar{x}\in \partial\Omega\quad \text{as}\quad  k\to \infty.
			\end{eqnarray}
			In the following, we show that neither of these two cases happens.
			
			{\it Case (i).} Suppose that \eqref{uniq} holds. There exists an $\epsilon>0$ such that $B_\epsilon(\bar x)\subset \Omega$ and
			\begin{equation}
				\left\{
				\begin{aligned}
					&	\Delta  	(\tilde\psi_{1,\bar{\delta}}-\tilde\psi_2)(x)+c(x) 	(\tilde\psi_{1,\bar{\delta}}-\tilde\psi_2)(x)=0,\quad &  \text{in}\,\, B_\epsilon(\bar{x}),\\
					&	(\tilde\psi_{1,\bar{\delta}}-\tilde\psi_2)(x)\geq 0,\quad &\text{in}\,\, B_\epsilon(\bar{x}),\\
					& 	(\tilde\psi_{1,\bar{\delta}}-\tilde\psi_2)(\bar{x})=0,
				\end{aligned}\right.
			\end{equation}
			where
			\begin{eqnarray*}
				c(x)=-\frac{f(\tilde\psi_{1,\bar{\delta}}(x))-f(\tilde\psi_2({x}))}{\tilde\psi_{1,\bar{\delta}}(x)-\tilde\psi_2({x})}.
			\end{eqnarray*}
			Applying the strong maximum principle yields $\tilde\psi_{1,\bar\delta}(x)=\tilde\psi_2(x)$ for $x\in B_\epsilon(\bar x)$. Thus $\tilde\psi_{1,\bar\delta}(x)=\tilde\psi_2(x)$ holds in $\Omega$. This leads to a contradiction.
			
			{\it  Case (ii).} Suppose that \eqref{uniq1} holds. Without loss of generality, one can assume that
			\begin{equation}
				\delta_k>\frac{\bar\delta}{2} \ \ \text{ and } \ \ |x^k-\bar{x}|<\frac{\bar\delta}{10} \ \text{ for }  k \text{ sufficiently large}.
			\end{equation}
			This implies that
			\begin{eqnarray}
				\tilde\psi_{1,\delta_k}(x^k)=\tilde\psi_1(x_1^k,x_2^k+\delta_k)=Q\ \ \text{for }(x_1^k,x_2^k)\in \Omega.
			\end{eqnarray}
			
			On the other hand, one has
			\begin{equation}
				0<\tilde\psi_2(x_1^k,x_2^k)<Q \ \ \text{for }(x_1^k,x_2^k)\in \Omega.
			\end{equation}
			Then $\tilde\psi_{1,{\delta}_k}(x^k)-\tilde\psi_2(x^k)>0$. This leads to a contradiction.
			
			Thus $\bar\delta=0$. By continuity, one has $\psi_1(x)\geq \psi_2(x)$ in $\Omega$. Similarly, one can prove that $\psi_2(x)\geq \psi_1(x)$ in $\Omega$. Therefore,
			\begin{eqnarray}
				\psi_1(x)\equiv \psi_2(x) \  \text{ in } \ \Omega.
			\end{eqnarray}
			This finishes the proof of the proposition.
		\end{proof}

		Combining Propositions \ref{propunique} and \ref{propstream}  yields Theorem \ref{thmunique}.

		\section{Existence of solutions in a general nozzle and analysis on the stagnation points}\label{secminimizer}
		In this section, we prove the existence of solutions for the Euler flows in a general nozzle whose downstream far field has asymptotic height not bigger than $1$. The existence of solutions is proved by the variational method. The key issue is that an obstacle type free boundary (the boundary of the set containing stagnation points) may appear where the solutions take the value $0$ or $Q$.
		Since the governing equations in the regions on the two sides of the free boundary are different, the analysis on the regularity and associated properties is pretty subtle and should be very careful. 
		
		\subsection{Set up and existence in truncated domain}

		For any given $N>0$, denote $\Omega_N=\Omega\cap\{(x_1,x_2): |x_1|<N\}$ and $g_N\in H^1(\Omega_N)$ satisfies
\begin{equation}
	g_N=\left\{
	\begin{aligned}
		& 0,\quad &&\text{if}\,\, (x_1, x_2)\in \Gamma_{0,N},\\
		& Q,\quad &&\text{if}\,\, (x_1, x_2)\in \Gamma_{1,N},\\
		& \bar{\varphi}_{h_1(\pm N)-h_0(\pm N)}(x_2-h_0(\pm N)),\,\, &&\text{if}\,\, x_1=\pm N, x_2\in (h_0(\pm N), h_1(\pm N)),
	\end{aligned}\right.
\end{equation}
where 
\begin{equation}\label{defGammaiN}
\Gamma_{i, N} :=  \{(x_1, x_2): x_2 ={h}_{i,N}(x_1), -N<x_1<N\},\quad i=0,1.
\end{equation}
Define
	\[
\mathcal{K}_N =\{\psi: \psi\in H^1(\Omega_N), \psi=g_N \,\,\text{on}\,\, \partial \Omega_N\}
\]
and
	\[
\mathcal{E}_N(\psi)= \int_{\Omega_N} \frac{|\nabla \psi|^2}{2} +F(\psi)dx,
\]
		where $F$ is defined in \eqref{def}.
		
		\begin{lemma}\label{lem6.2}
			For any given $N>0$, there exists a minimizer $\psi_N$ for the energy functional $\mathcal{E}_N$ over $\mathcal{K}_N$, i.e.,
	\[
\mathcal{E}_N(\psi_N)=\mathcal{I}_N:=\inf_{\psi\in \mathcal{K}_N}\mathcal{E}_N(\psi)=\inf_{\psi\in \mathcal{K}_N} \int_{\Omega_N} \frac{|\nabla \psi|^2}{2} +F(\psi)dx.
		\]
Furthermore, $\psi_N\in C^{1, \alpha_0}({\Omega}_N)$ for some $\alpha_0\in (0,1)$ satisfies
\[
	0\leq \psi_N(x) \leq {Q} \quad \text{for any}\,\, x\in \Omega_N.
	\]
		\end{lemma}
		\begin{proof}
			The proof is divided into two steps.

\emph{Step 1. Existence.}
Clearly, for any $\psi\in \mathcal{K}_N$, $\mathcal{E}_N(\psi)\geq 0$. Hence there exists a minimizing sequence $\{\psi_N^{(n)}\}$ such that $\lim\limits_{n\to \infty}\mathcal{E}_N(\psi_N^{(n)})=\mathcal{I}_N$. Note that
	\[
	\|\nabla \psi_N^{(n)}\|_{L^2(\Omega_N)}^2\leq 2\mathcal{I}_N+1.
	\]
	Since $\psi_N^{(n)}=0$ on $\Gamma_0\cap \partial\Omega_N$, it follows from Poincar\'{e} inequality that
	\[
	\|\psi_N^{(n)}\|_{H^1(\Omega_N)}\leq C.
	\]

	Therefore, there exists a subsequence $\{\psi_N^{(n_k)}\}$ such that
	\[
	\psi_N^{(n_k)}\rightharpoonup \psi_N \quad \text{in} \,\, H^1(\Omega_N) \quad \text{and}\,\, \  \psi_N^{(n_k)}\to \psi_N \,\,a.e.\,\, \text{in}\,\, \Omega_N.
	\]
This implies
	\[
	\mathcal{E}_N(\psi_N)\leq \liminf_{k\to \infty} \mathcal{E}_N(\psi_N^{(n_k)}).
	\]
	Thus $\psi_N$ is the minimizer for $\mathcal{E}_N$ over $\mathcal{K}_N$.

\emph{Step 2. Boundedness and regularity.} Let $\psi_{N,+}=\max\{\psi_N, 0\}$. Clearly, one has
	\[
	{\mathcal{E}}_N(\psi_{N,+})\leq {\mathcal{E}}_N(\psi_N).
	\]
	Hence {$\mathcal{H}^2(\{\psi_N<0\})=0$}, where $\mathcal{H}^2$ is the two-dimensional Hausdorff measure. This implies $\psi_N\geq 0$ in $\Omega_N$. Similarly, one can prove that $\psi_N\leq Q$ in $\Omega_N$.
	
	It follows from the standard regularity theory for the minimizer of variational solutions (cf. \cite{Giusti}) that there exists an $\alpha_0\in (0,1)$ such that $\psi_{N}\in C^{1, \alpha_0}({\Omega}_N)$. Hence the proof of the lemma is completed.
		\end{proof}
		
		\begin{lemma}
If $\psi_N$ is a minimizer of $\mathcal{E}_N$ over $\mathcal{K}_N$, then for any $\beta\in (0,1)$,  $\psi_N\in C^{2, \beta}({\Omega}_N\cap \{x:0<\psi_N(x)<Q\})$ and satisfies
	\begin{equation}\label{ELeq}
	\Delta \psi_N =f(\psi_N)\quad \text{in} \ \,\,{\Omega}_N\cap \{x:0<\psi_N(x)<Q\}.
	\end{equation}
Furthermore, one has
	\[
	|\nabla \psi_N|=0\quad \text{on}\,\, \ (\partial\{\psi_N>0\}\cup\partial\{\psi_N<Q\})\cap \Omega_N.
	\]
\end{lemma}
\begin{proof}
The proof is divided into two steps.

\emph{Step 1. Euler-Lagrange equation.} Since $\psi_N\in C({\Omega}_N)$, the set $\{x: 0<\psi_N(x)<Q\}$ is open in ${\Omega}_N$. Given any point $\bar{x}\in {\Omega}_N\cap\{x: 0<\psi_N(x)<Q\}$, there exists a $\delta>0$ such that $B_\delta(\bar{x})\subset  {\Omega}_N\cap\{x: 0<\psi_N(x)<Q\}$. For any $\eta\in C_0^\infty (B_\delta(\bar{x}))$, denote
\[
\mathcal{L}(\theta)=\int_{{\Omega}_N}\frac{|\nabla(\psi_N+\theta \eta)|^2}{2}+F(\psi_N+\theta \eta)dx.
\]
Therefore, there exists a $\theta_0>0$ such that $\psi_N+\theta \eta \in (0, Q)$ for $\theta \in (-\theta_0, \theta_0)$.
Clearly, $\mathcal{L}(\theta)$ achieves its minimum at $\theta=0$. Therefore, one has
\[
0=\mathcal{L}'(0)=\int_{\tilde{\Omega}_N} \nabla\psi_N \nabla \eta +f(\psi_N)\eta dx.
\]
This implies that $\psi_N$ solves the equation in \eqref{ELeq} in $B_\delta(\bar{x})$. It follows from the regularity theory for the weak solution of the equation in \eqref{ELeq} that $\psi_N\in C^{2, \beta}({\Omega}_N\cap \{x: 0<\psi_N(x)<Q\})$ for any $\beta\in (0,1)$.

\emph{Step 2. Properties of the energy minimizer on the free boundaries.} Let $\eta\in C_0^\infty(\Omega_{N}; \mathbb{R}^2)$ and  $\tau_\vartheta(x):=x+\vartheta\eta(x)$ for $\vartheta>0$. If $|\vartheta|$ is sufficiently small, then $\tau_\vartheta$ is a diffeomorphism of $\Omega_{N}$. Let $\psi^{\vartheta}_N(y):=\psi_N(\tau_\vartheta^{-1}(y))$. Since $\psi_N^\vartheta\in \mathcal{K}_N$ and $\psi_N$ is a minimizer, then
\begin{equation}\label{est_m1}
	\begin{aligned}
		0\leq &\mathcal{E}_{N}(\psi^{\vartheta}_N)-\mathcal{E}_N(\psi_N)\\
		=& \int_{\Omega_N}\left(\frac{1}{2}|\nabla \psi_N (\nabla \tau_\vartheta)^{-1}|^2+F(\psi_N) \right)\det (\nabla \tau_\vartheta)  -\int_{\Omega_N}\left(\frac{1}{2}|\nabla \psi_N|^2+F(\psi_N)\right)\\
		= &\ \vartheta\int_{\Omega_N}\left(\frac{|\nabla\psi_N|^2}{2}+F(\psi_N)\right)\nabla\cdot \eta -\vartheta\int_{\Omega_N} \nabla \psi_N \nabla \eta \nabla  \psi_N +o(\vartheta).
	\end{aligned}
\end{equation}
Dividing $\vartheta$ on both sides of \eqref{est_m1} and passing to the limit $\vartheta\rightarrow 0$ yield
\begin{align}\label{eq_51.5}
	\int_{\Omega_N}\left(\frac{|\nabla\psi_N|^2}{2}+F(\psi_N)\right)\nabla\cdot \eta  -\int_{\Omega_N} \nabla \psi_N \nabla \eta \nabla  \psi_N =0.
\end{align}
Note that
\begin{equation*}
	\begin{aligned}
		&\left(\frac{|\nabla\psi_N|^2}{2}+F(\psi_N)\right)\nabla\cdot \eta
		= \nabla\cdot \left[\left(\frac{|\nabla\psi_N|^2}{2}+F(\psi_N)\right)\eta\right]- \nabla\psi_N D^2\psi_N\eta -f(\psi_N)\nabla \psi_N\cdot\eta
	\end{aligned}
\end{equation*}
and
\[
\nabla \psi_N \nabla\eta\nabla\psi_N+ \nabla\psi_N D^2\psi_N \eta= \nabla(\eta\cdot \nabla\psi_N)\cdot \nabla\psi_N.
\]
Using the divergence theorem for \eqref{eq_51.5} we have
\begin{align*}
	0&=\lim_{s\searrow 0}\int_{\Omega_N\cap\{s<\psi_N<Q-s\}}\left(\frac{|\nabla\psi_N|^2}{2}+F(\psi_N)\right)\nabla\cdot \eta - \nabla \psi_N \nabla \eta \nabla  \psi_N \\
	&=\lim_{s\searrow 0}\int_{\p\{s<\psi_N<Q-s\}} \left[\frac{|\nabla\psi_N|^2}{2}+F(\psi_N)\right](\eta\cdot \nu)-(\eta\cdot\nabla \psi_N)\frac{\partial\psi_N}{\partial \nu}\\
	&\quad +\int_{\Omega_N\cap\{s<\psi_N<Q-s\}}(\Delta\psi_N  -f(\psi_N))\nabla \psi_N\cdot\eta \\
	&=\lim_{s\searrow 0}\int_{\p(\Omega_N\cap \{s<\psi_N<Q-s\})} \left[-\frac{|\nabla\psi_N|^2}{2}+F(\psi_N)\right](\eta\cdot \nu),
\end{align*}
where we have used the equation of $\psi_N$ in  the open set $\{x: 0<\psi_N(x)<Q\}$ to get the last equality. This finishes the proof of the lemma.
\end{proof}
		
		\subsection{Monotonicity and uniqueness}
		In this subsection, we prove two important properties of energy minimizers: monotonicity in the vertical direction and uniqueness of the energy minimizer.
		The proof for the monotonicity of the energy minimizer is based on the shifting method and comparison for the associated energy of the shifted functions.
  
		\begin{lemma}\label{mon}
			If ${\psi}_N$ is an energy minimizer for ${\mathcal{E}}_N$ over ${\mathcal{K}}_N$, i.e.,  ${\mathcal{E}}_N({\psi}_N)=\inf_{\phi\in {\mathcal{K}}_N}{\mathcal{E}}_N(\phi)$,  then $\psi_N$ satisfies
	\[
	\partial_{x_2}{\psi}_N\geq 0\quad \text{in}\,\, {\Omega}_{N}.
	\]
		\end{lemma}
		\begin{proof}
			Extend $\psi_N$ on the whole plane as follows
			\begin{equation}
				\tilde\psi_N(x_1,x_2)=\left\{
				\begin{aligned}
					& Q\quad &&\text{if}\,\,x_2\geq h_1(x_1),\\
					&\psi_N(x_1,x_2)\quad &&\text{if}\,\, (x_1,x_2)\in \Omega_N,\\
					&0\quad &&\text{if}\,\,x_2 \leq h_0(x_1).
				\end{aligned}\right.
			\end{equation}
			For ease of notations, we still denote $\tilde{\psi}_N$ by $\psi_N$.
For any $\tau\geq 0$, denote
\begin{eqnarray*}
{\psi}_N^{\tau}(x)={\psi}_N(x_1,x_2+\tau)\ \ \text{and} \ \ {\Omega}_N^{\tau}=\{(x_1,x_2): (x_1,x_2+\tau)\in{\Omega}_N\}.
\end{eqnarray*}
Let ${D}^{\tau}={\Omega}_N\cap {\Omega}_N^{\tau}$ (cf. Figure \ref{f5.1}) and
		\begin{eqnarray*}
		{\Psi}^{\tau}(x)={\psi}_N^{\tau}(x)-\psi_N({x})\ \ \mbox{for any}  \ x\in {{D}^{\tau}}.
	\end{eqnarray*}

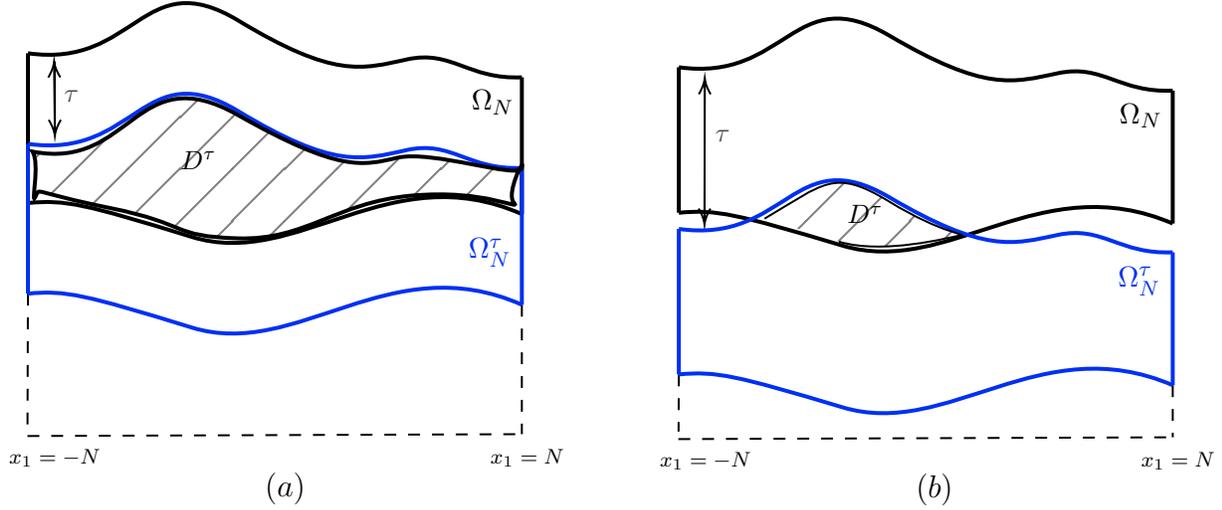
\begin{figure}[h]
				\centering
				\begin{minipage}[t]{.45\textwidth}



 
\tikzset{
pattern size/.store in=\mcSize, 
pattern size = 5pt,
pattern thickness/.store in=\mcThickness, 
pattern thickness = 0.3pt,
pattern radius/.store in=\mcRadius, 
pattern radius = 1pt}
\makeatletter
\pgfutil@ifundefined{pgf@pattern@name@_kvqonljwz}{
\pgfdeclarepatternformonly[\mcThickness,\mcSize]{_kvqonljwz}
{\pgfqpoint{0pt}{0pt}}
{\pgfpoint{\mcSize+\mcThickness}{\mcSize+\mcThickness}}
{\pgfpoint{\mcSize}{\mcSize}}
{
\pgfsetcolor{\tikz@pattern@color}
\pgfsetlinewidth{\mcThickness}
\pgfpathmoveto{\pgfqpoint{0pt}{0pt}}
\pgfpathlineto{\pgfpoint{\mcSize+\mcThickness}{\mcSize+\mcThickness}}
\pgfusepath{stroke}
}}
\makeatother
\tikzset{every picture/.style={line width=0.75pt}} 

\begin{tikzpicture}[x=0.55pt,y=0.7pt,yscale=-1,xscale=1]

\draw [line width=1.5]    (157,51) .. controls (231.67,59.33) and (235.67,2) .. (301.67,33.33) .. controls (367.67,64.67) and (386.67,59.67) .. (417.67,54.33) .. controls (448.67,49) and (459.67,65.33) .. (496.67,64) ;
\draw [line width=1.5]    (157,51) -- (157,132) ;
\draw [line width=1.5]    (496.67,64) -- (496.67,138) ;
\draw [line width=1.5]    (157,132) .. controls (182.67,130) and (196.67,132.67) .. (269.67,150.33) .. controls (342.67,168) and (405.67,106.33) .. (496.67,138) ;
\draw  [dash pattern={on 4.5pt off 4.5pt}]  (156.67,258) -- (496.67,257) ;
\draw [line width=0.75]  [dash pattern={on 4.5pt off 4.5pt}]  (157,181) -- (156.67,258) ;
\draw [line width=0.75]  [dash pattern={on 4.5pt off 4.5pt}]  (496.67,187) -- (496.67,257) ;
\draw [color={rgb, 255:red, 0; green, 0; blue, 0 }  ,draw opacity=1 ]   (175.4,55.5) -- (175.49,93.6) ;
\draw [shift={(175.5,95.6)}, rotate = 269.86] [color={rgb, 255:red, 0; green, 0; blue, 0 }  ,draw opacity=1 ][line width=0.75]    (10.93,-3.29) .. controls (6.95,-1.4) and (3.31,-0.3) .. (0,0) .. controls (3.31,0.3) and (6.95,1.4) .. (10.93,3.29)   ;
\draw [color={rgb, 255:red, 0; green, 0; blue, 0 }  ,draw opacity=1 ]   (175.5,95.6) -- (175.02,57.5) ;
\draw [shift={(175,55.5)}, rotate = 89.29] [color={rgb, 255:red, 0; green, 0; blue, 0 }  ,draw opacity=1 ][line width=0.75]    (10.93,-3.29) .. controls (6.95,-1.4) and (3.31,-0.3) .. (0,0) .. controls (3.31,0.3) and (6.95,1.4) .. (10.93,3.29)   ;
\draw [color={rgb, 255:red, 0; green, 48; blue, 243 }  ,draw opacity=1 ][line width=1.5]    (157,100) -- (157,181) ;
\draw [color={rgb, 255:red, 0; green, 48; blue, 243 }  ,draw opacity=1 ][line width=1.5]    (496.67,113) -- (496.67,187) ;
\draw [color={rgb, 255:red, 0; green, 48; blue, 243 }  ,draw opacity=1 ][line width=1.5]    (157,181) .. controls (182.67,179) and (196.67,181.67) .. (269.67,199.33) .. controls (342.67,217) and (405.67,155.33) .. (496.67,187) ;
\draw [color={rgb, 255:red, 0; green, 48; blue, 243 }  ,draw opacity=1 ][line width=1.5]    (157,100) .. controls (231.67,108.33) and (235.67,51) .. (301.67,82.33) .. controls (367.67,113.67) and (386.67,108.67) .. (417.67,103.33) .. controls (448.67,98) and (459.67,114.33) .. (496.67,113) ;
\draw [draw opacity=0][pattern=_kvqonljwz,pattern size=19.725pt,pattern thickness=0.75pt,pattern radius=0pt, pattern color={rgb, 255:red, 128; green, 128; blue, 128}][line width=1.5]    (160.08,104.17) .. controls (228.58,114.67) and (230.67,53) .. (301.13,84.3) .. controls (371.6,115.6) and (387.18,111.83) .. (411.13,107.5) .. controls (435.08,103.17) and (494.75,119.33) .. (496.67,113) .. controls (498.58,106.67) and (485.73,126.6) .. (491.08,131.67) .. controls (496.43,136.73) and (472.75,124.33) .. (420.67,128) .. controls (368.58,131.67) and (348.4,152.2) .. (301.53,151.1) .. controls (254.67,150) and (268.73,143.5) .. (213.53,134.7) .. controls (158.33,125.9) and (165.58,123.17) .. (162.08,128.67) .. controls (158.58,134.17) and (165.58,110.67) .. (160.08,104.17) -- cycle ;

\draw (142,263.4) node [anchor=north west][inner sep=0.75pt]  [font=\tiny]  {$x_{1} =-N$};
\draw (473,263.4) node [anchor=north west][inner sep=0.75pt]  [font=\tiny]  {$x_{1} =N$};
\draw (180,69.4) node [anchor=north west][inner sep=0.75pt]  [font=\footnotesize,color={rgb, 255:red, 74; green, 74; blue, 74 }  ,opacity=1 ]  {$\tau $};
\draw (460,69.4) node [anchor=north west][inner sep=0.75pt]  [font=\small]  {$\Omega _{N}$};
\draw (458,149.4) node [anchor=north west][inner sep=0.75pt]  [font=\small,color={rgb, 255:red, 0; green, 48; blue, 243 }  ,opacity=1 ]  {$\Omega _{N}^{\tau }$};
\draw (260,102.4) node [anchor=north west][inner sep=0.75pt]  [font=\footnotesize]  {$D^{\tau }$};
\draw (318,276.4) node [anchor=north west][inner sep=0.75pt]    {$( a)$};

\end{tikzpicture}

				\end{minipage}%
				\begin{minipage}[t]{.6\textwidth}
					\centering



 
\tikzset{
pattern size/.store in=\mcSize, 
pattern size = 5pt,
pattern thickness/.store in=\mcThickness, 
pattern thickness = 0.3pt,
pattern radius/.store in=\mcRadius, 
pattern radius = 1pt}
\makeatletter
\pgfutil@ifundefined{pgf@pattern@name@_sq2y31xcg}{
\pgfdeclarepatternformonly[\mcThickness,\mcSize]{_sq2y31xcg}
{\pgfqpoint{0pt}{0pt}}
{\pgfpoint{\mcSize+\mcThickness}{\mcSize+\mcThickness}}
{\pgfpoint{\mcSize}{\mcSize}}
{
\pgfsetcolor{\tikz@pattern@color}
\pgfsetlinewidth{\mcThickness}
\pgfpathmoveto{\pgfqpoint{0pt}{0pt}}
\pgfpathlineto{\pgfpoint{\mcSize+\mcThickness}{\mcSize+\mcThickness}}
\pgfusepath{stroke}
}}
\makeatother
\tikzset{every picture/.style={line width=0.75pt}} 

\begin{tikzpicture}[x=0.55pt,y=0.68pt,yscale=-1,xscale=1]

\draw [line width=1.5]    (157,44) .. controls (231.67,52.33) and (235.67,-5) .. (301.67,26.33) .. controls (367.67,57.67) and (386.67,52.67) .. (417.67,47.33) .. controls (448.67,42) and (459.67,58.33) .. (496.67,57) ;
\draw [line width=1.5]    (157,44) -- (157,125) ;
\draw [line width=1.5]    (496.67,57) -- (496.67,131) ;
\draw [line width=1.5]    (157,125) .. controls (182.67,123) and (196.67,125.67) .. (269.67,143.33) .. controls (342.67,161) and (405.67,99.33) .. (496.67,131) ;
\draw  [dash pattern={on 4.5pt off 4.5pt}]  (156.67,251) -- (496.67,250) ;
\draw [line width=0.75]  [dash pattern={on 4.5pt off 4.5pt}]  (157,215) -- (156.83,246.83) ;
\draw [line width=0.75]  [dash pattern={on 4.5pt off 4.5pt}]  (496.67,221) -- (496.67,250) ;
\draw [color={rgb, 255:red, 0; green, 0; blue, 0 }  ,draw opacity=1 ]   (174.83,50.18) -- (174.67,129.5) ;
\draw [shift={(174.67,131.5)}, rotate = 270.12] [color={rgb, 255:red, 0; green, 0; blue, 0 }  ,draw opacity=1 ][line width=0.75]    (10.93,-3.29) .. controls (6.95,-1.4) and (3.31,-0.3) .. (0,0) .. controls (3.31,0.3) and (6.95,1.4) .. (10.93,3.29)   ;
\draw [color={rgb, 255:red, 0; green, 0; blue, 0 }  ,draw opacity=1 ]   (174.67,131.5) -- (174.83,52.18) ;
\draw [shift={(174.83,50.18)}, rotate = 90.12] [color={rgb, 255:red, 0; green, 0; blue, 0 }  ,draw opacity=1 ][line width=0.75]    (10.93,-3.29) .. controls (6.95,-1.4) and (3.31,-0.3) .. (0,0) .. controls (3.31,0.3) and (6.95,1.4) .. (10.93,3.29)   ;
\draw [color={rgb, 255:red, 0; green, 48; blue, 243 }  ,draw opacity=1 ][line width=1.5]    (157,134) -- (157,215) ;
\draw [color={rgb, 255:red, 0; green, 48; blue, 243 }  ,draw opacity=1 ][line width=1.5]    (496.67,147) -- (496.67,221) ;
\draw [color={rgb, 255:red, 0; green, 48; blue, 243 }  ,draw opacity=1 ][line width=1.5]    (157,215) .. controls (182.67,213) and (196.67,215.67) .. (269.67,233.33) .. controls (342.67,251) and (405.67,189.33) .. (496.67,221) ;
\draw [color={rgb, 255:red, 0; green, 48; blue, 243 }  ,draw opacity=1 ][line width=1.5]    (157,134) .. controls (231.67,142.33) and (235.67,85) .. (301.67,116.33) .. controls (367.67,147.67) and (386.67,142.67) .. (417.67,137.33) .. controls (448.67,132) and (459.67,148.33) .. (496.67,147) ;
\draw [draw opacity=0][pattern=_sq2y31xcg,pattern size=19.725pt,pattern thickness=0.75pt,pattern radius=0pt, pattern color={rgb, 255:red, 128; green, 128; blue, 128}]   (215.93,128.3) .. controls (257.13,108.7) and (263.53,99.9) .. (301.53,118.7) .. controls (339.53,137.5) and (356.5,136.72) .. (349.83,137.72) .. controls (343.17,138.72) and (297.7,148.48) .. (267.13,141.1) ;

\draw (142,256.4) node [anchor=north west][inner sep=0.75pt]  [font=\tiny]  {$x_{1} =-N$};
\draw (473,256.4) node [anchor=north west][inner sep=0.75pt]  [font=\tiny]  {$x_{1} =N$};
\draw (180,77.4) node [anchor=north west][inner sep=0.75pt]  [font=\footnotesize,color={rgb, 255:red, 74; green, 74; blue, 74 }  ,opacity=1 ]  {$\tau $};
\draw (458,62.4) node [anchor=north west][inner sep=0.75pt]  [font=\small]  {$\Omega _{N}$};
\draw (457,153.4) node [anchor=north west][inner sep=0.75pt]  [font=\small,color={rgb, 255:red, 0; green, 48; blue, 243 }  ,opacity=1 ]  {$\Omega _{N}^{\tau }$};
\draw (270.8,119.2) node [anchor=north west][inner sep=0.75pt]  [font=\footnotesize]  {$D^{\tau }$};
\draw (318,268.4) node [anchor=north west][inner sep=0.75pt]    {$( b)$};

\end{tikzpicture}

				\end{minipage}
    \caption{The intersection of $\Omega_N$ and $\Omega_N^{\tau}$}\label{f5.1}
			\end{figure}
To prove the lemma, it suffices to show that for any $\tau>0$, one has
\begin{eqnarray*}
		{\Psi}^{\tau}(x)\geq 0\ \ \mbox{for any}  \ x\in {{D}^{\tau}}.
	\end{eqnarray*}
We prove this property by the contradiction argument.
Suppose that there exists $\bar\tau>0$ such that the set
\begin{eqnarray}\label{m2}
U=\{x\in {D}^{\bar\tau}: {\Psi}^{\bar\tau}(x)< 0 \}
\end{eqnarray}
is not empty. Since $ {\Psi}^{\bar\tau}$ is a continuous function, $U$ must be an open set. Hence one has either $U={D}^{\bar\tau}$ or $U\subsetneq
			{D}^{\bar\tau}$. In the following, we show that neither of these possibilities happens.

	(i) If $U={D}^{\bar\tau}$. Denote
\begin{eqnarray*}
				\hat{x}_1=\text{min}\{x_1: (x_1,x_2)\in \partial D^{\bar\tau}\}.
			\end{eqnarray*}
One can see that $\psi(\hat{x})=0$ and $0<\psi_N^{\bar\tau}(\hat x)\leq Q$  for $\hat{x}=(\hat{x}_1, \hat{x}_2)\in \partial D^{\bar\tau}\cap \Gamma_{0}$. Indeed, if $\hat{x}_1=-N$, then $\psi_N^{\bar\tau}(\hat x)\in (0, Q)$ as described in Figure \ref{f5.1}(a), and if $\hat{x}_1>-N$, then $\psi_N^{\bar\tau}(\hat x)=Q$ as described in Figure \ref{f5.1}(b). Therefore, one has
			\begin{eqnarray}
				{\Psi}^{\bar\tau}(\hat{x})> 0.
			\end{eqnarray}
			On the other hand, it follows from continuity that
			\begin{eqnarray}
				{\Psi}^{\bar\tau}(x)\leq 0\ \  \ \text{for any}\,\, x\in \overline{{D}^{\bar\tau}}.
			\end{eqnarray}
			This leads to a contradiction.

	(ii) Suppose that $U\subsetneq{D}^{\bar\tau}$, let $V$ be any connected component of $U$. Clearly, for $x\in \partial V\cap D^{\bar\tau}$, one has $\Psi^{\bar\tau}(x)=0$. Furthermore, for $x\in \overline{D^{\bar\tau}}\cap \{(x_1, x_2): x_1=-N \,\, \text{or}\,\, N\}$, one has $\Psi^{\bar\tau}(x)>0$. Finally, for any $x\in \partial D^{\bar\tau}\cap \{(x_1, x_2): x_2=h_1(x_1)-\bar{\tau}\}$, one has
			\[
			\Psi^{\bar\tau}=\psi_N^{\bar\tau}(x)-\psi_N(x)=Q-\psi_N(x)\geq 0.
			\]
			Similarly, for any $x\in \partial D^{\bar\tau}\cap \{(x_1,x_2): x_2=h_0(x_1)\}$, one has
			\[
			\Psi^{\bar\tau}=\psi_N^{\bar\tau}(x)-\psi_N(x)=\psi_N^{\bar\tau}(x)-0\geq 0.
			\]
			This implies that
			$$ {\Psi}^{\bar\tau}(x)=0 \ \ \ \ \text{on} \ \ \partial V$$
			and
			$$\partial V \cap \{(x_1,x_2)\in \overline{D^{\bar\tau}}: x_1=-N \,\, \text{or}\,\, N\}=\emptyset. $$
			Furthermore, we analyze the problem via the following three subcases.

\ \ {\it  Case (ii-1).}  Assume that
	\begin{eqnarray}\label{m5}
	\int_{V}\frac{|\nabla {\psi}_N^{{\bar\tau}}|^2}{2}+F({\psi}_N^{{\bar\tau}})dx < \int_{V}\frac{|\nabla \psi_N|^2}{2}+F(\psi_N)dx.	\end{eqnarray}
 
	Define
	\[
	{\psi}_*=\left\{
	\begin{aligned}
		& {\psi}_N^{{\bar\tau}},\quad \text{in}\,\, V,\\
		& \psi_N,\quad \text{in}\,\, {\Omega}_{N}\setminus V.
	\end{aligned}\right.
	\]
	Clearly, $\psi_*\in {\mathcal{K}}_N$ and $\psi_*$ satisfies
	\[
	\int_{{\Omega}_{N}}\frac{|\nabla {\psi}_*|^2}{2}+F({\psi}_*)dx < \int_{{\Omega}_{N}}\frac{|\nabla {\psi}_N|^2}{2}+F({\psi}_N)dx.	
	\]
	This contradicts with the fact that ${\psi}_N$ is an energy minimizer for ${\mathcal{E}}_N$ over ${\mathcal{K}}_N$.

\ \ {\it  Case (ii-2).} If
	\begin{eqnarray}\label{m6}
	\int_{V}\frac{|\nabla \psi_N|^2}{2}+F(\psi_N)dx < \int_{V}\frac{|\nabla {\psi}_N^{{\bar\tau}}|^2}{2}+F({\psi}_N^{{\bar\tau}})dx.	
\end{eqnarray}
Let $g_N^{\tau}(x_1,x_2)= g_N(x_1,x_2+\tau)$. Clearly, one has $g_N^{\tau}\in H^1({\Omega}_N^{\tau})$. Let $\mathcal{K}_N^{\tau} $ and $\mathcal{E}_N^{\tau}$ be the associated admissible set and energy functional on ${\Omega}_N^{\tau}$, respectively. More precisely, define
			\begin{equation}\label{E1}
			\mathcal{K}_N^{\tau} =\{\psi:  \psi\in H^1({\Omega}_N^{\tau}), \psi=g_N^{\tau} \,\,\text{on}\,\, \partial {\Omega}_N^{\tau}\}
			\end{equation}
			and
			\begin{equation}\label{E2}
			\mathcal{E}_N^{\tau}(\psi)= \int_{{\Omega}_N^{\tau}} \frac{|\nabla \psi|^2}{2} +F(\psi)dx.
						\end{equation}
			It is easy to see that ${\psi}_N^{\bar\tau}$ is an energy minimizer for ${\mathcal{E}}_N^{\bar\tau}$ over ${\mathcal{K}}_N^{\bar\tau}$.
			
			Define
			\[
			{\psi}_{**}=\left\{
			\begin{aligned}
				& {\psi}_N,\quad &&\text{in}\,\, V,\\
				&  {\psi}_N^{{\bar\tau}},\quad &&\text{in}\,\, {\Omega}_{N}^{\bar\tau}\setminus V.
			\end{aligned}\right.
			\]
			Clearly, $\psi_{**}\in {\mathcal{K}}_N^{\bar\tau}$ and $\psi_{**}$ satisfies
			\[
			\int_{{\Omega}_{N}^{\bar\tau}}\frac{|\nabla {\psi}_{**}|^2}{2}+F({\psi}_{**})dx < \int_{{\Omega}_{N}^{\bar\tau}}\frac{|\nabla {\psi}
_N^{{\bar\tau}}|^2}{2}+F({\psi}_N^{{\bar\tau}})dx.	
			\]
			{This contradicts with the fact that ${\psi}_N^{\bar\tau}$ is an energy minimizer for ${\mathcal{E}}_N^{\bar\tau}$ over ${\mathcal{K}}_N^{\bar\tau}$.}
	Hence \eqref{m6} cannot happen either.

\ \ {\it  Case (ii-3).} We assume that
\begin{eqnarray}\label{m4}
	\int_{V}\frac{|\nabla {\psi}_N^{{\bar\tau}}|^2}{2}+F({\psi}_N^{{\bar\tau}})dx = \int_{V}\frac{|\nabla \psi_N|^2}{2}+F(\psi_N)dx.
\end{eqnarray}
One can see that the complement of $V$ in $\mathbb{R}^2$ has a unique unbounded connected component $A$. Let $\Sigma=\partial A\cap D^{\bar{\tau}}$ and $\mathcal{O}=\mathbb{R}^2\setminus A$ (cf. Figure \ref{f5.2}).
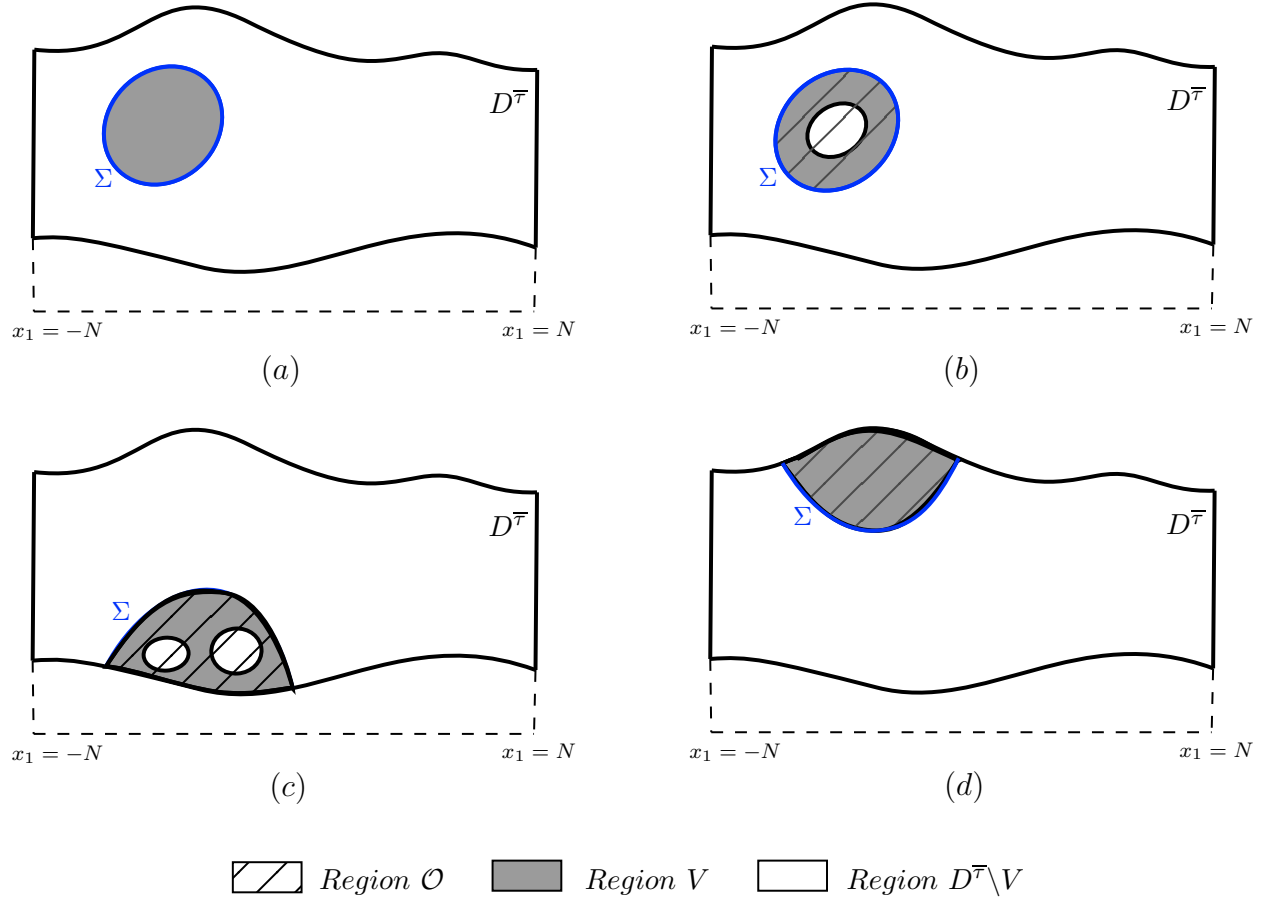
\begin{figure}[h]
				\centering
				\begin{minipage}[t]{.45\textwidth}

 
\tikzset{
pattern size/.store in=\mcSize, 
pattern size = 5pt,
pattern thickness/.store in=\mcThickness, 
pattern thickness = 0.3pt,
pattern radius/.store in=\mcRadius, 
pattern radius = 1pt}
\makeatletter
\pgfutil@ifundefined{pgf@pattern@name@_f8ca7rewt}{
\pgfdeclarepatternformonly[\mcThickness,\mcSize]{_f8ca7rewt}
{\pgfqpoint{0pt}{0pt}}
{\pgfpoint{\mcSize+\mcThickness}{\mcSize+\mcThickness}}
{\pgfpoint{\mcSize}{\mcSize}}
{
\pgfsetcolor{\tikz@pattern@color}
\pgfsetlinewidth{\mcThickness}
\pgfpathmoveto{\pgfqpoint{0pt}{0pt}}
\pgfpathlineto{\pgfpoint{\mcSize+\mcThickness}{\mcSize+\mcThickness}}
\pgfusepath{stroke}
}}
\makeatother
\tikzset{every picture/.style={line width=0.75pt}} 

\begin{tikzpicture}[x=0.56pt,y=0.59pt,yscale=-1,xscale=1]

\draw [line width=1.5]    (161,53) .. controls (235.67,61.33) and (239.67,4) .. (305.67,35.33) .. controls (371.67,66.67) and (390.67,61.67) .. (421.67,56.33) .. controls (452.67,51) and (463.67,67.33) .. (500.67,66) ;
\draw [line width=1.5]    (160,174) .. controls (185.67,172) and (199.67,174.67) .. (272.67,192.33) .. controls (345.67,210) and (408.67,148.33) .. (499.67,180) ;
\draw [line width=1.5]    (161,53) -- (160,174) ;
\draw [line width=1.5]    (500.67,66) -- (499.67,180) ;
\draw  [dash pattern={on 4.5pt off 4.5pt}]  (160.67,221) -- (498.67,221) ;
\draw [line width=0.75]  [dash pattern={on 4.5pt off 4.5pt}]  (160,174) -- (160.67,221) ;
\draw [line width=0.75]  [dash pattern={on 4.5pt off 4.5pt}]  (499.67,180) -- (498.67,221) ;
\draw  [color={rgb, 255:red, 0; green, 50; blue, 250 }  ,draw opacity=0.71 ][pattern=_f8ca7rewt,pattern size=15.524999999999999pt,pattern thickness=0.75pt,pattern radius=0pt, pattern color={rgb, 255:red, 0; green, 0; blue, 0}][line width=1.5]  (213.64,125.54) .. controls (202.3,109.3) and (208.44,85.42) .. (227.35,72.21) .. controls (246.27,59.01) and (270.8,61.47) .. (282.14,77.72) .. controls (293.48,93.96) and (287.34,117.84) .. (268.43,131.05) .. controls (249.51,144.25) and (224.98,141.79) .. (213.64,125.54) -- cycle ;
\draw  [color={rgb, 255:red, 0; green, 50; blue, 250 }  ,draw opacity=1 ][fill={rgb, 255:red, 155; green, 155; blue, 155 }  ,fill opacity=0.5 ][line width=1.5]  (213.69,125.61) .. controls (202.37,109.4) and (208.53,85.56) .. (227.45,72.35) .. controls (246.37,59.15) and (270.88,61.58) .. (282.19,77.79) .. controls (293.51,94) and (287.34,117.84) .. (268.43,131.05) .. controls (249.51,144.25) and (225,141.82) .. (213.69,125.61) -- cycle ;

\draw (144,226.4) node [anchor=north west][inner sep=0.75pt]  [font=\tiny]  {$x_{1} =-N$};
\draw (475,225.4) node [anchor=north west][inner sep=0.75pt]  [font=\tiny]  {$x_{1} =N$};
\draw (466,76.4) node [anchor=north west][inner sep=0.75pt]  [font=\small]  {$D^{\overline{\tau }}$};
\draw (199.19,127.83) node [anchor=north west][inner sep=0.75pt]  [font=\footnotesize,color={rgb, 255:red, 0; green, 50; blue, 250 }  ,opacity=1 ]  {$\Sigma $};
\draw (312,247.4) node [anchor=north west][inner sep=0.75pt]    {$( a)$};

\end{tikzpicture}

				\end{minipage}%
				\begin{minipage}[t]{.65\textwidth}
					\centering

 
\tikzset{
pattern size/.store in=\mcSize, 
pattern size = 5pt,
pattern thickness/.store in=\mcThickness, 
pattern thickness = 0.3pt,
pattern radius/.store in=\mcRadius, 
pattern radius = 1pt}
\makeatletter
\pgfutil@ifundefined{pgf@pattern@name@_0wghgboo5}{
\pgfdeclarepatternformonly[\mcThickness,\mcSize]{_0wghgboo5}
{\pgfqpoint{0pt}{0pt}}
{\pgfpoint{\mcSize+\mcThickness}{\mcSize+\mcThickness}}
{\pgfpoint{\mcSize}{\mcSize}}
{
\pgfsetcolor{\tikz@pattern@color}
\pgfsetlinewidth{\mcThickness}
\pgfpathmoveto{\pgfqpoint{0pt}{0pt}}
\pgfpathlineto{\pgfpoint{\mcSize+\mcThickness}{\mcSize+\mcThickness}}
\pgfusepath{stroke}
}}
\makeatother
\tikzset{every picture/.style={line width=0.75pt}} 

\begin{tikzpicture}[x=0.56pt,y=0.59pt,yscale=-1,xscale=1]

\draw [line width=1.5]    (161,58) .. controls (235.67,66.33) and (239.67,9) .. (305.67,40.33) .. controls (371.67,71.67) and (390.67,66.67) .. (421.67,61.33) .. controls (452.67,56) and (463.67,72.33) .. (500.67,71) ;
\draw [line width=1.5]    (160,179) .. controls (185.67,177) and (199.67,179.67) .. (272.67,197.33) .. controls (345.67,215) and (408.67,153.33) .. (499.67,185) ;
\draw [line width=1.5]    (161,58) -- (160,179) ;
\draw [line width=1.5]    (500.67,71) -- (499.67,185) ;
\draw  [dash pattern={on 4.5pt off 4.5pt}]  (160.67,226) -- (498.67,226) ;
\draw [line width=0.75]  [dash pattern={on 4.5pt off 4.5pt}]  (160,179) -- (160.67,226) ;
\draw [line width=0.75]  [dash pattern={on 4.5pt off 4.5pt}]  (499.67,185) -- (498.67,226) ;
\draw  [color={rgb, 255:red, 0; green, 50; blue, 250 }  ,draw opacity=0.71 ][fill={rgb, 255:red, 155; green, 155; blue, 155 }  ,fill opacity=0.5 ][line width=1.5]  (209.4,136.7) .. controls (197.99,120.36) and (204.82,95.89) .. (224.66,82.04) .. controls (244.5,68.19) and (269.83,70.21) .. (281.24,86.55) .. controls (292.64,102.89) and (285.81,127.36) .. (265.97,141.21) .. controls (246.13,155.06) and (220.8,153.04) .. (209.4,136.7) -- cycle ;
\draw  [fill={rgb, 255:red, 255; green, 255; blue, 255 }  ,fill opacity=1 ][line width=1.5]  (227,121.71) .. controls (222.76,114.01) and (227.52,103.25) .. (237.64,97.68) .. controls (247.75,92.11) and (259.39,93.84) .. (263.63,101.53) .. controls (267.87,109.23) and (263.11,119.99) .. (252.99,125.56) .. controls (242.88,131.13) and (231.24,129.41) .. (227,121.71) -- cycle ;
\draw  [color={rgb, 255:red, 0; green, 50; blue, 250 }  ,draw opacity=0.71 ][pattern=_0wghgboo5,pattern size=15.524999999999999pt,pattern thickness=0.75pt,pattern radius=0pt, pattern color={rgb, 255:red, 74; green, 74; blue, 74}][line width=1.5]  (209.4,136.7) .. controls (197.99,120.36) and (204.82,95.89) .. (224.66,82.04) .. controls (244.5,68.19) and (269.83,70.21) .. (281.24,86.55) .. controls (292.64,102.89) and (285.81,127.36) .. (265.97,141.21) .. controls (246.13,155.06) and (220.8,153.04) .. (209.4,136.7) -- cycle ;

\draw (144,231.4) node [anchor=north west][inner sep=0.75pt]  [font=\tiny]  {$x_{1} =-N$};
\draw (475,230.4) node [anchor=north west][inner sep=0.75pt]  [font=\tiny]  {$x_{1} =N$};
\draw (466,81.4) node [anchor=north west][inner sep=0.75pt]  [font=\small]  {$D^{\overline{\tau }}$};
\draw (190.33,133.7) node [anchor=north west][inner sep=0.75pt]  [font=\footnotesize,color={rgb, 255:red, 0; green, 50; blue, 250 }  ,opacity=1 ]  {$\Sigma $};
\draw (315,254.4) node [anchor=north west][inner sep=0.75pt]    {$( b)$};

\end{tikzpicture}

				\end{minipage}
			
				\begin{minipage}[t]{.5\textwidth}


 
\tikzset{
pattern size/.store in=\mcSize, 
pattern size = 5pt,
pattern thickness/.store in=\mcThickness, 
pattern thickness = 0.3pt,
pattern radius/.store in=\mcRadius, 
pattern radius = 1pt}
\makeatletter
\pgfutil@ifundefined{pgf@pattern@name@_gf5qcgh66}{
\pgfdeclarepatternformonly[\mcThickness,\mcSize]{_gf5qcgh66}
{\pgfqpoint{0pt}{0pt}}
{\pgfpoint{\mcSize+\mcThickness}{\mcSize+\mcThickness}}
{\pgfpoint{\mcSize}{\mcSize}}
{
\pgfsetcolor{\tikz@pattern@color}
\pgfsetlinewidth{\mcThickness}
\pgfpathmoveto{\pgfqpoint{0pt}{0pt}}
\pgfpathlineto{\pgfpoint{\mcSize+\mcThickness}{\mcSize+\mcThickness}}
\pgfusepath{stroke}
}}
\makeatother
\tikzset{every picture/.style={line width=0.75pt}} 

\begin{tikzpicture}[x=0.56pt,y=0.59pt,yscale=-1,xscale=1]

\draw [line width=1.5]    (163,57) .. controls (237.67,65.33) and (241.67,8) .. (307.67,39.33) .. controls (373.67,70.67) and (392.67,65.67) .. (423.67,60.33) .. controls (454.67,55) and (465.67,71.33) .. (502.67,70) ;
\draw [line width=1.5]    (162,178) .. controls (177.55,176.79) and (188.81,177.29) .. (211.51,181.8) .. controls (219.81,183.45) and (229.65,185.64) .. (241.79,188.48) .. controls (251.23,190.68) and (262.07,193.29) .. (274.67,196.33) .. controls (296.19,201.54) and (316.85,199.85) .. (337.61,195.7) .. controls (387.25,185.75) and (437.5,161.67) .. (501.67,184) ;
\draw [line width=1.5]    (163,57) -- (162,178) ;
\draw [line width=1.5]    (502.67,70) -- (501.67,184) ;
\draw  [dash pattern={on 4.5pt off 4.5pt}]  (162.67,225) -- (500.67,225) ;
\draw [line width=0.75]  [dash pattern={on 4.5pt off 4.5pt}]  (162,178) -- (162.67,225) ;
\draw [line width=0.75]  [dash pattern={on 4.5pt off 4.5pt}]  (501.67,184) -- (500.67,225) ;
\draw [color={rgb, 255:red, 0; green, 50; blue, 250 }  ,draw opacity=1 ][line width=1.5]    (211.51,181.8) .. controls (242.83,125.5) and (313.67,102) .. (337.61,195.7) ;
\draw [draw opacity=0][fill={rgb, 255:red, 155; green, 155; blue, 155 }  ,fill opacity=0.61 ][line width=1.5]    (211.51,181.8) .. controls (238.17,142.72) and (260.05,128.06) .. (291.77,134.56) .. controls (323.5,141.06) and (338.38,197.34) .. (337.61,195.7) .. controls (336.83,194.06) and (312.17,201.39) .. (286.33,197.9) .. controls (260.5,194.41) and (225.5,181.72) .. (211.51,181.8) -- cycle ;
\draw  [fill={rgb, 255:red, 255; green, 255; blue, 255 }  ,fill opacity=1 ][line width=1.5]  (236.98,174.81) .. controls (236.6,169) and (243.1,163.85) .. (251.48,163.31) .. controls (259.87,162.77) and (266.97,167.05) .. (267.34,172.86) .. controls (267.72,178.68) and (261.22,183.83) .. (252.83,184.37) .. controls (244.45,184.9) and (237.35,180.63) .. (236.98,174.81) -- cycle ;
\draw  [fill={rgb, 255:red, 255; green, 255; blue, 255 }  ,fill opacity=1 ][line width=1.5]  (282.73,172.92) .. controls (282.22,164.96) and (289.38,158.02) .. (298.73,157.42) .. controls (308.09,156.82) and (316.08,162.79) .. (316.59,170.75) .. controls (317.1,178.71) and (309.94,185.65) .. (300.58,186.25) .. controls (291.23,186.85) and (283.24,180.89) .. (282.73,172.92) -- cycle ;
\draw [draw opacity=0][pattern=_gf5qcgh66,pattern size=14.850000000000001pt,pattern thickness=0.75pt,pattern radius=0pt, pattern color={rgb, 255:red, 0; green, 0; blue, 0}][line width=1.5]    (211.51,181.8) .. controls (238.17,142.72) and (258.5,131.06) .. (291.77,134.56) .. controls (325.05,138.06) and (338.38,197.34) .. (337.61,195.7) .. controls (336.83,194.06) and (312.17,201.39) .. (286.33,197.9) .. controls (260.5,194.41) and (225.5,181.72) .. (211.51,181.8) -- cycle ;

\draw (146,230.4) node [anchor=north west][inner sep=0.75pt]  [font=\tiny]  {$x_{1} =-N$};
\draw (477,229.4) node [anchor=north west][inner sep=0.75pt]  [font=\tiny]  {$x_{1} =N$};
\draw (468,80.4) node [anchor=north west][inner sep=0.75pt]  [font=\small]  {$D^{\overline{\tau }}$};
\draw (213.33,138.7) node [anchor=north west][inner sep=0.75pt]  [font=\footnotesize,color={rgb, 255:red, 0; green, 50; blue, 250 }  ,opacity=1 ]  {$\Sigma $};
\draw (320,247.4) node [anchor=north west][inner sep=0.75pt]    {$( c)$};

\end{tikzpicture}

				\end{minipage}%
				\begin{minipage}[t]{.55\textwidth}
					\centering

 
\tikzset{
pattern size/.store in=\mcSize, 
pattern size = 5pt,
pattern thickness/.store in=\mcThickness, 
pattern thickness = 0.3pt,
pattern radius/.store in=\mcRadius, 
pattern radius = 1pt}
\makeatletter
\pgfutil@ifundefined{pgf@pattern@name@_ds84z0u88}{
\pgfdeclarepatternformonly[\mcThickness,\mcSize]{_ds84z0u88}
{\pgfqpoint{0pt}{0pt}}
{\pgfpoint{\mcSize+\mcThickness}{\mcSize+\mcThickness}}
{\pgfpoint{\mcSize}{\mcSize}}
{
\pgfsetcolor{\tikz@pattern@color}
\pgfsetlinewidth{\mcThickness}
\pgfpathmoveto{\pgfqpoint{0pt}{0pt}}
\pgfpathlineto{\pgfpoint{\mcSize+\mcThickness}{\mcSize+\mcThickness}}
\pgfusepath{stroke}
}}
\makeatother
\tikzset{every picture/.style={line width=0.75pt}} 

\begin{tikzpicture}[x=0.56pt,y=0.59pt,yscale=-1,xscale=1]

\draw [line width=1.5]    (162,180) .. controls (187.67,178) and (201.67,180.67) .. (274.67,198.33) .. controls (347.67,216) and (410.67,154.33) .. (501.67,186) ;
\draw [line width=1.5]    (163,59) -- (162,180) ;
\draw [line width=1.5]    (502.67,72) -- (501.67,186) ;
\draw  [dash pattern={on 4.5pt off 4.5pt}]  (162.67,227) -- (500.67,227) ;
\draw [line width=0.75]  [dash pattern={on 4.5pt off 4.5pt}]  (162,180) -- (162.67,227) ;
\draw [line width=0.75]  [dash pattern={on 4.5pt off 4.5pt}]  (501.67,186) -- (500.67,227) ;
\draw [draw opacity=0][fill={rgb, 255:red, 155; green, 155; blue, 155 }  ,fill opacity=0.61 ][line width=1.5]    (210.82,54.36) .. controls (226.05,73.02) and (240.9,96.74) .. (274.08,97.67) .. controls (307.26,98.6) and (327.19,53.02) .. (329.65,52.04) .. controls (332.1,51.05) and (331.19,54.17) .. (307.67,41.33) .. controls (284.14,28.5) and (260.9,29.88) .. (235.48,43.6) .. controls (210.05,57.31) and (232.9,45.6) .. (210.82,54.36) -- cycle ;
\draw [draw opacity=0][pattern=_ds84z0u88,pattern size=14.100000000000001pt,pattern thickness=0.75pt,pattern radius=0pt, pattern color={rgb, 255:red, 74; green, 74; blue, 74}][line width=1.5]    (210.82,54.36) .. controls (225.48,73.31) and (242.69,98.6) .. (274.08,97.67) .. controls (305.48,96.74) and (328.9,50.45) .. (329.65,52.04) .. controls (330.39,53.62) and (332.24,53.62) .. (306.62,42.74) .. controls (281,31.86) and (261.76,29.02) .. (235.48,43.6) .. controls (209.19,58.17) and (232.9,45.6) .. (210.82,54.36) -- cycle ;
\draw [line width=1.5]    (163,59) .. controls (237.67,67.33) and (241.67,10) .. (307.67,41.33) .. controls (373.67,72.67) and (392.67,67.67) .. (423.67,62.33) .. controls (454.67,57) and (465.67,73.33) .. (502.67,72) ;
\draw [color={rgb, 255:red, 0; green, 50; blue, 250 }  ,draw opacity=1 ][line width=1.5]    (210.82,54.36) .. controls (226.17,79.59) and (249.37,99.19) .. (274.08,97.67) .. controls (298.8,96.15) and (313.74,82) .. (329.77,51.22) ;

\draw (146,232.4) node [anchor=north west][inner sep=0.75pt]  [font=\tiny]  {$x_{1} =-N$};
\draw (477,231.4) node [anchor=north west][inner sep=0.75pt]  [font=\tiny]  {$x_{1} =N$};
\draw (468,82.4) node [anchor=north west][inner sep=0.75pt]  [font=\small]  {$D^{\overline{\tau }}$};
\draw (216.05,80.41) node [anchor=north west][inner sep=0.75pt]  [font=\footnotesize,color={rgb, 255:red, 0; green, 50; blue, 250 }  ,opacity=1 ]  {$\Sigma $};
\draw (318,250.4) node [anchor=north west][inner sep=0.75pt]    {$( d)$};

\end{tikzpicture}

				\end{minipage}
    

 \hspace{.15in}

 
\tikzset{
pattern size/.store in=\mcSize, 
pattern size = 5pt,
pattern thickness/.store in=\mcThickness, 
pattern thickness = 0.3pt,
pattern radius/.store in=\mcRadius, 
pattern radius = 1pt}
\makeatletter
\pgfutil@ifundefined{pgf@pattern@name@_1ex8uvsm8}{
\pgfdeclarepatternformonly[\mcThickness,\mcSize]{_1ex8uvsm8}
{\pgfqpoint{0pt}{0pt}}
{\pgfpoint{\mcSize+\mcThickness}{\mcSize+\mcThickness}}
{\pgfpoint{\mcSize}{\mcSize}}
{
\pgfsetcolor{\tikz@pattern@color}
\pgfsetlinewidth{\mcThickness}
\pgfpathmoveto{\pgfqpoint{0pt}{0pt}}
\pgfpathlineto{\pgfpoint{\mcSize+\mcThickness}{\mcSize+\mcThickness}}
\pgfusepath{stroke}
}}
\makeatother
\tikzset{every picture/.style={line width=0.75pt}} 

\begin{tikzpicture}[x=0.75pt,y=0.6pt,yscale=-1,xscale=1]

\draw  [pattern=_1ex8uvsm8,pattern size=14.475000000000001pt,pattern thickness=0.75pt,pattern radius=0pt, pattern color={rgb, 255:red, 0; green, 0; blue, 0}] (70,196) -- (104.67,196) -- (104.67,215) -- (70,215) -- cycle ;
\draw  [fill={rgb, 255:red, 155; green, 155; blue, 155 }  ,fill opacity=0.61 ] (201,195) -- (235.67,195) -- (235.67,214) -- (201,214) -- cycle ;
\draw   (335,195) -- (369.67,195) -- (369.67,214) -- (335,214) -- cycle ;

\draw (112,198.4) node [anchor=north west][inner sep=0.75pt]  [font=\small]  {$Region\ \mathcal{O}$};
\draw (246,198.4) node [anchor=north west][inner sep=0.75pt]  [font=\small]  {$Region\ V$};
\draw (378,195.4) node [anchor=north west][inner sep=0.75pt]  [font=\small]  {$Region\ D^{\overline{\tau }} \backslash V$};

\end{tikzpicture}

\caption{The regions $V$ and $\mathcal{O}$ in the intersection domain}\label{f5.2}

			\end{figure}

We claim that
			\begin{eqnarray}\label{equiv}
				{\psi}_N={\psi}_N^{{\bar\tau}}\not\equiv 0 \text{ and } Q\ \  \ \ \text{on}\  \Sigma.
			\end{eqnarray}
If {${\psi}_N\equiv Q$ on $ \Sigma$}, define
	\[
	\hat{\psi}=\left\{
	\begin{aligned}
		& Q,\quad &&\text{in}\,\, \mathcal{O},\\
		& {\psi}_N,\quad &&\text{in}\,\, {\Omega}_{N}\setminus \mathcal{O}.
	\end{aligned}\right.
	\]
	One can obtain that $\hat{\psi}\in {\mathcal{K}}_N$ and
		\begin{equation}\label{energy11}
	\int_{{\Omega}_{N}}\frac{|\nabla \hat{\psi}|^2}{2}+F(\hat{\psi})dx \leq \int_{{\Omega}_{N}}\frac{|\nabla {\psi}_N|^2}{2}+F({\psi}_N)dx.	
	\end{equation}
	The equality in \eqref{energy11} holds
if and only if
${\psi}_N\equiv Q$ in $\mathcal{O}$. In fact, if ${\psi}_N\equiv Q$ in $\mathcal{O}$, this yields that $0\not\equiv {\psi}_N^{{\bar\tau}} <Q$ in $V$ and leads to a contradiction. Hence the strictly inequality must hold, i.e.
	\[
	\int_{{\Omega}_{N}}\frac{|\nabla \hat{\psi}|^2}{2}+F(\hat{\psi})dx < \int_{{\Omega}_{N}}\frac{|\nabla {\psi}_N|^2}{2}+F({\psi}_N)dx.	
	\]
	This also leads to a contradiction with the fact that ${\psi}_N$ is an energy minimizer for ${\mathcal{E}}_N$ over $\mathcal{K}_N$. Therefore, \eqref{equiv} holds. Similarly, one can prove that ${\psi}_N^{{\bar\tau}}\not\equiv 0$ on $\Sigma$. Thus, there exists an $\bar{x}\in \Sigma$ such that $0<{\psi}_N(\bar{x})={\psi}_N^{{\bar\tau}}(\bar{x})<Q$. Therefore, in a neighborhood of $\bar{x}$, $B_\epsilon(\bar{x})=\{x: |x-\bar{x}|<\epsilon\}$, one has
\begin{equation}
				\left\{
				\begin{aligned}
					&\Delta \Psi^{\bar\tau}(x) =c(x)\Psi^{\bar\tau}(x), \quad &&\text{in}\,\, B_\epsilon(\bar{x})\cap V,\\
					&\Psi^{\bar\tau}(x)< 0, \quad &&\text{in}\,\, B_\epsilon(\bar{x})\cap V,\\
					&\Psi^{\bar\tau}(\bar{x})=0,
				\end{aligned}
				\right.
			\end{equation}
where
	\begin{eqnarray*}
		c(x)=\frac{f({\psi}_N^{{\bar\tau}}(x))-f({\psi}_N(x))}{{\psi}_N^{{\bar\tau}}(x)-{\psi}_N(x)}.
	\end{eqnarray*}
It follows from Hopf lemma (\hspace{1sp}\cite{evans}) that one has
			\begin{equation}\label{intnormalhopf1}
				\frac{\partial \Psi^{\bar\tau}}{\partial \nu}(\bar{x})>0,
			\end{equation}
where $\nu$ is unit out normal at $\bar{x}$ to the domain $B_\epsilon(\bar{x})\cap V$.
Define
\[
\check{\psi}=\left\{
\begin{aligned}
	&{\psi}_N^{{\bar\tau}},\quad &&\text{in}\,\, V,\\
&{\psi}_N, \quad && \text{in}\,\, {\Omega}_N\setminus V.
\end{aligned}
\right.
\]
It follows from \eqref{m4} that {$\check{\psi}$ is also an energy minimizer for ${\mathcal{E}}_N$ over $\mathcal{K}_N$. Note that $\check{\psi}(\bar{x})\in (0, Q)$. It follows from the standard regularity theory for variational problems (cf. \cite{Giusti}) that there exists an $\alpha_0>0$ such that $\check{\psi}\in C^{1,\alpha_0}_{\text{loc}}(\{0<\psi_N<Q\}\cap \Omega_N)$.} More precisely, $ \check{\psi}$ must be smooth in the neighborhood of $\bar{x}$. However, \eqref{intnormalhopf1} implies that $\check{\psi}$ is not smooth in the neighborhood of $\bar{x}$. This leads to a contradiction.

Therefore, for any $\tau>0$, one has
			\begin{eqnarray*}
				{\Psi}^{\tau}(x)\geq 0\ \ \mbox{for any}  \ x\in {{D}^{\tau}},
			\end{eqnarray*}
			which implies
\[
	\partial_{x_2}{\psi}_N\geq 0\quad \text{in}\,\, {\Omega}_{N}.
	\]
This completes the proof of the lemma.
		\end{proof}
		
Applying the ideas developed in the proof of Lemma 	\ref{mon}, we can prove the following uniqueness results.


\begin{lemma}\label{lem6.5}
	Suppose that ${\psi}_{N,1}$ and ${\psi}_{N,2}$ are two energy minimizers for ${\mathcal{E}}_N$ over ${\mathcal{K}}_N$. Then $\psi_{N,1}\equiv \psi_{N,2}$ in $\Omega_N$.
\end{lemma}
\begin{proof}
The proof is quite similar to that for Lemma \ref{mon}. For $i=1,2$, define 
\begin{equation}
				\tilde\psi_{N,i}(x_1,x_2)=\left\{
				\begin{aligned}
					& Q\quad &&\text{if}\,\,x_2\geq h_1(x_1),\\
					&\psi_{N,i}(x_1,x_2)\quad &&\text{if}\,\, (x_1,x_2)\in \Omega_N,\\
					&0\quad &&\text{if}\,\,x_2 \leq h_0(x_1).
				\end{aligned}\right.
			\end{equation}
   For ease of notations, we still denote $\tilde\psi_{N,i}$ by $\psi_{N,i}$.

For any $\tau> 0$, denote
\begin{eqnarray*}
{\psi}_{N,1}^{\tau}(x)={\psi}_{N,1}(x_1,x_2+\tau)\ \ \text{and} \ \ {\psi}_{N,2}^{\tau}(x)={\psi}_{N,2}(x_1,x_2+\tau). \end{eqnarray*}

Let
\begin{equation}
{\Omega}_N^{\tau}=\{(x_1,x_2): (x_1,x_2+\tau)\in{\Omega}_N\} \ \ \text{and} \ \ 
{\mathcal{D}}^{\tau}={\Omega}_N\cap {\Omega}_N^{\tau},
\end{equation}
and define
		\begin{eqnarray*}
		{\Phi}^{\tau}(x)={\psi}_{N,1}^{\tau}(x)-\psi_{N,2}({x})\ \ \text{and} \ \ \tilde{\Phi}^{\tau}(x)={\psi}_{N,2}^{\tau}(x)-{\psi}_{N,1}(x) \ \ 
\mbox{for any}  \ x\in {{\mathcal{D}}^{\tau}}.
	\end{eqnarray*}
 
We show that for any $\tau>0$, one has
\begin{eqnarray}\label{p1}
		{\Phi}^{\tau}(x)\geq 0\ \ \text{and}  \ \ 
 {\tilde\Phi}^{\tau}(x)\geq 0\ \ \mbox{for any}  \ x\in {{\mathcal{D}}^{\tau}}.
	\end{eqnarray}
This implies ${\psi}_{N,1}(x)\geq \psi_{N,2}({x})$ and ${\psi}_{N,2}(x)\geq \psi_{N,1}({x})$ in $\Omega_N$. 

Thus, one has
$${\psi}_{N,1}(x)\equiv\psi_{N,2}({x}) \ \ \text{in}\ \Omega_N,$$
which implies the uniqueness of the energy minimizer. 

In the following, we prove ${\Phi}^{\tau}(x)\geq 0$ for any $x\in {{\mathcal{D}}^{\tau}}$ by the contradiction argument. The proof for ${\tilde\Phi}^{\tau}(x)\geq 0$ for any $x\in {{\mathcal{D}}^{\tau}}$ is similar. Suppose that there exists $\bar\tau>0$ such that the set
${U}=\{x\in {\mathcal{D}}^{\bar\tau}: {\Phi}^{\bar\tau}(x)< 0 \}$
is not empty. Since $ {\Phi}^{\bar\tau}$ is a continuous function, ${U}$ must be an open set. Hence one has either ${U}={\mathcal{D}}^{\bar\tau}$ or ${U}\subsetneq
			{\mathcal{D}}^{\bar\tau}$. Next, we show that neither of these possibilities happens.

	(i) If ${U}={\mathcal{D}}^{\bar\tau}$. Denote
\begin{eqnarray*}
				\hat{x}_1=\text{min}\{x_1: (x_1,x_2)\in \partial \mathcal{D}^{\bar\tau}\}.
			\end{eqnarray*}
One can see that $\psi_{N,2}(\hat{x})=0$ and $0<\psi^{\bar\tau}_{N,1}(\hat x)\leq Q$ for $\hat{x} = (\hat{x}_1, \hat{x}_2)\in \partial \mathcal{D}^{\bar\tau}\cap\Gamma_{0}$. 
Therefore, one has ${\Phi}^{\bar\tau}(\hat{x})> 0$.  
			On the other hand, it follows from continuity that ${\Phi}^{\bar\tau}(x)\leq 0$ for any $x\in \overline{{\mathcal{D}}^{\bar\tau}}$ which
			leads to a contradiction.

	(ii) Suppose that $U\subsetneq{\mathcal{D}^{\bar\tau}}$, let $V$ be any connected component of $U$. It follows from the argument in Lemma \ref{mon} that
			$$ {\Phi}^{\bar\tau}(x)=0 \ \ \ \ \text{on} \ \ \partial V$$
			and
			$$\partial V \cap \{(x_1,x_2)\in \overline{{\mathcal{D}}^{\bar\tau}}: x_1=-N \,\, \text{or}\,\, N\}=\emptyset. (\text{cf. Figure \ref{f5.2.2}}) $$
	
\begin{figure}[h]
				\centering
				\begin{minipage}[t]{.5\textwidth}

\tikzset{every picture/.style={line width=0.75pt}} 

\begin{tikzpicture}[x=0.62pt,y=0.6pt,yscale=-1,xscale=1]

\draw [line width=1.5]    (181,69) .. controls (255.67,77.33) and (259.67,20) .. (325.67,51.33) .. controls (391.67,82.67) and (418.67,93.33) .. (449.67,88) ;
\draw [line width=1.5]    (180,190) .. controls (205.67,188) and (219.67,190.67) .. (292.67,208.33) .. controls (365.67,226) and (356.67,175.33) .. (449.67,188.33) ;
\draw [line width=1.5]    (181,69) -- (180,190) ;
\draw [line width=1.5]    (449.67,88) -- (449.67,188.33) ;
\draw  [dash pattern={on 4.5pt off 4.5pt}]  (179.67,230) -- (449.67,229.96) ;
\draw [line width=0.75]  [dash pattern={on 4.5pt off 4.5pt}]  (180,190) -- (179.67,230) ;
\draw [line width=0.75]  [dash pattern={on 4.5pt off 4.5pt}]  (449.67,188.33) -- (449.67,229.96) ;
\draw  [color={rgb, 255:red, 0; green, 0; blue, 0 }  ,draw opacity=1 ][fill={rgb, 255:red, 155; green, 155; blue, 155 }  ,fill opacity=0.5 ][line width=1.5]  (239.68,132.77) .. controls (231.13,120.52) and (235.5,102.69) .. (249.45,92.95) .. controls (263.4,83.21) and (281.64,85.25) .. (290.2,97.5) .. controls (298.76,109.76) and (294.38,127.59) .. (280.43,137.33) .. controls (266.48,147.07) and (248.24,145.03) .. (239.68,132.77) -- cycle ;

\draw (164,237.4) node [anchor=north west][inner sep=0.75pt]  [font=\tiny]  {$x_{1} =-N$};
\draw (427,237.4) node [anchor=north west][inner sep=0.75pt]  [font=\tiny]  {$x_{1} =N$};
\draw (411,102.4) node [anchor=north west][inner sep=0.75pt]  [font=\small]  {$\mathcal{D}^{\overline{\tau }}$};
\draw (299,256.4) node [anchor=north west][inner sep=0.75pt]    {$( a)$};
\draw (261,110.4) node [anchor=north west][inner sep=0.75pt]  [font=\footnotesize]  {$V$};

\end{tikzpicture}

				\end{minipage}%
				\begin{minipage}[t]{.5\textwidth}
					\centering

\tikzset{every picture/.style={line width=0.75pt}} 

\begin{tikzpicture}[x=0.62pt,y=0.62pt,yscale=-1,xscale=1]

\draw  [color={rgb, 255:red, 0; green, 0; blue, 0 }  ,draw opacity=1 ][fill={rgb, 255:red, 155; green, 155; blue, 155 }  ,fill opacity=0.5 ][line width=1.5]  (229.4,145.7) .. controls (217.99,129.36) and (223.09,106.1) .. (240.79,93.74) .. controls (258.49,81.38) and (282.09,84.61) .. (293.49,100.95) .. controls (304.9,117.29) and (299.8,140.55) .. (282.1,152.91) .. controls (264.4,165.27) and (240.8,162.04) .. (229.4,145.7) -- cycle ;
\draw  [fill={rgb, 255:red, 255; green, 255; blue, 255 }  ,fill opacity=1 ][line width=1.5]  (247,130.71) .. controls (242.76,123.01) and (245.44,113.4) .. (252.97,109.25) .. controls (260.51,105.1) and (270.06,107.97) .. (274.3,115.67) .. controls (278.54,123.37) and (275.87,132.98) .. (268.33,137.13) .. controls (260.79,141.28) and (251.24,138.41) .. (247,130.71) -- cycle ;
\draw [line width=1.5]    (184,71) .. controls (258.67,79.33) and (262.67,22) .. (328.67,53.33) .. controls (394.67,84.67) and (421.67,95.33) .. (452.67,90) ;
\draw [line width=1.5]    (183,192) .. controls (208.67,190) and (222.67,192.67) .. (295.67,210.33) .. controls (368.67,228) and (359.67,177.33) .. (452.67,190.33) ;
\draw [line width=1.5]    (184,71) -- (183,192) ;
\draw  [dash pattern={on 4.5pt off 4.5pt}]  (182.67,224.89) -- (451.67,224.67) ;
\draw [line width=1.5]    (451.67,91) -- (451.67,191.33) ;
\draw [line width=0.75]  [dash pattern={on 4.5pt off 4.5pt}]  (183,192) -- (182.67,224.89) ;
\draw [line width=0.75]  [dash pattern={on 4.5pt off 4.5pt}]  (451.67,191.33) -- (451.67,224.67) ;

\draw (301,252.4) node [anchor=north west][inner sep=0.75pt]    {$( b)$};
\draw (165,232.4) node [anchor=north west][inner sep=0.75pt]  [font=\tiny]  {$x_{1} =-N$};
\draw (430,231.4) node [anchor=north west][inner sep=0.75pt]  [font=\tiny]  {$x_{1} =N$};
\draw (414,104.4) node [anchor=north west][inner sep=0.75pt]  [font=\small]  {$\mathcal{D}^{\overline{\tau }}$};
\draw (279,129.4) node [anchor=north west][inner sep=0.75pt]  [font=\footnotesize]  {$V$};

\end{tikzpicture}
					
				\end{minipage}
			
				\begin{minipage}[t]{.5\textwidth}

\tikzset{every picture/.style={line width=0.75pt}} 

\begin{tikzpicture}[x=0.62pt,y=0.6pt,yscale=-1,xscale=1]

\draw [color={rgb, 255:red, 0; green, 0; blue, 0 }  ,draw opacity=1 ][line width=1.5]    (231.51,197.8) .. controls (262.83,141.5) and (333.14,114.97) .. (357.08,208.67) ;
\draw [draw opacity=0][fill={rgb, 255:red, 155; green, 155; blue, 155 }  ,fill opacity=0.61 ][line width=1.5]    (231.51,197.8) .. controls (258.17,158.72) and (279.08,141.67) .. (313.08,148.67) .. controls (347.08,155.67) and (356.63,214.72) .. (357.61,211.7) .. controls (358.58,208.67) and (332.92,220.44) .. (305.58,214.67) .. controls (278.25,208.89) and (241.67,200.68) .. (231.51,197.8) -- cycle ;
\draw  [fill={rgb, 255:red, 255; green, 255; blue, 255 }  ,fill opacity=1 ][line width=1.5]  (281.71,181.23) .. controls (281.2,173.16) and (289.35,166.07) .. (299.93,165.4) .. controls (310.51,164.72) and (319.5,170.71) .. (320.02,178.77) .. controls (320.54,186.84) and (312.38,193.93) .. (301.8,194.61) .. controls (291.23,195.29) and (282.23,189.3) .. (281.71,181.23) -- cycle ;
\draw [line width=1.5]    (184,74) .. controls (258.67,82.33) and (262.67,25) .. (328.67,56.33) .. controls (394.67,87.67) and (421.67,98.33) .. (452.67,93) ;
\draw [line width=1.5]    (183,195) .. controls (208.67,193) and (222.67,195.67) .. (295.67,213.33) .. controls (368.67,231) and (359.67,180.33) .. (452.67,193.33) ;
\draw [line width=1.5]    (184,74) -- (183,195) ;
\draw  [dash pattern={on 4.5pt off 4.5pt}]  (182.67,235) -- (452.67,236.33) ;
\draw [line width=1.5]    (451.67,94) -- (451.67,194.33) ;
\draw [line width=0.75]  [dash pattern={on 4.5pt off 4.5pt}]  (183,195) -- (182.67,235) ;
\draw [line width=0.75]  [dash pattern={on 4.5pt off 4.5pt}]  (451.67,194.33) -- (451.67,235.96) ;

\draw (304,253.4) node [anchor=north west][inner sep=0.75pt]    {$( c)$};
\draw (167,242.4) node [anchor=north west][inner sep=0.75pt]  [font=\tiny]  {$x_{1} =-N$};
\draw (430,242.4) node [anchor=north west][inner sep=0.75pt]  [font=\tiny]  {$x_{1} =N$};
\draw (414,107.4) node [anchor=north west][inner sep=0.75pt]  [font=\small]  {$\mathcal{D}^{\overline{\tau }}$};
\draw (325,193.4) node [anchor=north west][inner sep=0.75pt]  [font=\footnotesize]  {$V$};

\end{tikzpicture}

				\end{minipage}%
				\begin{minipage}[t]{.5\textwidth}
					\centering

\tikzset{every picture/.style={line width=0.75pt}} 

\begin{tikzpicture}[x=0.62pt,y=0.6pt,yscale=-1,xscale=1]

\draw [draw opacity=0][fill={rgb, 255:red, 155; green, 155; blue, 155 }  ,fill opacity=0.61 ][line width=1.5]    (233.82,84.36) .. controls (245.67,104) and (263.9,126.74) .. (297.08,127.67) .. controls (330.26,128.6) and (350.19,83.02) .. (352.65,82.04) .. controls (355.1,81.05) and (354.19,84.17) .. (330.67,71.33) .. controls (307.14,58.5) and (283.9,59.88) .. (258.48,73.6) .. controls (233.05,87.31) and (255.9,75.6) .. (233.82,84.36) -- cycle ;
\draw [color={rgb, 255:red, 0; green, 0; blue, 0 }  ,draw opacity=1 ][line width=1.5]    (233.7,85.18) .. controls (249.05,110.41) and (272.25,130.01) .. (296.96,128.49) .. controls (321.68,126.96) and (336.62,112.82) .. (352.65,82.04) ;
\draw [line width=1.5]    (186,89) .. controls (260.67,97.33) and (264.67,40) .. (330.67,71.33) .. controls (396.67,102.67) and (423.67,113.33) .. (454.67,108) ;
\draw [line width=1.5]    (185,210) .. controls (210.67,208) and (224.67,210.67) .. (297.67,228.33) .. controls (370.67,246) and (361.67,195.33) .. (454.67,208.33) ;
\draw [line width=1.5]    (186,89) -- (185,210) ;
\draw  [dash pattern={on 4.5pt off 4.5pt}]  (184.67,250) -- (454.67,251.33) ;
\draw [line width=1.5]    (453.67,109) -- (453.67,209.33) ;
\draw [line width=0.75]  [dash pattern={on 4.5pt off 4.5pt}]  (185,210) -- (184.67,250) ;
\draw [line width=0.75]  [dash pattern={on 4.5pt off 4.5pt}]  (453.67,209.33) -- (453.67,250.96) ;
\draw  [fill={rgb, 255:red, 255; green, 255; blue, 255 }  ,fill opacity=1 ][line width=1.5]  (262.81,99.79) .. controls (262.35,92.58) and (268.41,86.32) .. (276.36,85.81) .. controls (284.3,85.3) and (291.11,90.73) .. (291.57,97.94) .. controls (292.04,105.15) and (285.97,111.41) .. (278.03,111.92) .. controls (270.09,112.43) and (263.28,107) .. (262.81,99.79) -- cycle ;
\draw  [fill={rgb, 255:red, 255; green, 255; blue, 255 }  ,fill opacity=1 ][line width=1.5]  (301,86.66) .. controls (300.64,81.04) and (306.24,76.1) .. (313.51,75.64) .. controls (320.77,75.17) and (326.96,79.35) .. (327.32,84.97) .. controls (327.68,90.59) and (322.08,95.53) .. (314.81,96) .. controls (307.54,96.46) and (301.36,92.28) .. (301,86.66) -- cycle ;

\draw (302,266.4) node [anchor=north west][inner sep=0.75pt]    {$( d)$};
\draw (169,257.4) node [anchor=north west][inner sep=0.75pt]  [font=\tiny]  {$x_{1} =-N$};
\draw (432,257.4) node [anchor=north west][inner sep=0.75pt]  [font=\tiny]  {$x_{1} =N$};
\draw (416,122.4) node [anchor=north west][inner sep=0.75pt]  [font=\small]  {$\mathcal{D}^{\overline{\tau }}$};
\draw (303,102.4) node [anchor=north west][inner sep=0.75pt]  [font=\footnotesize]  {$V$};

\end{tikzpicture}

				\end{minipage}
    


\caption{The region $V$ in the intersection domain}\label{f5.2.2}

			\end{figure}
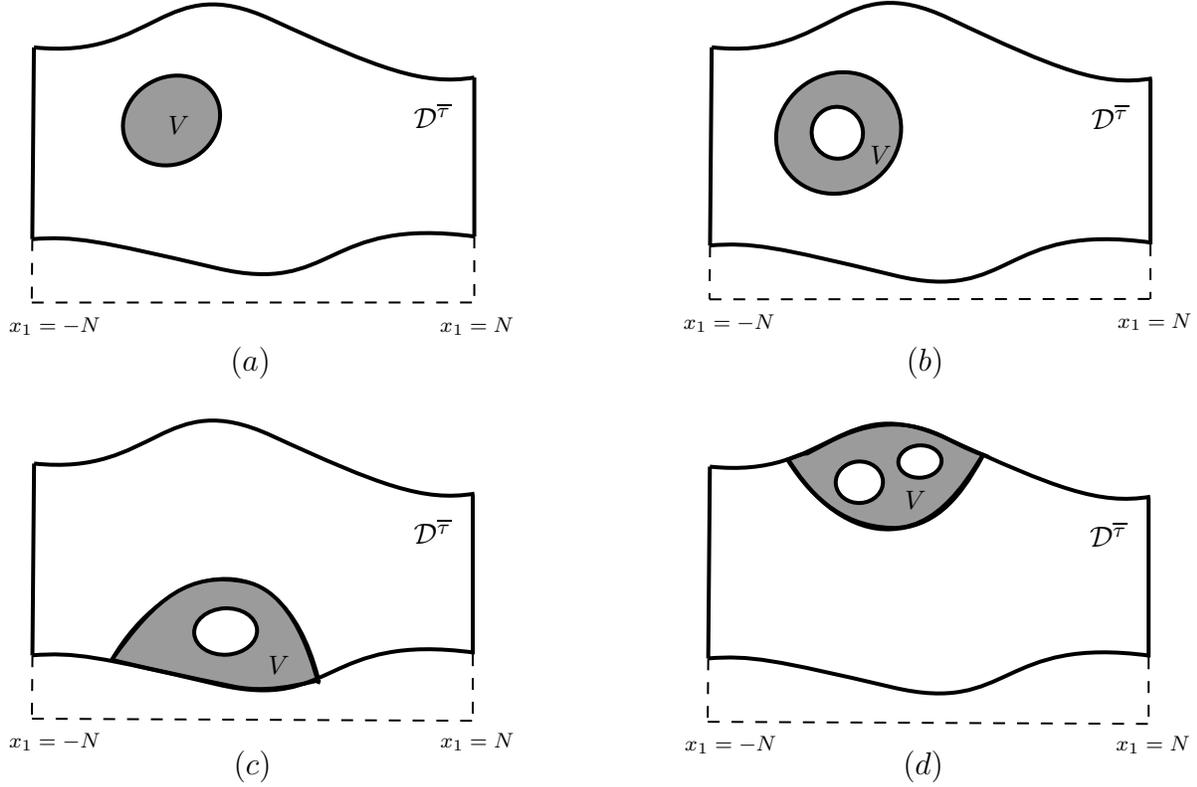

 Similarly, we analyze the problem via the comparison for the associated energies of the shifted functions.

\ \ {\it  Case (ii-1).}  Assume that
	\begin{eqnarray}\label{m5-2}
	\int_{V}\frac{|\nabla {\psi}_{N,1}^{{\bar\tau}}|^2}{2}+F({\psi}_{N,1}^{{\bar\tau}})dx < \int_{V}\frac{|\nabla \psi_{N,2}|^2}{2}+F(\psi_{N,2})dx.	\end{eqnarray}
 
	Define
	\[
	{\psi}^*=\left\{
	\begin{aligned}
		& {\psi}_{N,1}^{{\bar\tau}},\quad \text{in}\,\, V,\\
		& \psi_{N,2},\quad \text{in}\,\, {\Omega}_{N}\setminus V.
	\end{aligned}\right.
	\]
	Clearly, ${\psi}^*\in {\mathcal{K}}_N$ and ${\psi}^*$ satisfies
	\[
	\int_{{\Omega}_{N}}\frac{|\nabla {\psi}^*|^2}{2}+F({{\psi}^*})dx < \int_{{\Omega}_{N}}\frac{|\nabla {\psi}_{N,2}|^2}{2}+F({\psi}_{N,2})dx.	
	\]
	This contradicts with the fact that ${\psi}_{N,2}$ is an energy minimizer for ${\mathcal{E}}_N$ over ${\mathcal{K}}_N$.

\ \ {\it  Case (ii-2).} If
	\begin{eqnarray}\label{m6-2}
	\int_{V}\frac{|\nabla \psi_{N,2}|^2}{2}+F(\psi_{N,2})dx < \int_{V}\frac{|\nabla {\psi}_{N,1}^{{\bar\tau}}|^2}{2}+F({\psi}_{N,1}^{{\bar\tau}})dx.	
\end{eqnarray}

			Define
			\[
			{{\psi}^{**}}=\left\{
			\begin{aligned}
				& {\psi}_{N,2},\quad &&\text{in}\,\, V,\\
				&  {\psi}_{N,1}^{{\bar\tau}},\quad &&\text{in}\,\, {\Omega}_{N}^{\bar\tau}\setminus V.
			\end{aligned}\right.
			\]
			Clearly, ${\psi}^{**}\in {\mathcal{K}}_N^{\bar\tau}$ and ${\psi}^{**}$ satisfies
			\[
			\int_{{\Omega}_{N}^{\bar\tau}}\frac{|\nabla {\psi}^{**}|^2}{2}+F({\psi}^{**})dx < \int_{{\Omega}_{N}^{\bar\tau}}\frac{|\nabla {\psi}_{N,1}^{{\bar\tau}}|^2}{2}+F({\psi}_{N,1}^{{\bar\tau}})dx.	
			\]
			{This contradicts with the fact that ${\psi}_{N,1}^{\bar\tau}$ is an energy minimizer for ${\mathcal{E}}_N^{\bar\tau}$ over ${\mathcal{K}}_N^{\bar\tau}$.} Here ${\mathcal{E}}_N^{\bar\tau}$ and ${\mathcal{K}}_N^{\bar\tau}$ are defined in \eqref{E1} and \eqref{E2}. This implies  \eqref{m6-2} cannot happen.

\ \ {\it  Case (ii-3).} We assume that
\begin{eqnarray}\label{m4-2.1}
	\int_{V}\frac{|\nabla {\psi}_{N,1}^{{\bar\tau}}|^2}{2}+F({\psi}_{N,1}^{{\bar\tau}})dx = \int_{V}\frac{|\nabla \psi_{N,2}|^2}{2}+F(\psi_{N,2})dx.
\end{eqnarray}
Note that the complement of $V$ in $\mathbb{R}^2$ has a unique unbounded connected component $A$. Let $\Sigma=\partial A\cap \mathcal{D}^{\bar\tau}$. Using the same argument in Lemma \ref{mon} that there exists an $\bar{x}\in \Sigma$ such that $0<{\psi}_{N,2}(\bar{x})={\psi}_{N,1}^{{\bar\tau}}(\bar{x})<Q$. Therefore, one has
\begin{equation}
				\left\{
				\begin{aligned}
					&\Delta \Phi^{\bar\tau}(x) =c(x)\Phi^{\bar\tau}(x), \quad &&\text{in}\,\, B_\epsilon(\bar{x})\cap V,\\
					&\Phi^{\bar\tau}(x)< 0, \quad &&\text{in}\,\, B_\epsilon(\bar{x})\cap V,\\
					&\Phi^{\bar\tau}(\bar{x})=0,
				\end{aligned}
				\right.
			\end{equation}
where
	\begin{eqnarray*}
		c(x)=\frac{f({\psi}_{N,1}^{{\bar\tau}})-f({\psi}_{N,2})}{{\psi}_{N,1}^{{\bar\tau}}-{\psi}_{N,2}}.
	\end{eqnarray*}
It follows from Hopf lemma (\hspace{1sp}\cite{evans}) that one has
			\begin{equation}\label{intnormalhopf1}
				\frac{\partial \Phi^{\bar\tau}}{\partial \nu}(\bar{x})>0,
			\end{equation}
where $\nu$ is unit out normal at $\bar{x}$ to the domain $B_\epsilon(\bar{x})\cap V$.
Define
\[
\hat{\psi}(x)=\left\{
\begin{aligned}
&	{\psi}_{N,2}(x), \quad \text{if}\,\, x\in {\Omega}_N\setminus V,\\
&	{\psi}_{N,1}^{{\bar\tau}}(x),\quad \text{if}\,\, x\in V.
\end{aligned}
\right.
\]
It follows from \eqref{m4-2.1} that $\hat{\psi}$ is also an energy minimizer for ${\mathcal{E}}_N$ over $\mathcal{K}_N$. Combining the standard regularity theory for variational problems (cf. \cite{Giusti}) and the fact that $\hat{\psi}(\bar{x})\in (0, Q)$ yields $\hat{\psi}$ must be smooth in the neighborhood of $\bar{x}$. However, \eqref{intnormalhopf1} implies that $\hat{\psi}$ is not smooth in the neighborhood of $\bar{x}$. This leads to a contradiction.
This completes the proof of the lemma.
\end{proof}


	\subsection{Regularity  of the boundary of stagnation region (the free boundary)}

 In this subsection, we prove the regularity of the boundary of stagnation region when the set of stagnation points  is not empty. In the next subsection, we prove the existence of stagnation points when the boundary of nozzle satisfies certain geometric conditions. 

In fact, the boundary of stagnation region  can be regarded as an obstacle type free boundary. In order to prove its regularity,  
 we first establish uniform $C^{1,1}$ local regularity of the energy minimizer in the truncated domain $\Omega_N$. With the aid of this regularity result, we prove that the boundary of non-stagnant region is $C^1$. The key point is to use the sliding method to derive some monotonicity property which helps to show the positive density of the coincidence set.
 This, together with the analysis near the touching points between the free and fixed boundary, yields global regularity of the boundary of non-stagnant region. 
 It should be remarked that in general $C^1$-regularity is the optimal smoothness near the touching points of the free boundary and the fixed boundary. In fact, there are examples of optimality, in the sense that the free boundary can not be $C^{1,Dini}$ at the touching point; see \cite{SU_Duke}. We also show that the boundary of stagnation region away from the solid boundary can have better regularity and belong to $C_{loc}^{1,\beta}$ for some $\beta>0$.
 
It follows from Lemma \ref{mon} that one knows $\partial_{x_2}\psi_N\geq 0$ in $\Omega_N$. For any $x_1\in(-N,N)$, define
\begin{equation}\label{defh0N}
    \tilde{h}_{0,N}(x_1):=\sup\{x_2: \psi_N(x_1,x_2)=0, x_2\in [h_0(x_1), h_1(x_1)]\}
\end{equation}
 and 
\begin{equation}\label{defh1N}
    \tilde{h}_{1,N}(x_1):=\inf\{x_2: \psi_N(x_1,x_2)=Q, x_2\in [h_0(x_1), h_1(x_1)]\}.
\end{equation}
 Denote
 \[
 \tilde{\Omega}_N=\{(x_1,x_2): \tilde{h}_{0,N}(x_1)<x_2<\tilde{h}_{1,N}(x_1), -N<x_1<N\}, \quad \Omega_N^\diamond =\Omega_N\setminus \tilde{\Omega}_N, 
 \]
 and
\[
\tilde{\Gamma}_{0, N}=\partial \tilde{\Omega}_N \cap \{x: \psi(x)=0\}, \,\, \tilde{\Gamma}_{1, N}=\partial \tilde{\Omega}_N \cap \{x: \psi(x)=Q\},
\,\, \check{\Gamma}_{i, N}=\tilde{\Gamma}_{i,N}\setminus \Gamma_i\,\, \text{for}\,\, i=0,1.
\]

 Note that $\Omega_N\setminus \tilde{\Omega}_N$ is the stagnation region and  $\check{\Gamma}_{i,N}$ forms the boundary of stagnation region, which is indeed a free boundary of obstacle type. 
In the next subsection, we show that $ \Omega_N^\diamond$ is not empty when $h_i$ satisfies certain conditions.  
The main goal of this subsection is to study the regularity of $\partial \tilde{\Omega}_N$ when $ \Omega_N^\diamond$ is not an empty set. 
Note that  the free boundary $\check \Gamma_{i,N}$ ($i=0,1$) may well touch the corresponding fixed boundary $\Gamma_{i,N}$ $(i=0,1)$ at some points, and may cease to exist for certain $x_1$-intervals, see Figure \ref{1-9}.
Hence one of the major difficulties is to prove the global regularity of $\tilde{\Gamma}_{i,N}$, which gives the behavior in the neighborhood of the intersection points of $\Gamma_i$  and $\overline{\check{\Gamma}_{i,N}}$.

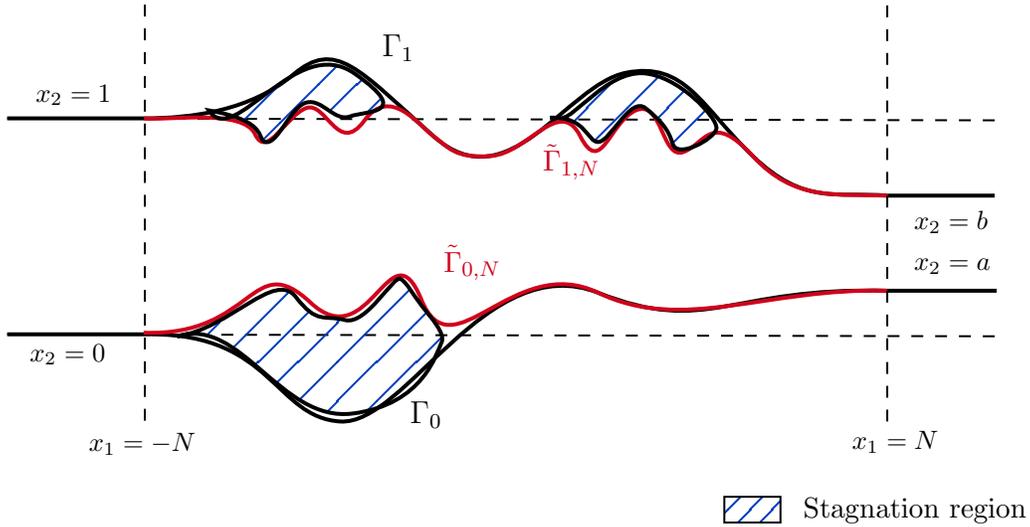
\begin{figure}[h]
			\centering

 
\tikzset{
pattern size/.store in=\mcSize, 
pattern size = 5pt,
pattern thickness/.store in=\mcThickness, 
pattern thickness = 0.3pt,
pattern radius/.store in=\mcRadius, 
pattern radius = 1pt}
\makeatletter
\pgfutil@ifundefined{pgf@pattern@name@_1dhpxtpvc}{
\pgfdeclarepatternformonly[\mcThickness,\mcSize]{_1dhpxtpvc}
{\pgfqpoint{0pt}{0pt}}
{\pgfpoint{\mcSize+\mcThickness}{\mcSize+\mcThickness}}
{\pgfpoint{\mcSize}{\mcSize}}
{
\pgfsetcolor{\tikz@pattern@color}
\pgfsetlinewidth{\mcThickness}
\pgfpathmoveto{\pgfqpoint{0pt}{0pt}}
\pgfpathlineto{\pgfpoint{\mcSize+\mcThickness}{\mcSize+\mcThickness}}
\pgfusepath{stroke}
}}
\makeatother

 
\tikzset{
pattern size/.store in=\mcSize, 
pattern size = 5pt,
pattern thickness/.store in=\mcThickness, 
pattern thickness = 0.3pt,
pattern radius/.store in=\mcRadius, 
pattern radius = 1pt}
\makeatletter
\pgfutil@ifundefined{pgf@pattern@name@_5gnp5tr71}{
\pgfdeclarepatternformonly[\mcThickness,\mcSize]{_5gnp5tr71}
{\pgfqpoint{0pt}{0pt}}
{\pgfpoint{\mcSize+\mcThickness}{\mcSize+\mcThickness}}
{\pgfpoint{\mcSize}{\mcSize}}
{
\pgfsetcolor{\tikz@pattern@color}
\pgfsetlinewidth{\mcThickness}
\pgfpathmoveto{\pgfqpoint{0pt}{0pt}}
\pgfpathlineto{\pgfpoint{\mcSize+\mcThickness}{\mcSize+\mcThickness}}
\pgfusepath{stroke}
}}
\makeatother

 
\tikzset{
pattern size/.store in=\mcSize, 
pattern size = 5pt,
pattern thickness/.store in=\mcThickness, 
pattern thickness = 0.3pt,
pattern radius/.store in=\mcRadius, 
pattern radius = 1pt}
\makeatletter
\pgfutil@ifundefined{pgf@pattern@name@_uaem9s2a4}{
\pgfdeclarepatternformonly[\mcThickness,\mcSize]{_uaem9s2a4}
{\pgfqpoint{0pt}{0pt}}
{\pgfpoint{\mcSize+\mcThickness}{\mcSize+\mcThickness}}
{\pgfpoint{\mcSize}{\mcSize}}
{
\pgfsetcolor{\tikz@pattern@color}
\pgfsetlinewidth{\mcThickness}
\pgfpathmoveto{\pgfqpoint{0pt}{0pt}}
\pgfpathlineto{\pgfpoint{\mcSize+\mcThickness}{\mcSize+\mcThickness}}
\pgfusepath{stroke}
}}
\makeatother

 
\tikzset{
pattern size/.store in=\mcSize, 
pattern size = 5pt,
pattern thickness/.store in=\mcThickness, 
pattern thickness = 0.3pt,
pattern radius/.store in=\mcRadius, 
pattern radius = 1pt}
\makeatletter
\pgfutil@ifundefined{pgf@pattern@name@_humqgq9e4}{
\pgfdeclarepatternformonly[\mcThickness,\mcSize]{_humqgq9e4}
{\pgfqpoint{0pt}{0pt}}
{\pgfpoint{\mcSize+\mcThickness}{\mcSize+\mcThickness}}
{\pgfpoint{\mcSize}{\mcSize}}
{
\pgfsetcolor{\tikz@pattern@color}
\pgfsetlinewidth{\mcThickness}
\pgfpathmoveto{\pgfqpoint{0pt}{0pt}}
\pgfpathlineto{\pgfpoint{\mcSize+\mcThickness}{\mcSize+\mcThickness}}
\pgfusepath{stroke}
}}
\makeatother
\tikzset{every picture/.style={line width=0.75pt}} 

\begin{tikzpicture}[x=0.75pt,y=0.75pt,yscale=-1,xscale=1]

\draw [line width=0.75]  [dash pattern={on 4.5pt off 4.5pt}]  (75.33,75) -- (576.33,76) ;
\draw [line width=0.75]  [dash pattern={on 4.5pt off 4.5pt}]  (75.33,184) -- (576.33,185) ;
\draw [line width=1.5]    (144.33,75) .. controls (168.67,75) and (198.33,71) .. (219.64,52.87) .. controls (240.96,34.75) and (252.33,50) .. (290.33,84) .. controls (328.33,118) and (357.33,58) .. (387.33,52) .. controls (417.33,46) and (425.33,69) .. (451.33,94) .. controls (477.33,119) and (493.33,112) .. (519.33,114) ;
\draw [line width=1.5]    (148,184) .. controls (209.33,184) and (210.33,226) .. (243.33,228) .. controls (276.33,230) and (317.33,141) .. (373.33,163) .. controls (429.33,185) and (438.17,160) .. (519.17,162) ;
\draw [line width=1.5]    (75.33,75) -- (144.33,75) ;
\draw [line width=1.5]    (519.33,114) -- (573.33,114) ;
\draw [line width=1.5]    (75.33,184) -- (148,184) ;
\draw [line width=1.5]    (519.17,162) -- (574.33,162) ;
\draw [color={rgb, 255:red, 208; green, 2; blue, 27 }  ,draw opacity=1 ][line width=1.5]    (144.33,75) .. controls (157.01,76.18) and (194.33,70) .. (200.33,83) .. controls (206.33,96) and (216.19,73.29) .. (226.21,69.78) .. controls (236.23,66.28) and (242.21,94.18) .. (255.41,76.18) .. controls (268.61,58.18) and (282.86,78.62) .. (290.33,84) .. controls (297.81,89.38) and (300.61,95.38) .. (319.33,94) .. controls (338.06,92.62) and (353.81,63.38) .. (364.33,84) .. controls (374.86,104.62) and (378.93,77.05) .. (392.33,71) .. controls (405.73,64.95) and (405.33,96) .. (416.33,92) .. controls (427.33,88) and (430.33,71) .. (451.33,94) .. controls (472.33,117) and (493.01,112.98) .. (519.33,114) ;
\draw [color={rgb, 255:red, 208; green, 2; blue, 27 }  ,draw opacity=1 ][line width=1.5]    (144.13,183.2) .. controls (200.47,185.2) and (198.33,140) .. (228.33,168) .. controls (258.33,196) and (268.57,127.87) .. (284.19,166.43) .. controls (299.81,204.98) and (330.33,142.33) .. (373.33,163) .. controls (416.33,183.67) and (458.33,160.33) .. (519.17,162) ;
\draw [draw opacity=0][pattern=_1dhpxtpvc,pattern size=14.475000000000001pt,pattern thickness=0.75pt,pattern radius=0pt, pattern color={rgb, 255:red, 7; green, 62; blue, 192}][line width=1.5]    (177,71) .. controls (190.13,90.8) and (220.93,31.6) .. (250.53,52.4) .. controls (280.13,73.2) and (256.33,71) .. (248.33,74) .. controls (240.33,77) and (234.33,65) .. (226.33,67) .. controls (218.33,69) and (212.42,86.96) .. (206.33,87) .. controls (200.24,87.04) and (209.33,76.56) .. (177,71) -- cycle ;
\draw [draw opacity=0][pattern=_5gnp5tr71,pattern size=15.149999999999999pt,pattern thickness=0.75pt,pattern radius=0pt, pattern color={rgb, 255:red, 7; green, 62; blue, 192}][line width=1.5]    (349,74.67) .. controls (368,73.89) and (367.68,87.46) .. (373.67,86.89) .. controls (379.65,86.32) and (389.67,63.56) .. (397.33,69.89) .. controls (405,76.22) and (408.67,80.22) .. (411,87.89) .. controls (413.33,95.56) and (432,85.56) .. (433,80.22) .. controls (434,74.89) and (413.33,52.89) .. (395.67,52.22) .. controls (378,51.56) and (361.33,74.67) .. (355.33,74.67) ;
\draw [draw opacity=0][pattern=_uaem9s2a4,pattern size=14.924999999999999pt,pattern thickness=0.75pt,pattern radius=0pt, pattern color={rgb, 255:red, 7; green, 62; blue, 192}][line width=1.5]    (161,185) .. controls (187.83,181.17) and (210.13,151.2) .. (220.93,165.2) .. controls (231.73,179.2) and (234.93,174) .. (245.73,176.8) .. controls (256.53,179.6) and (270.35,153.86) .. (273.73,156) .. controls (277.12,158.14) and (279.88,163.82) .. (284.53,169.78) .. controls (289.19,175.73) and (296.61,183.68) .. (294.33,188.22) .. controls (292.06,192.77) and (285.67,219.56) .. (249.33,224) .. controls (213,228.44) and (193.33,174.67) .. (161,185) -- cycle ;
\draw  [pattern=_humqgq9e4,pattern size=8.399999999999999pt,pattern thickness=0.75pt,pattern radius=0pt, pattern color={rgb, 255:red, 7; green, 62; blue, 192}] (437,266) -- (465.33,266) -- (465.33,279) -- (437,279) -- cycle ;
\draw [line width=0.75]  [dash pattern={on 4.5pt off 4.5pt}]  (519.17,18) -- (519.17,162) -- (519.17,228) ;
\draw [line width=0.75]  [dash pattern={on 4.5pt off 4.5pt}]  (144.73,17.6) -- (144.73,70.6) -- (144.73,161.6) -- (144.73,227.6) ;

\draw (85.33,187.4) node [anchor=north west][inner sep=0.75pt]  [font=\footnotesize]  {$x_{2} =0$};
\draw (88.33,56.4) node [anchor=north west][inner sep=0.75pt]  [font=\footnotesize]  {$x_{2} =1$};
\draw (531.33,144.4) node [anchor=north west][inner sep=0.75pt]  [font=\footnotesize]  {$x_{2} =a$};
\draw (531.33,120.4) node [anchor=north west][inner sep=0.75pt]  [font=\footnotesize]  {$x_{2} =b$};
\draw (277,216.4) node [anchor=north west][inner sep=0.75pt]    {$\Gamma _{0}$};
\draw (263,31.4) node [anchor=north west][inner sep=0.75pt]    {$\Gamma _{1}$};
\draw (344,86.4) node [anchor=north west][inner sep=0.75pt]  [font=\small,color={rgb, 255:red, 208; green, 2; blue, 27 }  ,opacity=1 ]  {$\tilde{\Gamma }_{1,N}$};
\draw (475.33,265) node [anchor=north west][inner sep=0.75pt]  [font=\small] [align=left] {Stagnation region};
\draw (500,231.4) node [anchor=north west][inner sep=0.75pt]  [font=\footnotesize]  {$x_{1} =N$};
\draw (115,232.4) node [anchor=north west][inner sep=0.75pt]  [font=\footnotesize]  {$x_{1} =-N$};
\draw (294,137.4) node [anchor=north west][inner sep=0.75pt]  [font=\small,color={rgb, 255:red, 208; green, 2; blue, 27 }  ,opacity=1 ]  {$\tilde{\Gamma }_{0,N}$};

\end{tikzpicture}

\caption{The free boundaries $\tilde{\Gamma }_{0,N}$ and $\tilde{\Gamma }_{1,N}$}\label{1-9}
\end{figure}

In order to prove the regularity of $\tilde{\Gamma}_{i, N} (i=0,1)$, we investigate the optimal regularity of the function $\psi_N$ when the free boundary $\check{\Gamma}_{i,N}$ is present; if $\check{\Gamma}_{i,N}=\emptyset$, then $\tilde{h}_{i,N}=h_{i,N}$.  
Since  the stream function  satisfies the free boundary problem
\begin{equation}\label{pbpsiN}
	\left\{
	\begin{aligned}	
&	\Delta \psi_N=f(\psi_N)\chi_{\{0<\psi_N<Q\}}\quad &&\text{in}\,\, \Omega_N,\\
&\psi_N=g_N\quad &&\text{on}\,\, \partial \Omega_N,
\end{aligned}
\right.
\end{equation}
where $f\in C^{1/2}([0, Q])\cap C^{\infty}((0,Q))$ is shown in Proposition \ref{propstream}, we are dealing with an obstacle problem of semilinear type.

By uniform  continuity of $\psi_N $ over the nozzle, we know that the free boundaries $\check{\Gamma}_{0,N}$ and $\check{\Gamma}_{1,N}$ are uniformly away from each other, by a distance $2r_0$ for some $r_0 >0$.

Note that the analysis for $\tilde{\Gamma}_{0,N}$ is almost the same as that for $\tilde{\Gamma}_{1,N}$. Hence in the rest of this subsection, we focus on the analysis on $\tilde{\Gamma}_{0,N}$.
Recall that close to $\tilde \Gamma_{0,N}$ we have the following free boundary problem for $\psi_N$
\begin{equation}\label{e1}
	\Delta \psi_N=f(\psi_N)\chi_{\{\psi_N>0\}}=6Q(1-2\kappa(\psi_N))\chi_{\{\psi_N>0\}}.
\end{equation}
The classical regularity theory for elliptic equations implies that $\psi_N\in C_{loc}^{1, \alpha}(\Omega_N)$ for some $\alpha\in(0,1)$. 
Nevertheless, since we deal with a  free boundary problem, we can prove a stronger result.
\begin{pro}\label{Thm:C11}
There exists a universal constant $C_0$ (independent of $N$)
such that
$$
|D^2 \psi_N | \leq C_0 \quad \hbox{in } K \cap \Omega_N,
$$
for any compact set $K$ in $\mathbb{R}^2$.
\end{pro}

For    $C^{3,\alpha}$ boundary $\Gamma_i$,   the uniform $W^{2,\infty} \cap C^{1,1}$ regularity of solutions was proved in  \cite{Ger85, jensen80}.  Such an estimate can also be achieved when the regularity of the fixed boundary was relaxed to     $C^{2,\alpha}$ in   \cite{Liberman}.

For the analysis of the free boundary regularity up to the fixed boundary (the boundary of nozzle) there seems no need to require $C^{2,\alpha}$ regularity for the fixed boundary. Indeed, we shall prove below, using blow-up techniques, that $C^{1,Dini}$ boundaries give rise to quadratic decay  for solutions to our problem toward any free boundary point on the fixed one.

The proof of Proposition \ref{Thm:C11} relies crucially on the following quadratic growth property of the solutions.

\begin{lemma}\label{lem5.7} For $N \geq 2$,
assume  $\Gamma_{0,N}$ is at least uniformly $C^{1,Dini}$.
Then there is a universal constant $C > 0$ 
(independent of $N$) such that for any 
$z \in \overline{\check \Gamma}_{0,N  } \cap \overline{\Omega}_{N-1}$ we have 
\begin{equation}\label{Bdry-estimate}
    \psi_N (x) \leq C|x-z|^2  \qquad \text{for any} \ x \in \Omega_N.
\end{equation}
A similar conclusion holds for $z \in  \overline{\check\Gamma}_{1,N}\cap  \overline{\Omega}_{N -1}  $, where the estimate now becomes 
\begin{equation}\label{Bdry-estimate2}
  Q-   \psi_N (x) \leq C|x-z|^2  \qquad \text{for any} \ x \in \Omega_N.
\end{equation}
\end{lemma}

In Lemma \ref{lem5.7} we take $z $ in a smaller domain $\Omega_{N-1}$ to avoid the lateral  boundary of $\Omega_N$. Note that the $C^{1,Dini}$ assumption is needed to assure the uniform $C^1$-regularity of the solutions up the fixed boundary.

\begin{proof}
As the proofs of both estimates are similar, we only prove the estimate \eqref{Bdry-estimate}. 

By elliptic regularity theory, the solutions 
$\{\psi_N\}$ are uniformly $C^{1,\alpha}(\overline{\Omega_{N-\frac{1}{2}}})$.
For $0 < r  < 1/10 $, define
$$S_r (\psi_N ,z) = \sup_{B_r(z)} \psi_N.$$ We aim to prove 
$S_r (\psi_N,z) \leq C_0 r^2$ for a universal $C_0 > 0$.

Suppose towards a contradiction this fails. That is for any positive integer $j $, there is a sequence $\psi_{N_j}$ solving \eqref{pbpsiN}, a  free boundary point $z^j$,
and radius $r_j$ satisfying 
$$S_r (\psi_{N_j},z^j  ) \leq j r^2  \quad \text{for any}\,\, \ r \geq r_j, \qquad \hbox{and}\qquad  
S_{r_j} (\psi_{N_j},z^j) = j r_j^2. $$
Define
$$  \psi^\#_{N_j} (x): = \frac{\psi_{N_j}(r_j x + z^j )}{jr_j^2} \quad\text{and}\quad   \Omega^\#_{N_j}: = \frac{1}{r_j}( \Omega_N - z^j)=\left\{\frac{x-z^j}{r_j}: x\in \Omega_N\right\}.$$
Note that 
$$
|\Delta  \psi^\#_{N_j} | \leq \frac{C}{j} \quad \hbox{in }  \Omega^\#_{N_j},\quad 
\sup_{B_1 \cap  \Omega^\#_{N_j}} \psi^\#_{N_j}= 1,\quad \text{and}\quad 
\sup_{B_2 \cap  \Omega^\#_{N_j}}\psi^\#_{N_j} \leq 4.
$$

We consider three separate case.
For a subsequence of $\{z^j\}$ (labeled again as $z^j$),  one of the following cases must hold.

1)  $z^j \in \overline{\check \Gamma}_{0,N}\cap \Gamma_{0,N-1}$.

2)  $z^j \in  \check \Gamma_{0,N}\cap \Omega_{N-1}  $, and $dist(z^j, \overline{\check \Gamma}_{0,N}\cap \Gamma_{0,N}) = o( r_j ) $.

3)  $z^j \in  \check \Gamma_{0,N}\cap \Omega_{N-1}  $, and $dist(z^j, \overline{\check \Gamma}_{0,N}\cap \Gamma_{0,N}) > c_0 r_j $ (for some $c_0 > 0$).

To prove  cases 1) and 2), 
we note further that (after a suitable rotation)
$$B_1 \cap  \Omega^\#_{N_j} \  \to \  B_1 \cap \{ x_2 > 0\} ,$$
since $\partial \Omega_{N_j}$ is (uniformly) $C^{1,Dini}$ and thus  has a normal and a tangent plane.

A crucial further property is that $ \psi^\#_{N_j}$ are uniformly $C^1$ up to the boundary of the nozzle, and hence the convergence is uniformly  in $C^1$ in the closure of the  varying/scaled  domain, $B_1 \cap  \Omega^\#_{N_j}$. Therefore, the limit function $\psi_0$ (after a a suitable subsequence) satisfies
$$
\Delta \psi_0 = 0 , \quad 
\psi_0 \geq 0 \quad \hbox{in } B_1 \cap \{x_2 > 0\},
$$
$$
\psi_0 = 0 \quad \hbox{on } B_1 \cap \{x_2 = 0\},$$
$$
\sup_{B_1 \cap \{x_2 > 0\} }  \psi_0 = 1,\qquad 
\sup_{B_2 \cap \{x_2 > 0\}}  \psi_0 \leq 4, \qquad 
\nabla \psi_0 (0) = 0,
$$
where the latter follows from uniform $C^{1,Dini}$ regularity and that $\nabla \psi^\#_{N_j} (0) = 0$, as $\{z^j\}$ are  free boundary points.
This leads to contradiction with Hopf's boundary point lemma (see \cite{Hopf}).
Hence the lemma is proved in cases 1) and 2).

In case 3) the same argument leads us to an interior point, such that 
$\Delta \psi_0 = 0$ in $B_{c_0} (0)$, and it takes a local minimum at the origin. This also contradicts to the local minimum principle and finishes the proof of the lemma. 
\end{proof}

Now we are ready to prove Proposition \ref{Thm:C11}.

\begin{proof}[Proof of Proposition \ref{Thm:C11}]
 
  Let $y \in \Omega_{N-1}$ be such that $0 < \psi_N(y) < Q$, 
  and $z \in 
  (\overline{\check \Gamma_{0,N} \cup  \check \Gamma_{1,N})  \cap \Omega_N}$ be the closest point to $y$. We may assume,
  $\psi_N (z) =0$ (i.e., $z \in 
  \overline{\check \Gamma_{0,N}   \cap \Omega_N}$), as the other case is treated similarly. 
  Set 
  $r_y = |y-z| $, 
  $$  \psi^\flat_N (x) = \frac{\psi_N(r_y x + y) }{r_y^2} \quad \text{and} \quad \Omega^\flat_N=\frac1{r_y}(\Omega_N - y) .$$

It follows from Lemma \ref{lem5.7} that $\psi^\flat_N$ is bounded in $B_1(0) $ and satisfies a scaled free boundary equation in the scaled domain. We observe that in $B_{1/2}$ there are no more  free boundary points, and the boundary 
$\partial \Omega^\flat_N$ is $C^{2,\alpha}$.

 Since  $f\in C^{1/2}([0, Q])\cap C^{\infty}((0,Q))$
 and  $\psi_N\in C^{1, \alpha}(\Omega_N)$ for some $\alpha\in(0,1)$, one has ${f}^\flat (x) := f(\psi_N (x) )\in C^{{\alpha}/{2}}(\Omega_N)$. 
From this and that  
$\Delta \psi^\flat_N =  f^\flat$ in 
$\Omega^\flat_N $, we can invoke elliptic regularity theory to conclude    $ \psi^\flat_N$ is $C^2$ in 
$B_{1} (0)  \cap \Omega^\flat_N$. More precisely,
$$
|D^2 \psi_N(y) | = |D^2  \psi^\flat_N(0) | \leq C_0
$$
for a universal constant $C_0$. This completes the proof of Proposition \ref{Thm:C11}.
\end{proof}


Furthermore, we have the following lemma on the nondegeneracy of free boundary. 

\begin{lemma}
Let $\psi_N$ be a solution of \eqref{e1}. Then for any $x_0\in \bar{\check
\Gamma}_{0,N}\cap \bar{\Omega}_{N-1}$, there exists a constant $c$ such that 
\begin{eqnarray}
\sup_{B_r(x_0)}{\psi_N}\geq cr^2>0
\end{eqnarray}
provided $B_r(x_0)\subset \Omega_{N-1} \cap \{(x_1,x_2): x_2<\frac{1}{2}\}$.
\end{lemma}
\begin{proof}
The proof is same as the proof for the classical obstacle problem. Please refer to \cite{C-R}. 
\end{proof}

In order to study the regularity of $\tilde{\Gamma}_{i, N}$ ($i=0,1$), we first prove
 the following lemma on the continuity of $\tilde{h}_{i,N}$ ($i=0$, $1$).

\begin{lemma}Assume 
$\Gamma_{i,N}$ ($i=0,1$) is at least uniformly $C^{1,Dini}$.
The functions $x_2=\tilde{h}_{i,N}(x_1)$ ($i=0,1$) are continuous. 
    \end{lemma}

    \begin{proof} 
 
   
   We  consider the problem only for  $\tilde{h}_{0,N}$, as the case for $\tilde{h}_{1,N}$  is similar.
   Suppose that $\psi_N=0$ on a
    line segment 
    $$L:=\{(x_1^*, x_2): \underline{x}_2< x_2<\bar{x}_2\}.$$
   Set $ x_2^* = (\underline{x}_2 + \bar{x}_2)/2$,  $x^* = (x_1^*, x_2^*)$, 
   and 
   consider two cases:
   \begin{enumerate}
       \item  There exists a $r >0$ such that $\psi_N>0$ in  at least one of the sets
       $$ U_r^\pm := B_{r}(x^*) \cap \{ \pm( x_1 - x_1^*) >0 \}.$$
       \item For each  $r > 0$ (small)  case (1)  does not hold. 
       
   \end{enumerate}

   
   
In case (1)
   define   $W(x_1, x_2):=\psi_N(x_1, x_2)-\psi_N(x_1, x_2-\tau)$, and 
   for definiteness  assume $\psi_N > 0$ in $U^-$.
   Then    
      there exist a sufficiently small $\tau >0$  such that
\begin{equation*}
\left\{
\begin{aligned}
& \Delta W(x)=12Q(\kappa(\psi_N(x_1, x_2-\tau))-\kappa(\psi_N(x)))\leq 0 ,\quad &&\text{in}\,\, U_r^-,\\
& W(x) \geq 0 , \quad &&\text{in}\,\, U_r^-,\\
					& W(x^*)=0,
\end{aligned}
\right.
\end{equation*}
			where    the non-negativity of $W$ follows from the monotonicity of $\psi_N$ in direction $x_2$, see  Lemma  \ref{mon}. 
			It follows from Hopf's lemma (\hspace{1sp}\cite{evans,Hopf}) 
   that ${\partial _{x_1}W}(x^*)<0$. 
   On the other hand 
$$0= {\partial_{x_1} \psi_N}(x^*)-{\partial_{x_1} \psi_N}(x_1^*, x_2^*-\tau)={\partial_{x_1} W}(x^*) < 0,$$
which is a contradiction with $\nabla \psi_N(x^*)=0$. This implies that $x_2=\tilde{h}_{0,N}(x_1)$ has no jump discontinuity, in  case (1).

In case (2), we consider the solution only in $B_{r_0} (x^*)$ for some  small enough $r_0$, and  redefine it by setting 
$\tilde \psi_N := \psi_N \chi_{U_{r_0}^-} $ in $B_{r_0} (x^*)$. Since $\psi_N = |\nabla \psi_N | = 0$
on the line segment $L$, we must have that $\tilde \psi_N$ 
 solves the same local obstacle problem in $B_{r_0} (x^*)$,  with  $U_{r_0}^+ \subset \{ \tilde \psi_N  = 0 \}$. 
 Invoking the assumption that for each $r_j> 0$,   $U_{r_j}^-$ contains free boundary points, $z^j \in \check \Gamma_{0,N}$, we may  scale $\tilde \psi_N$ at $z^j$ with distance 
 $t_j =   x_1^* - z_1^j > 0 $, and define
$$
\psi^\flat_{N,j} = \frac{\tilde \psi_N (t_j x + z^j)}{t_j^2}.
$$
Since $z^j \in \check \Gamma_{0,N}$ we must have 
$(0,0) \in  \frac1j (\check\Gamma_{0,N} - z^j)$, and hence $\tilde \psi_{N,j} (0,0) = 0$.

By standard free boundary regularity and non-degeneracy \cite[Sections 2.1 and  3.1]{PSU} we have $\psi^\flat_{N,j}$ converges to a global solution $\psi_0 $ in the entire plane. Using the fact that $U_{r_0}^+ \subset \{ \tilde \psi_N  = 0 \}$ we must have 
$$\{ \psi_0 = 0\} \supset   \lim_j \frac{1}{t_j} (U^+_{r_0} - z^j) = \{x_1 \geq  1\}, \quad \hbox{and }\  \psi_0 (0,0) = 0.$$

By the classification of global solutions of the obstacle problem (see  \cite[Section 5.1]{PSU})
we know that $\psi_0$ is one dimensional. Consequently  $\psi_0 (x) = 3Q\left(\max(1 - x_1, 0 )\right)^2$  and therefore  it can not be zero at $(0,0)$. This  contradiction implies that $x_2=\tilde{h}_{0,N}(x_1)$ cannot be discontinuous as stated in  case  (2). 

   The proof of the lemma is now completed.
       \end{proof}

As long as Proposition \ref{Thm:C11} is at our disposal, we can delve into examining the local regularity of the free boundary. It's worth noting that the analysis takes a slightly different course since the free boundary may come into contact with the fixed boundary $\partial \Omega_N$. In order to use the theory for free boundary problem achieved in the last several decades, we need to lay some groundwork to tailor these findings to our specific situation.

Caffarelli's groundbreaking work \cite{Caf1} asserts that the free boundary is $C^{1}$ within the interior of $\Omega_N$ and close to points where the coincidence set ${\psi_N =0 }$ (similarly ${\psi_N =Q }$) possesses a positive volume density. Additionally,  the classical obstacle problem can give rise to singular points along the free boundary \cite{Sch,Sakai}. However, our problem possesses unique characteristics that supply additional information to help prevent the occurrence of singular points.

In order to prove positive volume density of the coincidence set for interior free boundary points,
 we may use the geometry of the nozzle and boundary values, that  allows us to apply the  well-known sliding method to compare the solutions with a translated version of the solution, and derive monotonicity
 properties for the solution in directions close to ${\bf e}_2=(0,1)$. More precisely, we have the following lemma.
\begin{lemma}\label{lemdirder}
There exists a $\sigma>0$ such that if $  |{\bf e}- {\bf e}_2|<\sigma$, one has
\begin{equation}\label{monotone1}
 \partial_{\bf e} \psi_N \geq 0 \qquad \text{in}\ \ \Omega_N.
\end{equation}
\end{lemma}
\begin{proof}
Extend $\psi_N$ as follows
	\begin{equation}
		\tilde\psi_N(x_1,x_2)=\left\{
		\begin{aligned}
			&Q\quad &&\text{if}\,\,x_2\geq h_1(x_1),\\
			&\psi_N(x_1,x_2)\quad &&\text{if}\,\, (x_1,x_2)\in \Omega_N,\\
				&0\quad &&\text{if}\,\,x_2 \leq h_0(x_1).
		\end{aligned}\right.
	\end{equation}	
For simplicity, we still denote it by $\psi_N$. We denote by  $\psi_{N,\be,t} (x) = \psi_N (x + t {\bf e}) $ for ${\bf e}=(e_1,e_2)$ satisfying $|{\bf e}- {\bf e}_2|<\sigma$, and such that the graph
 $\{x_2 = h_{i,N} (x_1 )\}$ does  not touch $\{x_2 +te_2  = h_{i,N} (x_1 + te_1)\}$, $i=0,1$ for $t>0$. 
 Let
 $$\Omega^t_{N, {\bf e}}=\{x: x+t{\bf e}\in \Omega_N\}.$$

Clearly, $\psi_{N,\be, t}$ solves the same problem in $\Omega^t_{N, {\bf e}}$, i.e.,
 $$
	\Delta \psi_{N,\be, t}=6Q(1-2\kappa(\psi_{N, \be,t}))\chi_{\{\psi_{N, \be, t}>0\}}\quad \text{in}\,\, \Omega^t_{N, {\bf e}}.
$$
We now take $t_0$ the smallest positive number such that $\Omega_N\cap \Omega^{t_0}_{N, {\bf e}}= \emptyset$.
Now moving up $\Omega^{t_0}_{N, {\bf e}}$ slightly, by taking $t= t_0 -\epsilon$ for $\epsilon$ small, the two nozzles will have an intersection with small width. For such small values of $\epsilon$ we obviously have
\begin{equation*}
\psi_{N,\be, t_0 -\epsilon} \geq \psi_N \ \ \text{   in }\  \Omega_N\cap \Omega^{t_0}_{N, {\bf e}}.
\end{equation*}
This depends on continuity of $\psi_N$, and  
the boundary values of these functions; observe that  the lateral boundary is monotone non-decreasing  in $x_2$-direction.

Now let us take $t_1$ to be the smallest number for which
$\psi_{N,\be, t_1} \geq \psi_N$. We aim to show that $t_1 = 0$. We prove this property by contradiction argument.
Suppose $t_1 > 0$, so that $\Omega_N\cap \Omega^{t_1}_{N, {\bf e}}  \neq \emptyset$.
Define
 \begin{equation*}
\mfu = \psi_{N,\be, t_1}- \psi_N \ \ \text{   in }\  \Omega_N\cap \Omega^{t_1}_{N, {\bf e}}.
\end{equation*}
Obviously there is a point $z$ such that $\mfu (z) = 0$, otherwise we could move the nozzle upward in ${\bf e}$-direction slightly (say with $\delta$ distance) and still keep the inequality  $\psi_{N,\be, {t_1 -\delta }} \geq \psi_N$, which contradicts the minimality of $t_1$.

Two different possibilities may appear: either $z \in \{ \psi_N > 0\}$ or $z \in \check\Gamma_{0,N}$. Now we consider these two cases one by one.

{\it Case 1. $z \in \{ \psi_N > 0\}$.}
From the equations for $\psi_N$, $\psi_{N,\be, {t_1 }}$ and that $\kappa$ is monotone non-decreasing  we have that
$\mfu $ is  superharmonic in its domain of positivity, i.e.
$$
\Delta \mfu = 12Q\left(\kappa (\psi_N) - \kappa (\psi_{N,\be, {t_1 }}) \right) \leq 0 \qquad \hbox{in } \Omega_N\cap \Omega^{t_1}_{N, {\bf e}}.
$$
By the minimum principle $\mfu$ can not attain an interior minimum, and hence $z \in \{ \psi_N > 0\}$ cannot occur.

{\it Case 2. $z \in \check\Gamma_{0,N}$.} We see immediately that $z \in  \check\Gamma_{0,N} \cap \check\Gamma^{t_1}_{0,N}$, where  $ \check\Gamma^{t_1}_{0,N} :=\{(x_1, h_{0,N}(x_1+te_1) -te_2): (x_1,x_2)\in \check{\Gamma}_{0,N}\}$ is the associated shift of $\check\Gamma_{0,N}$.
Therefore,  by  Hopf's boundary point lemma,  $z$ cannot belong to any smooth part of the free boundary. Suppose now $z$ belongs to a singular free boundary point. By results
in \cite[Theorem 5]{Caff-Riv},
we know that the free boundary at a singular point has the following  {nice {\bf feature}:}

\medskip
\noindent
{\it There are two  non-intersecting $C^1$-curves $G_i$, through $z$, and such that the coincidence set
$\{\psi_N = 0\} \cap B_{r_0} (z)) $ is captured between these two curves; here, $r_0$ depends on the distance of the singular point to the  boundary of the nozzle.
Hence  $G_i$ $(i=0,1)$
is represented
by   $\partial \{\mfg_i(x) >0\} $ for $\mfg_i  \in C^1 (B_{r_0} (z))$, such that $\mfg_i(z) = 0 $,
$B_{r_0}  (z) \cap \{\mfg_0(x) >0\}  \cap \{\mfg_1(x) >0\} = \emptyset $, and  $\{\psi_N = 0\}  \cap B_{r_0}  (z) \subset  \{\mfg_0(x) <0\}  \cap \{\mfg_1(x) <0\} $.}

\medskip


In particular, the  superharmonic function $\mfu$ is non-negative in the set $\{g_1 > 0\}$.
Since $z \in \check \Gamma_{0,N}$, there exists a constant $C_0$ such that
 \begin{equation}\label{quadratic}
\sup_{B_r(z) }\mfu (x)  \leq C_0 r^2 , \qquad  \text{for any}\  r < r_0 ,
\end{equation}
due to bounds on the second derivatives of $\psi_N$.

Now, we consider a cone with  opening $\pi - \epsilon$, with vertex at $z$, and  pointing inside
$\{\mfg_0 > 0\} \subset \{\psi_N > 0\}$. Let us truncate this cone with the ball $B_{r_0} (z)$ and denote it by  $K_{r_0}(z)$, so that  $K_{r_0}(z)\subset \{ \psi_N > 0\}$.
 In fact, in the cone $K_{r_0}(z)$,  there exists
 a  positive homogeneous harmonic function $W$ with zero boundary values on the boundary of the cone such that
  \begin{equation}\label{subquadratic}
 \sup_{B_r (z) } W \geq c_0 r^{2-\delta_\epsilon} \qquad \text{for any} \ \ r < r_0 ,
 \end{equation}
 for some positive $\delta_\epsilon$.

 Since  $\mfu > 0$ is superharmonic on $\partial K_{r_0} (z) $, for some large $M$, one has
  \begin{equation}
M \mfu  \geq W  \qquad \hbox{on } \partial K_{r_0} (z) .
 \end{equation}
 Then by comparison principle, one derives
  \begin{equation}\label{comparison}
M \mfu  \geq W  \qquad \hbox{in }  K_{r_0} (z).
 \end{equation}
This leads to a contradiction with \eqref{quadratic}-\eqref{subquadratic}. Hence the proof of the lemma is completed.
\end{proof}

It follows from \cite{BS2003}
we  know that  the free boundary in a universal neighborhood of the touching points of the free boundary with the nozzle, is uniformly a $C^1$-graph. Therefore, to prove $C^1$-regularity of $\tilde{h}_{i,N}$ we need to 
prove that the free boundary touches the fixed one tangentially, in a $C^1$-fashion.
This is the following main regularity theory for the boundary of the non-stagnant region, the freee boundary.

\begin{theorem}\label{regularity}
Assume 
$\Gamma_{i,N}$ ($i=0,1$) is at least uniformly $C^{1,Dini}$.
Then  $\tilde{h}_{i,N}$ is a $C^1$ in the interval   
$x_1 \in  [-N +1, N -1]$.
\end{theorem}

It   suffices to prove this theorem for one of the graphs, say $\tilde \Gamma_{0, N}$.
The proof of this theorem follows in a similar way as that of
\cite[Theorems C and D]{SU_Duke}. We shall sketch the argument below without too many details.
For the argument to work, one needs quadratic growth as in Lemma 5.7, which only uses $C^{1,Dini}$ regularity of the boundary of the nozzle. 
The main argument of the $C^1$-regularity of $\tilde{\Gamma}_{i,N}$ uses indirect argument, that we explain by some hand-waving here. One may refer to \cite{SU_Duke} for more details.

Let $z $ be a point on the solid boundary where free boundary point touches. One can 
 rotate locally the nozzle such that 
 the inward normal vector to the boundary of the nozzle at $z$ is pointing in the $x_2$-direction. 
 We want to prove that for each $\epsilon > 0$, there is $r_\epsilon$ such that 
 a cone of opening approximately  $  (\pi - \epsilon ) $,   truncated by $B_{r_\epsilon } (z) $,  is inside the support of the function $\psi_N$. More precisely, we have the following lemma.
 
 \begin{lemma} 
 Given  $\epsilon >0 $, there
 exists $r_\epsilon$  such that
 for any solution  $\psi_N$ of the problem \eqref{pbpsiN}, and any  touching point $z$, 
after suitable rotation of $\Gamma_{0,N}$, we have
\begin{eqnarray}
    \label{cap}
    \check \Gamma_{0,N}\cap B_{r_\epsilon} (z)  \subset   (B_{r_\epsilon} (z)\setminus K^z_\epsilon )  ,
\end{eqnarray}
where 
$K^z_\epsilon:= \{x: \ x_2  >  \epsilon |x_1 - z_1| + z_2 \}. $
\end{lemma}

\begin{proof}
The proof of this lemma follows a similar approach to the proof in \cite[Theorems C and D]{SU_Duke}. We elucidate the argument in a more geometric manner. However, we specifically focus on the case $\check \Gamma_{0,N}$ in this discussion.

If the assertion of the lemma is false, it implies the existence of sequences $\psi_{N_j}$ (representing solutions), $z^j$ (denoting touching points), and
 \begin{equation}\label{xj}
       x^j \in   \check \Gamma_{0,N_j}     \cap   K^{z^j}_{\epsilon}   , \qquad \hbox{(after suitable rotation of $\Gamma_{0,N}$)} .
        \end{equation} 
By scaling $\psi_{N_j}$ with $r_j = |x^j - z^j|$, we define
$$
 \psi^\flat_{N_j} (x) = \frac{\psi_{N_j}( z^j + r_j x) }{r_j^2} \quad \text{and}\quad \Omega^\flat_{N_j}=\frac{1}{r_j} (\Omega_{N_j} - z^j).
$$
As demonstrated in the proof of Lemma \ref{lem5.7},one has
$$
\Delta  \psi^\flat_{N_j} = f(r_j^2\psi^\flat_{N_j} (x))
\chi_{\{ \psi^\flat_{N_j}  > 0\}}\quad \text{in}\quad \Omega^\flat_{N_j}.
$$

 Similarly, $\psi^\flat_{N_j}$ converges to a global solution $\psi_0$ in the upper half-plane that satisfies
$$\Delta \psi_0 = 6Q \chi_{\{ \psi_0 > 0 \} },$$
and due to the quadratic growth property and non-degeneracy, we conclude that $\psi_0 \not \equiv 0 $ and exhibits optimally quadratic growth. Given that $\psi_0 \geq 0$, according to the classification of global solutions in the half-space ${x_2 > 0}$ (refer to \cite[Theorem B]{SU_Duke}), we deduce that $\psi_0 = 3Qx_2^2$.

Returning to \eqref{xj}, we observe that the free boundary point $x^j$ is now scaled to
$$ \tilde x^j =   \frac{1}{r_j}(x^j  -z^j) \in   
\frac{1}{r_j} (\check \Gamma_{0,N}  -z^j)   \cap   K^{0}_{\epsilon }.$$
Thus $\tilde x^j_2 > \epsilon $, with $\tilde x^j \in \partial B_1$. Consequently, its limit $\tilde x^0 \in \partial B_1$ satisfies $\tilde x^0_2 > \epsilon $.

As $\tilde x^0$ is also a free boundary point, it follows that
$$0 < 3Q\epsilon^2<  3Q(\tilde x^0_2)^2 =  \psi_0 (\tilde x^0) = 0.$$
This contradiction establishes the conclusion of the lemma.
\end{proof}

\begin{proof}[Proof of Theorem \ref{regularity}]
The above lemma shows that the free boundary touches the fixed one in a tangential fashion. In other words, 
inverting  $r_\epsilon $ (as a function of $\epsilon$), that we call 
$\omega (r) \searrow 0$,
we see that this defines a modulus of continuity of how 
the free boundary touches the fixed one. More precisely, after suitable rotation of $\Gamma_{0,N}$, we have
$$  
   ( \check \Gamma_{0,N}     \cap  B_{r} )
    \subset  (   B_{r} \setminus K^z_{\omega(r) }).
$$

The remainder of the argument for proving  Theorem \ref{regularity} closely parallels the proofs for  \cite[Theorems C and D]{SU_Duke}, with the important distinction that in \cite[Theorem C]{SU_Duke}, the fixed boundary is assumed to be flat, and in Theorem D, the fixed boundary is assumed to be $C^3$. However, in the proof of \cite[Theorem D]{SU_Duke}, the requirement of $C^3$ smoothness is only necessary to achieve optimal smoothness for solutions that change sign, which is in contrast to our problem dealing with non-negative functions. As the reasoning in the proofs remains essentially the same, we omit the details here, and conclude that the above argument proves Theorem \ref{regularity}.
\end{proof}


In fact the boundary of the stagnation region inside the nozzle can have higher regularity and belongs to $C^{1,\beta}$ for some $\beta>0$. This is a consequence of  Lipschitz regularity of $f(\psi_N(x))$ near the interior free boundary point, a proof of which appears to be slightly more intricate and relies on the $C^1$-smoothness of the stagnation boundary.

\begin{lemma}\label{lemnew} Assume 
$\Gamma_{i,N}$ ($i=0,1$) is at least uniformly $C^{1,Dini}$.
There exists a sufficiently small $\delta_0>0$ such that for any $x\in \tilde{\Omega}_N$ satisfying $$d(x):=dist(x,{\check\Gamma}_{0,N})<\delta_0 < 2\delta_0 < dist(x,{\Gamma}_{0,N}),$$ 
one has 
\begin{equation}\label{nondeg}
    \psi_N (x) \geq \frac{Q}{100}  d^2(x).
\end{equation}
\end{lemma}

\begin{proof}
We prove the lemma by way of  contradiction. Suppose that \eqref{nondeg} is not true, for any $j\in\mathbb N$, there exists $x_j$ satisfying $d(x_j)<1/j$ such that
\begin{equation}\label{no}
\psi_N(x_j)< \frac{Q}{100} d^2(x_j). 
\end{equation}
Choose $y_j\in {\check\Gamma}_{0,N}$ such that
$d(x_j)=|x_j-y_j|$, and 
define
$$ r_j=d(x_j), \quad   \psi^\flat_{N,j} (x) = \frac{\psi_N(r_j x+y_j)}{r_j^2},\quad \text{and} \ \ \Omega^\flat_{N,j}=\frac{1}{r_j}( \tilde{\Omega}_N - y_j).$$
It follows that
$$\Delta \psi^\flat_{N,j}(x)=f(r^2_j\psi^\flat_{N,j} (x))\chi_{\{\psi^\flat_{N,j} (x)>0\}} \ \  \ \text{in}\ \ \Omega^\flat_{N,j}.$$
By elliptic regularity, over a subsequence,  $\psi^\flat_{N,j}$  converges to a global solution $\psi_0$ in the upper half plane.
Let us rotate the nozzle such that the inward normal vectors to ${\check\Gamma}_{0,N}$ are pointing in the $x_2$ direction. Then the limit function $\psi_0$ satisfies
\begin{equation}
	\left\{
	\begin{aligned}	
&	\Delta \psi_0(x)=6Q\chi_{\{\psi_0 (x)>0\}},\quad &&\text{in}\,\, \{x_2>0\},\\
&\psi_0(x)\geq 0,\quad &&\text{in}\,\, \{x_2>0\},\\
&\psi_0(x)= 0,\quad &&\text{on}\,\, \{x_2=0\}.\\
\end{aligned}
\right.
\end{equation}

From the classification of global solutions (cf. \cite{PSU}), one can obtain that
\begin{equation}\label{limit}
\psi_0(x)={3Q}[(x\cdot {\bf e}_2)^+]^2.
\end{equation}
Thus, one has
\begin{equation*}
\psi_0({\bf e}_2)= {3Q}.
\end{equation*}

On the other hand, by \eqref{no} we have 
\begin{equation}\label{limit2}
\psi_0({\bf e}_2)\leq \frac{Q}{100}.
\end{equation}
This leads to a contradiction. The proof of the lemma is completed.
\end{proof}

\begin{lemma}\label{Lip-f}
For any $x\in  \check \Gamma_{0,N}\cap \Omega_{N-1} $, there exists $\delta_0>0$ such that  $f(\psi_N(x))\in C^{0,1}(B_{\delta_0}(x)\cap \tilde\Omega_N)$.
\begin{proof}
It follows from Lemmas \ref{lem5.7} and \ref{lemnew} that there exists a $\delta>0$ such that for any $x\in \tilde{\Omega}_N$ satisfying $d(x):=dist(x,\check\Gamma_{0,N})<\delta$, one has
	\begin{eqnarray}\label{S1}
    \frac{1}{100}  d^2(x)\leq \psi_N (x) \leq c d^2(x),
	\end{eqnarray}	 
	for some constant $c>0$.
	
To prove the lemma, it suffices to show that for any $x, y\in \tilde{\Omega}_N$, if $|x-y|<\delta$, $x\neq y$, $d(x)<\delta$ and $d(y)<\delta$, then there exists a constant $C$ such that 
\begin{eqnarray}\label{lip}
\frac{|f(\psi_N(x))-f(\psi_N(y))|}{|x-y|}\leq C.
\end{eqnarray} 
It follows from \eqref{psi} that 
\begin{eqnarray}\label{est1}
|\kappa(\psi_N(x))|\leq \sqrt{\frac{\psi_N(x)}{Q}}\quad \text{and}\quad |\kappa(\psi_N(y))|\leq \sqrt{\frac{\psi_N(y)}{Q}}.
\end{eqnarray}

We first claim that	there exists a constant $C_0$ such that
\begin{eqnarray}\label{sqrt}
|\kappa(\psi_N(x))-\kappa(\psi_N(y))|\leq\frac{C_0|\psi_N(x))-\psi_N(y)|}{\sqrt{\psi_N(x)}+\sqrt{\psi_N(y)}}.
\end{eqnarray}
We may assume that $\psi_N(x)>\psi_N(y)$. There are two possibilities:

{\it (i)} $0\leq \psi_N(y) \leq \frac{\psi_N(x)}{2}$;

{\it (ii)} $\frac{\psi_N(x)}{2}<\psi_N(y) < \psi_N(x)$.
In case ${\it (i)}$, it follows from \eqref{est1} that
 \begin{eqnarray*}
 |\kappa(\psi_N(x))-\kappa(\psi_N(y))|\leq |\kappa(\psi_N(x))|\leq \sqrt{\frac{\psi_N(x)}{Q}}\leq \frac{C_0|\psi_N(x))-\psi_N(y)|}{\sqrt{\psi_N(x)}+\sqrt{\psi_N(y)}}.
 \end{eqnarray*}
 In case ${\it (ii)}$, one can derive that there exists $\theta\in(0,1)$ such that
 \begin{eqnarray*}
 |\kappa(\psi_N(x)-\kappa(\psi_N(y)|&=& |\kappa'(\theta\psi_N(x))+(1-\theta)\psi_N(y))||\psi_N(x)-\psi_N(y)|\\ &
 \leq & \frac{C_1|\psi_N(x))-\psi_N(y)|}{\sqrt{\theta\psi_N(x)+(1-\theta)\psi_N(y)}}\\
 &\leq& \frac{C_0|\psi_N(x))-\psi_N(y)|}{\sqrt{\psi_N(x)}+\sqrt{\psi_N(y)}}.
 \end{eqnarray*}

In order to prove \eqref{lip}, we first have
\begin{equation}\label{eq5.51}
\begin{aligned}
\frac{|f(\psi_N(x))-f(\psi_N(y))|}{|x-y|}=&\frac{12Q|\kappa(\psi_N(x)-\kappa(\psi_N(y)|}{|x-y|}\\ 
\leq&\frac{C_0|\psi_N(x)-\psi_N(y)|}{\sqrt{\psi_N(x)}+\sqrt{\psi_N(y)}}\cdot\frac{1}{|x-y|}.
\end{aligned}
\end{equation}
We fix a positive constant $\epsilon_0>0$ and consider two cases:

{\it Case (i)} $|x-y|\leq \epsilon_0 (d(x)+d(y))$. It follows from \eqref{eq5.51} that there exists $\tau\in (0,1)$ such that
\begin{eqnarray}
\frac{|f(\psi_N(x))-f(\psi_N(y))|}{|x-y|} \leq \frac{|\nabla \psi_N(\tau x+(1-\tau)y)|}{\sqrt{\psi_N(x)}+\sqrt{\psi_N(y)}} \leq\frac{C_2d({\tau x+(1-\tau)y})}{{d(x)+d(y)}}\leq C_3,
\end{eqnarray}
where Proposition \ref{Thm:C11} has been used.

{\it Case (ii)} $|x-y|> \epsilon_0 (d(x)+d(y))$. Lemma \ref{lemnew} implies
\begin{eqnarray}
\frac{|f(\psi_N(x))-f(\psi_N(y))|}{|x-y|}\leq\frac{C_4(d^2(x)+d^2(y))}{(d(x)+d(y))^2}\leq C_5.
\end{eqnarray}
This completes the proof of the lemma. 	
	\end{proof}
	\end{lemma}

\begin{theorem}
    If $\tilde{h}_{0, N}(x_1)\neq h_{0,N}(x_1)$ for $x_1\in [-N+1, N-1]$, then there exists a $\delta>0$ such that $\tilde{h}_{0, N}\in C^{1, \beta}(x_1-\delta, x_1+\delta)$ for all $\beta>0$.
\end{theorem}

\begin{proof}
Theorem  \ref{regularity} implies that the free boundary is $C^1$, and hence Lemma \ref{Lip-f} implies that $\tilde f(x):= f (\psi_N (x))$ is a Lipschitz function in the vicinity of the free boundary in the nozzle.

From the work of \cite{KN}, it follows that close to such points, the free boundary is $C^{1,\beta}$ for all $\beta \in (0,1)$. We also remark that the $C^{1,\beta}$-norm depends on the distance to the closest touching point between the fixed and the free boundary.
\end{proof}

\begin{remark}
    Note that $f(\psi(\cdot))\in W^{1,p}$ for $p>1$. It follows from the recent result in \cite{Ros-Oton} that the free boundary is also $C^{1,\beta}$ for some $\beta>0$, when staying away from the touching points.
\end{remark}

 \subsection{Existence of stagnation points}\label{secexistfbpoint}
In this subsection, we prove that $\tilde{h}_{0, N}\neq h_0$ when the nozzle satisfies some geometric condition, which implies the existence of stagnation points.
\begin{pro}
For $\tilde{h}_{0,N}$ defined in \eqref{defh0N} and $h_0(\bar{x}_1)<0$ for some $\bar{x}_1\in (-N,N)$, one has $\tilde{h}_{0,N}(\bar{x}_1)\geq 0$ if one of the following conditions holds, 
\begin{enumerate}
    \item[(i)] $h_1(x_1)=1-h_0(x_1)$ for $x_1\in \mathbb{R}$ (cf. Figure \ref{ex}(a));
    \item[(ii)] $h_1(x_1)\equiv 1$ for $x_1\in \mathbb{R}$ (cf. Figure \ref{ex}(b)).
\end{enumerate}
\end{pro}
\begin{figure}[h]
    \centering

 
\tikzset{
pattern size/.store in=\mcSize, 
pattern size = 5pt,
pattern thickness/.store in=\mcThickness, 
pattern thickness = 0.3pt,
pattern radius/.store in=\mcRadius, 
pattern radius = 1pt}
\makeatletter
\pgfutil@ifundefined{pgf@pattern@name@_wxp972546}{
\pgfdeclarepatternformonly[\mcThickness,\mcSize]{_wxp972546}
{\pgfqpoint{0pt}{0pt}}
{\pgfpoint{\mcSize+\mcThickness}{\mcSize+\mcThickness}}
{\pgfpoint{\mcSize}{\mcSize}}
{
\pgfsetcolor{\tikz@pattern@color}
\pgfsetlinewidth{\mcThickness}
\pgfpathmoveto{\pgfqpoint{0pt}{0pt}}
\pgfpathlineto{\pgfpoint{\mcSize+\mcThickness}{\mcSize+\mcThickness}}
\pgfusepath{stroke}
}}
\makeatother

 
\tikzset{
pattern size/.store in=\mcSize, 
pattern size = 5pt,
pattern thickness/.store in=\mcThickness, 
pattern thickness = 0.3pt,
pattern radius/.store in=\mcRadius, 
pattern radius = 1pt}
\makeatletter
\pgfutil@ifundefined{pgf@pattern@name@_vf6rtsu87}{
\pgfdeclarepatternformonly[\mcThickness,\mcSize]{_vf6rtsu87}
{\pgfqpoint{0pt}{0pt}}
{\pgfpoint{\mcSize+\mcThickness}{\mcSize+\mcThickness}}
{\pgfpoint{\mcSize}{\mcSize}}
{
\pgfsetcolor{\tikz@pattern@color}
\pgfsetlinewidth{\mcThickness}
\pgfpathmoveto{\pgfqpoint{0pt}{0pt}}
\pgfpathlineto{\pgfpoint{\mcSize+\mcThickness}{\mcSize+\mcThickness}}
\pgfusepath{stroke}
}}
\makeatother

 
\tikzset{
pattern size/.store in=\mcSize, 
pattern size = 5pt,
pattern thickness/.store in=\mcThickness, 
pattern thickness = 0.3pt,
pattern radius/.store in=\mcRadius, 
pattern radius = 1pt}
\makeatletter
\pgfutil@ifundefined{pgf@pattern@name@_a9yuw1scm}{
\pgfdeclarepatternformonly[\mcThickness,\mcSize]{_a9yuw1scm}
{\pgfqpoint{0pt}{0pt}}
{\pgfpoint{\mcSize+\mcThickness}{\mcSize+\mcThickness}}
{\pgfpoint{\mcSize}{\mcSize}}
{
\pgfsetcolor{\tikz@pattern@color}
\pgfsetlinewidth{\mcThickness}
\pgfpathmoveto{\pgfqpoint{0pt}{0pt}}
\pgfpathlineto{\pgfpoint{\mcSize+\mcThickness}{\mcSize+\mcThickness}}
\pgfusepath{stroke}
}}
\makeatother

 
\tikzset{
pattern size/.store in=\mcSize, 
pattern size = 5pt,
pattern thickness/.store in=\mcThickness, 
pattern thickness = 0.3pt,
pattern radius/.store in=\mcRadius, 
pattern radius = 1pt}
\makeatletter
\pgfutil@ifundefined{pgf@pattern@name@_f6nkgilrb}{
\pgfdeclarepatternformonly[\mcThickness,\mcSize]{_f6nkgilrb}
{\pgfqpoint{0pt}{0pt}}
{\pgfpoint{\mcSize+\mcThickness}{\mcSize+\mcThickness}}
{\pgfpoint{\mcSize}{\mcSize}}
{
\pgfsetcolor{\tikz@pattern@color}
\pgfsetlinewidth{\mcThickness}
\pgfpathmoveto{\pgfqpoint{0pt}{0pt}}
\pgfpathlineto{\pgfpoint{\mcSize+\mcThickness}{\mcSize+\mcThickness}}
\pgfusepath{stroke}
}}
\makeatother
\tikzset{every picture/.style={line width=0.75pt}} 

\begin{tikzpicture}[x=0.7pt,y=0.75pt,yscale=-1,xscale=1]

\draw [line width=0.75]  [dash pattern={on 4.5pt off 4.5pt}]  (20.33,63) -- (328.33,63) ;
\draw [line width=0.75]  [dash pattern={on 4.5pt off 4.5pt}]  (24.33,173) -- (328.33,173) ;
\draw [line width=1.5]    (65,63) .. controls (89.33,63) and (96,65) .. (118,48) .. controls (140,31) and (140,57) .. (162,69) .. controls (184,81) and (197,55) .. (218,47) .. controls (239,39) and (256,82) .. (290,85) ;
\draw [line width=1.5]    (20.33,63) -- (65,63) ;
\draw [line width=1.5]    (290,85) -- (328.33,85) ;
\draw [line width=1.5]    (24.33,173) -- (66,173) ;
\draw  [pattern=_wxp972546,pattern size=8.399999999999999pt,pattern thickness=0.75pt,pattern radius=0pt, pattern color={rgb, 255:red, 7; green, 62; blue, 192}] (492.33,258.67) -- (520.67,258.67) -- (520.67,271.67) -- (492.33,271.67) -- cycle ;
\draw [line width=1.5]    (66,173) .. controls (115,171) and (118,210) .. (144,187) .. controls (170,164) and (171,157) .. (187,169) .. controls (203,181) and (213,203) .. (232,191) .. controls (251,179) and (253,156) .. (288,153) ;
\draw [line width=1.5]    (288,153) -- (326.33,153) ;
\draw [line width=1.5]  [dash pattern={on 5.63pt off 4.5pt}]  (25.33,119) -- (328.33,119) ;
\draw [color={rgb, 255:red, 208; green, 2; blue, 27 }  ,draw opacity=1 ][line width=1.5]    (66,173) .. controls (90.27,170.67) and (88.53,142.33) .. (106.27,163.17) .. controls (124,184) and (123.59,136.74) .. (141,159) .. controls (158.41,181.26) and (160.81,162.06) .. (173.77,161.67) .. controls (186.72,161.27) and (186.03,174.33) .. (199.27,171.67) .. controls (212.5,169) and (213,151) .. (224.31,164.29) .. controls (235.63,177.58) and (270.41,152.46) .. (288,153) ;
\draw [line width=1.5]    (362,63.33) -- (650.54,63.33) ;
\draw [line width=0.75]  [dash pattern={on 4.5pt off 4.5pt}]  (364.33,172.33) -- (651,172.33) ;
\draw [line width=1.5]    (364.33,172.33) -- (395.33,172.33) ;
\draw [line width=1.5]    (395.33,172.33) .. controls (444.33,170.33) and (447.33,209.33) .. (473.33,186.33) .. controls (499.33,163.33) and (498.67,147.33) .. (514.67,159.33) .. controls (530.67,171.33) and (538.67,161.33) .. (557.67,149.33) .. controls (576.67,137.33) and (582.33,155.33) .. (617.33,152.33) ;
\draw [line width=1.5]    (617.33,152.33) -- (650.54,152.33) ;
\draw [color={rgb, 255:red, 208; green, 2; blue, 27 }  ,draw opacity=1 ][line width=1.5]    (398.32,171.86) .. controls (422.59,169.52) and (417.86,141.67) .. (435.6,162.5) .. controls (453.33,183.33) and (452.33,139.83) .. (470.33,158.33) .. controls (488.33,176.83) and (494.2,155.13) .. (503.67,154.33) .. controls (513.13,153.53) and (516.32,166.66) .. (531.4,163.93) .. controls (546.48,161.21) and (559.28,140.41) .. (578.6,146.73) .. controls (597.92,153.06) and (599.92,153.46) .. (611.04,152) ;
\draw [draw opacity=0][pattern=_vf6rtsu87,pattern size=14.024999999999999pt,pattern thickness=0.75pt,pattern radius=0pt, pattern color={rgb, 255:red, 7; green, 62; blue, 192}][line width=1.5]    (70.58,175.8) .. controls (100.96,168.47) and (82.49,146.67) .. (102.88,164.67) .. controls (123.27,182.67) and (121.38,143.67) .. (139.38,162.17) .. controls (157.38,180.67) and (159.77,163.67) .. (167.88,164.67) .. controls (175.99,165.67) and (128.88,190.67) .. (130.88,194.67) .. controls (132.88,198.67) and (93.46,166.07) .. (78.58,171.4) ;
\draw [draw opacity=0][pattern=_a9yuw1scm,pattern size=14.774999999999999pt,pattern thickness=0.75pt,pattern radius=0pt, pattern color={rgb, 255:red, 7; green, 62; blue, 192}][line width=1.5]    (411.47,166.73) .. controls (441.85,159.4) and (417.48,148.33) .. (437.87,166.33) .. controls (458.25,184.33) and (452.76,151.07) .. (468.71,161.5) .. controls (484.67,171.93) and (480.55,164.53) .. (488.67,165.53) .. controls (496.78,166.53) and (458.21,190) .. (460.21,194) .. controls (462.21,198) and (421.55,165.8) .. (406.67,171.13) ;
\draw [draw opacity=0][pattern=_f6nkgilrb,pattern size=11.850000000000001pt,pattern thickness=0.75pt,pattern radius=0pt, pattern color={rgb, 255:red, 7; green, 62; blue, 192}][line width=1.5]    (191.16,176) .. controls (240.09,152.6) and (215.78,174.3) .. (230.89,170.2) .. controls (246,166.1) and (240.81,169.66) .. (248.67,169.44) .. controls (256.53,169.22) and (232.89,187.27) .. (226.49,191) .. controls (220.09,194.73) and (190.45,177.03) .. (196.89,172.2) ;
\draw [line width=0.75]  [dash pattern={on 4.5pt off 4.5pt}]  (64.73,27) -- (64.73,61.6) -- (64.73,152.6) -- (64.73,218.6) ;
\draw [line width=0.75]  [dash pattern={on 4.5pt off 4.5pt}]  (288.73,27) -- (288.73,61.6) -- (288.73,152.6) -- (288.73,218.6) ;
\draw [line width=0.75]  [dash pattern={on 4.5pt off 4.5pt}]  (396.4,26.67) -- (396.4,61.27) -- (396.4,152.27) -- (396.4,217.1) ;
\draw [line width=0.75]  [dash pattern={on 4.5pt off 4.5pt}]  (611.04,26.4) -- (611.04,61) -- (611.04,152) -- (611.04,218) ;

\draw (22.33,176.4) node [anchor=north west][inner sep=0.75pt]  [font=\scriptsize]  {$x_{2} =0$};
\draw (19.5,48.4) node [anchor=north west][inner sep=0.75pt]  [font=\scriptsize]  {$x_{2} =1$};
\draw (293,137.4) node [anchor=north west][inner sep=0.75pt]  [font=\scriptsize]  {$x_{2} =a$};
\draw (292,88.4) node [anchor=north west][inner sep=0.75pt]  [font=\scriptsize]  {$x_{2} =b$};
\draw (153.67,182.4) node [anchor=north west][inner sep=0.75pt]  [font=\small]  {$\Gamma _{0}$};
\draw (153.67,35.4) node [anchor=north west][inner sep=0.75pt]  [font=\small]  {$\Gamma _{1}$};
\draw (148.67,136.2) node [anchor=north west][inner sep=0.75pt]  [font=\small,color={rgb, 255:red, 208; green, 2; blue, 27 }  ,opacity=1 ]  {$x_{2} =\tilde{h}_{0}( x_{1})$};
\draw (531.67,257.67) node [anchor=north west][inner sep=0.75pt]  [font=\small] [align=left] {Stagnation region};
\draw (12,123.4) node [anchor=north west][inner sep=0.75pt]  [font=\scriptsize]  {$x_{2} =1/2$};
\draw (355.33,49.07) node [anchor=north west][inner sep=0.75pt]  [font=\scriptsize]  {$x_{2} =1$};
\draw (355.33,177.73) node [anchor=north west][inner sep=0.75pt]  [font=\scriptsize]  {$x_{2} =0$};
\draw (612.83,138.23) node [anchor=north west][inner sep=0.75pt]  [font=\scriptsize]  {$x_{2} =a$};
\draw (475.33,189.73) node [anchor=north west][inner sep=0.75pt]  [font=\small]  {$\Gamma _{0}$};
\draw (441.6,128.13) node [anchor=north west][inner sep=0.75pt]  [font=\small,color={rgb, 255:red, 208; green, 2; blue, 27 }  ,opacity=1 ]  {$x_{2} =\tilde{h}_{0}( x_{1})$};
\draw (35,218.07) node [anchor=north west][inner sep=0.75pt]  [font=\scriptsize]  {$x_{1} =-N$};
\draw (269.67,218.07) node [anchor=north west][inner sep=0.75pt]  [font=\scriptsize]  {$x_{1} =N$};
\draw (366.67,217.07) node [anchor=north west][inner sep=0.75pt]  [font=\scriptsize]  {$x_{1} =-N$};
\draw (591.67,221.73) node [anchor=north west][inner sep=0.75pt]  [font=\scriptsize]  {$x_{1} =N$};

\end{tikzpicture}

\caption{Existence of stagnation points}\label{ex}
\end{figure}
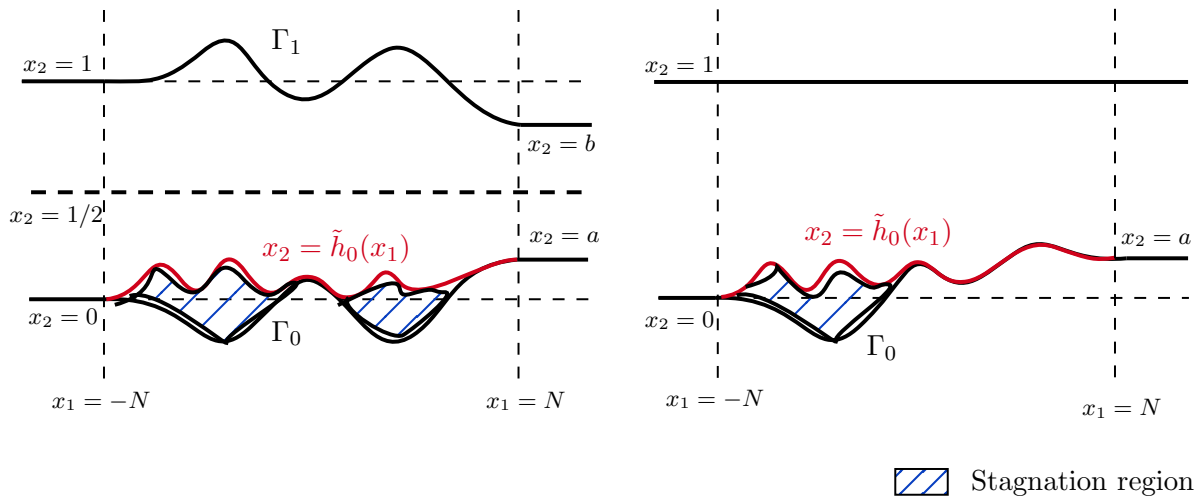
\begin{proof}
The proof relies on Corollary \ref{cor3.9}. 
Define
\begin{equation}\label{defvarphi1}
			\varphi_1(x_2)=\left\{			\begin{aligned}
				& \bar\varphi_1(x_2),\quad &&\text{if}\,\, x_2\in(0,1),\\
				& 0,\quad &&\text{if}\,\, x_2\leq 0,\\
                 & Q, \quad && \text{if}\,\, x_2\geq 1,
			\end{aligned}\right.
			\end{equation}
where $\bar{\varphi}_1$ is the one-dimensional solution defined in \eqref{defbarvarphi11}. Now we prove the two cases separately.

{\it Case (i).} Suppose that $h_1=1-h_0$. It follows from Lemma \ref{lem6.5} that the energy minimizer is unique. Hence we have $\psi_N(x_1,x_2) =Q-\psi_N(x_1, 1-x_2)$. 
Denote
\begin{eqnarray*}
\tilde{\psi}_N(x_1,x_2)=\min\{\varphi_1(x_2), \psi_N(x_1,x_2)\} \ \ \text{for}\ \ (x_1, x_2)\in \Omega_{N,-}
\end{eqnarray*}
and 
\begin{eqnarray*}
D_{N,-}=\{(x_1,x_2)\in \Omega_{N,-}: \varphi_1(x_2)<\psi_N(x_1,x_2)\},
\end{eqnarray*}
where
\begin{equation}
    \Omega_{N,-}=\{(x_1,x_2): (x_1,x_2)\in \Omega_N, x_2<1/2\}.
\end{equation}
Hence one also has $\tilde\psi_N(x_1,x_2) =Q-\tilde\psi_N(x_1, 1-x_2)$.

The straightforward computations yield
\begin{eqnarray}\label{energy5.42}
{\mathcal{E}}_N(\tilde{\psi}_N)&=&2\int_{\Omega_{N,-}} \frac{|\nabla \tilde\psi_N|^2}{2} + F(\tilde\psi_N)dx\\ \nonumber  &=&2\int_{D_{N,-}}\frac{|\varphi'_1|^2}{2} + F(\varphi_1)dx+2\int_{\Omega_{N,-}\setminus D_{N,-}}  \frac{|\nabla \psi_N|^2}{2} + F(\psi_N)dx\\ \nonumber
                          &\leq & 2\int_{-N}^N\int_{D_{N,-,x_1}} \frac{|\partial_{x_2}\psi_N|^2}{2} + F(\psi_N)dx_2dx_1 +2\int_{\Omega_{N,-}\setminus D_{N,-}}  \frac{|\nabla \psi_N|^2}{2} + F(\psi_N)dx\\ \nonumber
                           &\leq& 2\int_{\Omega_{N,-}}  \frac{|\nabla \psi_N|^2}{2} + F(\psi_N)dx={\mathcal{E}}_N({\psi}_N),
\end{eqnarray}
where $D_{N, -, x_1}=\{x_2:(x_1,x_2)\in D_{N, -}\}$ and Corollary \ref{cor3.9} has been used to get the first inequality.
This implies $\tilde{\psi}_N(x_1,x_2)=\psi_N(x_1,x_2)$ in $\Omega_{N,-}$.
Therefore, one deduces that
\begin{eqnarray*}
\psi_N(x_1,x_2)\leq \varphi_1(x_2) \ \ \ \text{in} \ \Omega_{N,-}.
\end{eqnarray*}
It follows from $h_0(\bar{x}_1)<0$ and the definition \eqref{defvarphi1} that
\begin{eqnarray}
\psi_N(\bar{x}_1,x_2)=0 \ \ \text{for} \ \ h_0(\bar{x}_1)\leq x_2\leq 0.
\end{eqnarray}
This implies $\tilde{h}_{0,N}(\bar{x}_1)\geq 0$.

{\it Case (ii).} Suppose that $h_1 \equiv 1$. Denote 
\begin{eqnarray*}
\tilde{\psi}_N(x_1,x_2)=\min\{\varphi_1(x_2), \psi_N(x_1,x_2)\} \ \ \text{for}\ \ (x_1, x_2)\in \Omega_{N}
\end{eqnarray*}
and 
\begin{eqnarray*}
D_{N}=\{(x_1,x_2)\in \Omega_{N}: \varphi_1(x_2)<\psi_N(x_1,x_2)\}.
\end{eqnarray*}
As same as the estimate \eqref{energy5.42},
one has
\[
\mathcal{E}_N(\tilde{\psi}_N )=\int_{\Omega_{N}} \frac{|\nabla \tilde\psi_N|^2}{2} + F(\tilde\psi_N)dx \leq \int_{\Omega_{N}}  \frac{|\nabla \psi_N|^2}{2} + F(\psi_N)dx={\mathcal{E}}_N({\psi}_N).
\]
Thus we have $\tilde{\psi}_N=\psi_N$ in $\Omega_N$, which implies
\[
\psi_N(x_1,x_2)\leq \varphi_1(x_2) \quad \text{in}\,\, \Omega_N.
\]
Therefore, it holds that 
\begin{eqnarray}
\psi_N(\bar{x}_1,x_2)=0 \ \ \text{for} \ \ h_0(\bar{x}_1)\leq x_2\leq 0.
\end{eqnarray}
This yields $\tilde{h}_{0,N}(\bar{x}_1)\geq 0$ and  completes the proof of the proposition.
\end{proof}

\begin{remark}
 Similarly,  if $h_0(x_1)\equiv 0$ for $x_1\in \mathbb{R}$ and $h_1(\bar{x}_1)>1$ for some $\bar{x}_1\in (-N, N)$, then $\tilde{h}_{1,N}(\bar{x}_1)\leq 1$.
\end{remark}


		\subsection{Existence of solutions in the whole nozzle and asymptotic behavior}
		In this subsection, we establish the existence of the solutions in the whole nozzle 
 $\Omega$. The key point is to study the monotonicity property of the normalized energy for approximate solutions.
		
		\begin{lemma}\label{lem6.9}
			Suppose that $\psi_{N}$ is an energy minimizer for ${\mathcal{E}}_N$ over ${\mathcal{K}}_N$. Then
			\begin{equation}\label{defzeta}
				\zeta(N)=	\int_{\Omega_{ N}}\frac{|\nabla\psi_{N}|^2}{2}+F(\psi_{N})dx  -N\bar{\mathcal{J}}_1-N\bar{\mathcal{J}}_{b-a}
			\end{equation}	is a decreasing function of $N$ as long as $N\geq L_0:= \text{max}\{\bar{L},-\underline{L}\}$, where $\mathcal{J}_d$ is defined in \eqref{1Denergy}. 		\end{lemma}
		\begin{proof}
			Given any $N_2>N_1\geq L_0$, define
			\[
			\tilde{\psi}_{N_2}=\left\{
			\begin{aligned}
				&\psi_{N_1}, \quad \ \ &&x\in \Omega_{N_1},\\
				&\bar\varphi_{b-a}(x_2-a), \quad &&x_1\in [N_1, N_2],\\
				&\bar\varphi_1, \quad \ \ &&x_1\in [-N_2, -N_1].
			\end{aligned}
			\right.
			\]
			Clearly, $\tilde{\psi}_{N_2}\Big|_{{\Omega}_{N_2}}\in {\mathcal{K}}_{N_2}$. Since $\psi_{N_2}\Big|_{{\Omega}_{N_2}}$ is an energy minimizer in ${\Omega}_{N_2}$, one has
			\begin{equation}
				\begin{aligned}
					&\int_{\Omega_{N_2}}\frac{|\nabla \psi_{N_2}|^2}{2}+F(\psi_{N_2}) dx\\
					\leq &
					\int_{\Omega_{N_2}}\frac{|\nabla \tilde{\psi}_{N_2}|^2}{2}+F(\tilde{\psi}_{N_2}) dx\\
					= & \int_{\Omega_{N_1}}\frac{|\nabla {\psi}_{N_1}|^2}{2}+F({\psi}_{N_1}) dx +(N_2-N_1)\bar{\mathcal{J}}_{b-a}+(N_2-N_1)\bar{\mathcal{J}}_1,
				\end{aligned}
			\end{equation}
			where $\bar{\mathcal{J}}_{b-a}$ is defined in \eqref{defbarJd}.
			Thus one has
			\begin{equation}
				\begin{aligned}
					\zeta(N_2)=&	\int_{\Omega_{N_2}}\frac{|\nabla\psi_{N_2}|^2}{2}+F(\psi_{N_2})dx  -N_2\bar{\mathcal{J}}_1-N_2\bar{\mathcal{J}}_{b-a} \\
					\leq&
					\int_{\Omega_{N_1}}\frac{|\nabla\psi_{N_1}|^2}{2}+F(\psi_{N_1})dx  -N_1\bar{\mathcal{J}}_1-N_1\bar{\mathcal{J}}_{b-a}=\zeta(N_1).
				\end{aligned}
			\end{equation}
			Therefore, $\zeta(N)$ is a decreasing function and the proof of the lemma is completed.
		\end{proof}
		
		In the following, we first establish the uniform estimate for the energy minimizers and then study the normalized energy for the limit of the approximate solutions. The asymptotic behavior is proved with the aid of Lemma \ref{lem6.9}.
		\begin{lemma}
			Let $\psi_N$ be an energy minimizer on the domain $\Omega_{N}$. Then
			\begin{equation}\label{lim}
				\psi_N\to \psi\quad \text{in}\,\, C^{1,\beta}_{\text{loc}}(\Omega) \ \text{         for some }\beta\in (0,1).
			\end{equation}
			{The limit function $\psi$ satisfies
				\begin{equation}\label{equ}
					\left\{
					\begin{aligned}
						&	\Delta\psi =f(\psi)\quad \text{in}\,\, \Omega\cap \{x: 0<\psi(x)<Q\},\\
						&	\psi=0\,\,\ \text{on}\,\,\ \Gamma_0,\quad \ \psi=Q\,\ \,\text{on}\,\ \, \Gamma_1,
					\end{aligned}\right.
			\end{equation}}
			and
			\begin{equation}\label{psiasymp}
				\psi(x_1, x_2)\to \left\{
				\begin{aligned}
					&\bar\varphi_{b-a}(x_2-a)\quad \ && \text{as}\,\, x_1\to \infty,\\
					&\bar\varphi_1(x_2) \quad \  && \text{as}\,\, x_1\to -\infty.	
				\end{aligned}
				\right.
			\end{equation}
			Furthermore, if one denotes
   \begin{equation}
       \tilde{h}_0(x_1)=\sup\{x_2: \psi(x_1, x_2)=0, h_0(x_1)\leq x_2\leq h_1(x_1)\}
       \end{equation}
       {and}
       \begin{equation}
       \tilde{h}_1(x_1)=\inf\{x_2: \psi(x_1, x_2)=Q, h_0(x_1)\leq x_2\leq h_1(x_1)\},
   \end{equation}
then  $\tilde{h}_i\in C^1(\mathbb{R})$ ($i=0$, $1$) and $\tilde{h}_i\in C^{1,\alpha}$ in  $\{x_1: \tilde{h}_i(x_1)\neq h_i(x_1)\}$. 
Finally, $\psi$ satisfies
			\begin{equation}
				\partial_{x_2}\psi(x_1, x_2)>0\,\, \ \text{in}\,\, \Omega \cap \{(x_1,x_2): \tilde{h}_0(x_1)<x_2< \tilde{h}_1(x_1)\}.
			\end{equation}
		\end{lemma}
		\begin{proof}
			The proof is divided into three steps.
			
			{\it Step 1. Convergence and monotonicity.} Given any $k>0$, one has
			\[
			\|\psi_N\|_{C^{1,\alpha}(\Omega_k)}\leq C(k) \ \text{  for some }\alpha\in(0,1),
			\]
			where $C(k)$ is a constant depending on $k$, but independent of $N$. Therefore, by a diagonal procedure, there exists a subsequence which is still labeled by $\{\psi_N\}$ such that
			\[
			\psi_N \to \psi \quad \text{in}\,\, C^{1, \beta}(\Omega_k)\,\, \text{for any}\,\, k>0, \ \beta\in(0,\alpha).
			\]
			It follows from Lemma \ref{mon} that $\partial_{x_2}\psi\geq 0$ in $\Omega$.
   Note that $\|\tilde{h}_{i,N}\|_{C^1(-N+1, N-1)}$ is independent of $N$. Hence for any $k>0$, there exists a subsequence which is still relabelled by $\{\tilde{h}_{i,N}\}$ such that $\tilde{h}_{i,N}(x_1) \to \tilde{h}_i(x_1)$ for $x_1\in [-k, k]$. Denote
   \[
\tilde{\Omega}=\{(x_1,x_2): \tilde{h}_0(x_1)<x_2<\tilde{h}_1(x_1), x_1\in \mathbb{R}\}
\]
and
\[
 \tilde{\Gamma}_i=\{(x_1,x_2)|x_2=\tilde{h}_i(x_1), x_1\in \mathbb{R}\}, \,\, i=0, 1.
   \]
   Clearly $\psi(x)=0$ for $x\in \tilde{\Gamma}_0$ and $\psi(x)=Q$ for $x\in \tilde{\Gamma}_1$. Note that $\partial_{x_2}\psi(x)\geq 0$ and $0\leq \psi(x)\leq Q$ for all $x\in \Omega$. Hence one has one has 
   \[
   \psi(x)\equiv 0\,\,\text{in}\,\,  \Omega\cap 
   \{x_2<\tilde{h}_0(x_1)\} \,\, \text{and}\,\, \psi(x)\equiv Q\,\,\text{in}\,\,  \Omega\cap 
   \{x_2>\tilde{h}_1(x_1)\}
   \]
   Note that for any $x\in \tilde{\Omega}$,  if $N$ is sufficiently large, then one has $x\in \tilde{\Omega}_N$  
\[
\frac{1}{2}\text{dist}(x, \tilde{\Gamma}_{i,N}) \leq \text{dist}(x, \tilde{\Gamma}_i)\leq 2\text{dist}(x, \tilde{\Gamma}_{i,N}). 
\]   
 It follows from Lemma \ref{lemnew} that one has
 \begin{equation}
     \psi_N(x)\geq \frac{Q}{400}  \text{dist}(x, \tilde{\Gamma}_0)^2>0.
 \end{equation}
 Similarly, one has $\psi(x)<Q$ for all $x\in 
 \tilde{\Omega}$. 
			Furthermore, one has
			\begin{equation}\label{equ1}
				\Delta\psi =f(\psi)\quad \text{in}\,\, \ \  \tilde\Omega,
			\end{equation}
			and $\partial_{x_2}\psi(x_1,x_2)\geq  0$ in $\tilde\Omega$. It follows from Lemma \ref{lemmapositive} that $\partial_{x_2}\psi>0$ in $\tilde{\Omega}$. ${\psi}(x_1,x_2)=0$ in $\Omega \cap \{(x_1,x_2): {h}_0(x_1)<x_2\leq \tilde{h}_0(x_1)\}$, and $\psi(x_1,x_2)=Q$ in $\Omega\cap  \{(x_1,x_2): \tilde{h}_1(x_1)\leq x_2< h_1(x_1)\}$.

			
			{\it Step 2. Boundedness of normalized energy.}  Define
			\[
			\Upsilon(N):=	\int_{\Omega_{N}}\frac{|\nabla\psi|^2}{2}+F(\psi)dx  -N\bar{\mathcal{J}}_1-N\bar{\mathcal{J}}_{b-a}.
			\]
			Since $\bar{\varphi}_{b-a}(x_2-a)$ and $\bar{\varphi}_1$ are energy minimizers of ${\mathcal{J}}_{b-a}$ and ${\mathcal{J}}_1$ over ${\mathcal{U}}_{b-a}$ and ${\mathcal{U}}_1$, respectively, one has
			\begin{equation}
				\begin{aligned}
					\Upsilon'(N)=&\int_a^b \left[\frac{|\nabla\psi|^2}{2}+F(\psi)\right](N,x_2)dx_2-\bar{\mathcal{J}}_{b-a}\\
					&+ \int_0^1 \left[ \frac{|\nabla\psi|^2}{2}+F(\psi)\right](-N,x_2)dx_2-\bar{\mathcal{J}}_1\geq 0.
				\end{aligned}
			\end{equation}
			Therefore, $\Upsilon$ is an increasing function as long as $N\geq L_0$. For any given constants $l,m>L_0$, denote\[\Omega_{-l, m}=\{(x_1,x_2): -l\leq x_1\leq m, x_2\in(h_0(x_1), h_1(x_1))\}.\]
			Similarly, for any $n>\text{max}\{l,m\}$, one can prove 
			\[
			\Upsilon_{n}(l, m):	=\int_{\Omega_{-l, m}}\frac{|\nabla\psi_n|^2}{2}+F(\psi_n)dx  -l\bar{\mathcal{J}}_1-m\bar{\mathcal{J}}_{b-a}
			\]
			is increasing with respect to both $l$ and $m$, and hence
			\[
			\Upsilon_n(l, m)\leq \Upsilon_n(n, n)=\zeta(n),
			\]
   where the function $\zeta$ is defined in \eqref{defzeta}.
			Since $\psi_n\to \psi$ in $C^{1, \beta}_{loc}(\Omega)$, for any $l$, $m<n$, it holds that
			\[
			\Upsilon(l, m)=\lim_{n\to \infty} \Upsilon_n(l, m)\leq \liminf_{n\to \infty} \Upsilon_n(n, n)=\liminf_{n\to \infty} \zeta(n).
			\]
   Note that $\zeta(n)$ is a decreasing function with a lower bound. Hence there exists a constant $C$ such that
			\begin{equation}\label{boundedenergy}
				\begin{aligned}
					-C\leq 	\int_{\Omega_{-l, {m}}}\frac{|\nabla\psi|^2}{2}+F(\psi)dx  -l \bar{\mathcal{J}}_1-m\bar{\mathcal{J}}_{b-a}\leq C\quad \text{for any}\, \, l, m>0.
				\end{aligned}
			\end{equation}
						
			{\it Step 3. Asymptotic behavior.}  Suppose that  \eqref{psiasymp} is not true. Without loss of generality, assume that $\psi$ does not converge to $\bar{\varphi}_1$ in the upstream. Hence there exists an $\epsilon_0>0$ such that there exists a sequence $\{x_1^{(n)}\}$ satisfying
			$x_1^{(n+1)}<x_1^{(n)}-1$ and
			\[
			|\psi(x_1^{(n)},x_2^{(n)})-\bar\varphi_1(x_2^{(n)})|\geq \epsilon_0.
			\]
			Since $\psi$ is uniformly bounded in $C^{1,\alpha}(\Omega)$, one can deduce that there exists a $\delta>0$ such that 			\begin{equation}\label{ee1}
			|\psi(x_1,x_2)-\bar\varphi_1(x_2)|\geq \epsilon_0/2\quad \text{for }\,\, (x_1,x_2)\in \Omega\cap[(x_1^{(n)}-\delta, x_1^{(n)}+\delta)\times (x_2^{(n)}-\delta, x_2^{(n)}+\delta)].
			\end{equation}
It follows from Lemma \ref{lem3.9} that there exists a $\sigma>0$ such that for sufficiently large $n$, one has
			\begin{equation*}
				{\mathcal{J}}_1(\psi(x_1, \cdot)) \geq {\mathcal{J}}_1(\bar{\varphi}_1)+\sigma\quad \text{for}\,\, x_1\in [x_1^{(n)}-\delta, x_1^{(n)}+\delta].
			\end{equation*}	and hence
			\begin{eqnarray}\label{geq}
		\int_{x_1^{(n)}-\delta}^{x_1^{(n)}+\delta}\int_{0}^ 1\frac{|\nabla\psi|^2}{2}+F(\psi)dx \geq 2\delta\int_{0}^1\frac{|\bar\varphi'_{1}|^2}{2}+F(\bar\varphi_{1})dx_2 +\sigma.
			\end{eqnarray}
   On the other hand, one has
	\begin{equation}
				\begin{aligned}
					& \int_{\Omega_{N}}\frac{|\nabla \psi|^2}{2}+F(\psi)dx -N\bar{\mathcal{J}}_1-N\bar{\mathcal{J}}_{b-a}\\
					=& \int_{\Omega_{-N_0,N}}\frac{|\nabla \psi|^2}{2}+F(\psi)dx -N_0\bar{\mathcal{J}}_1-N\bar{\mathcal{J}}_{b-a}\\
					&+\sum_{i=1}^n \left( \int_{x_1^{(i)}-\delta}^{x_1^{(i)}+\delta}\int_{0}^ 1\frac{|\nabla\psi|^2}{2}+F(\psi)dx - 2\delta\int_{0}^ 1\frac{|\bar\varphi'_{1}|^2}{2}+F(\bar\varphi_{1})dx_2\right)\\
					& +\int_{\Omega_{N}\setminus  (\Omega_{-N_0,N}\cup (\cup_{i=1}^n (x_1^{(i)}-\delta, x_1^{(i)}+\delta)\times (0,1)))}\frac{|\nabla\psi|^2}{2}+F(\psi) - \left[\frac{|\bar\varphi'_{1}|^2}{2}+F(\bar\varphi_{1})\right]dx.
				\end{aligned}
			\end{equation}
			
It follows from Lemma \ref{lem3.9} and the estimate \eqref{boundedenergy} that
\begin{eqnarray*}
\sum_{i=1}^n \left( \int_{x_1^{(i)}-\delta}^{x_1^{(i)}+\delta}\int_{0}^ 1\frac{|\nabla\psi|^2}{2}+F(\psi)dx - 2\delta\int_{0}^ 1\frac{|\bar\varphi'_{1}|^2}{2}+F(\bar\varphi_{1})dx_2\right)\leq 2C,
\end{eqnarray*}
		which contradicts \eqref{geq} as $N\to \infty$.
This finishes the proof of the lemma.
		\end{proof}
		

		Combining the lemmas and propositions in this section gives the existence of solutions in Theorem \ref{thm2}.


		\medskip
		
		{\bf Acknowledgement.}
		The research of Li was partially supported by NSFC grants 12031012 and 11831003. The research of Shahgholian was supported in part by Swedish Research Council (grant no. ~2021-03700).
		The research of  Xie was partially supported by  NSFC grants 11971307, 1221101620, and  12161141004, the Fundamental Research Funds for the Central Universities, Natural Science Foundation of Shanghai 21ZR1433300, and Program of Shanghai Academic Research Leader 22XD1421400.

\section*{Declarations}

\noindent {\bf  Data availability statement:} All data needed are contained in the manuscript.

\medskip
\noindent {\bf  Funding and/or Conflicts of interests/Competing interests:} The authors declare that there are no financial, competing or conflict of interests.

		\bibliographystyle{plain}

	\end{document}